\documentclass{article}

\usepackage{alltt}
\usepackage{amsfonts}
\usepackage{amsmath}
\usepackage{amssymb}
\usepackage{amsthm}
\usepackage{booktabs}
\usepackage{caption}
\usepackage{subcaption}
\usepackage{enumitem}
\usepackage{fancyhdr}
\usepackage{fullpage}
\usepackage{graphicx}
\usepackage{mathdots}
\usepackage{mathtools}
\usepackage{microtype}
\usepackage{multirow}
\usepackage{pdflscape}
\usepackage{pgfplots}
\usepackage{siunitx}
\usepackage{slashed}
\usepackage{tabularx}
\usepackage{tikz}
\usepackage{tkz-euclide}
\usepackage[normalem]{ulem}
\usepackage[all]{xy}
\usepackage{imakeidx}
\usepackage{verbatim}
\usepackage[
backend=bibtex,
style=numeric,
sorting = nyt
]{biblatex}
\usepackage{url}
\usepackage{breakurl}
\usepackage[breaklinks]{hyperref}

\usetikzlibrary{arrows}
\usetikzlibrary{arrows.meta}
\usetikzlibrary{decorations.markings}
\usetikzlibrary{decorations.pathmorphing}
\usetikzlibrary{positioning}
\usetikzlibrary{fadings}
\usetikzlibrary{intersections}
\usetikzlibrary{cd}

\pgfarrowsdeclarecombine{twolatex'}{twolatex'}{latex'}{latex'}{latex'}{latex'}
\tikzset{->/.style = {decoration={markings,
			mark=at position 1 with {\arrow[scale=2]{latex'}}},
		postaction={decorate}}}
\tikzset{<-/.style = {decoration={markings,
			mark=at position 0 with {\arrowreversed[scale=2]{latex'}}},
		postaction={decorate}}}
\tikzset{<->/.style = {decoration={markings,
			mark=at position 0 with {\arrowreversed[scale=2]{latex'}},
			mark=at position 1 with {\arrow[scale=2]{latex'}}},
		postaction={decorate}}}
\tikzset{->-/.style = {decoration={markings,
			mark=at position #1 with {\arrow[scale=2]{latex'}}},
		postaction={decorate}}}
\tikzset{-<-/.style = {decoration={markings,
			mark=at position #1 with {\arrowreversed[scale=2]{latex'}}},
		postaction={decorate}}}
\tikzset{->>/.style = {decoration={markings,
			mark=at position 1 with {\arrow[scale=2]{latex'}}},
		postaction={decorate}}}
\tikzset{<<-/.style = {decoration={markings,
			mark=at position 0 with {\arrowreversed[scale=2]{twolatex'}}},
		postaction={decorate}}}
\tikzset{<<->>/.style = {decoration={markings,
			mark=at position 0 with {\arrowreversed[scale=2]{twolatex'}},
			mark=at position 1 with {\arrow[scale=2]{twolatex'}}},
		postaction={decorate}}}
\tikzset{->>-/.style = {decoration={markings,
			mark=at position #1 with {\arrow[scale=2]{twolatex'}}},
		postaction={decorate}}}
\tikzset{-<<-/.style = {decoration={markings,
			mark=at position #1 with {\arrowreversed[scale=2]{twolatex'}}},
		postaction={decorate}}}

\tikzset{circ/.style = {fill, circle, inner sep = 0, minimum size = 3}}
\tikzset{scirc/.style = {fill, circle, inner sep = 0, minimum size = 1.5}}
\tikzset{mstate/.style={circle, draw, blue, text=black, minimum width=0.7cm}}

\tikzset{eqpic/.style={baseline={([yshift=-.5ex]current bounding box.center)}}}
\tikzset{commutative diagrams/.cd,cdmap/.style={/tikz/column 1/.append style={anchor=base east},/tikz/column 2/.append style={anchor=base west},row sep=tiny}}

\theoremstyle{definition}

\newtheorem*{question}{Question}

\newtheorem{nthm}{Theorem}[section]

\newtheorem{cor}[nthm]{Corollary}
\newtheorem{conjecture}[nthm]{Conjecture}
\newtheorem{nlemma}[nthm]{Lemma}
\newtheorem{nprop}[nthm]{Proposition}
\newtheorem{assumption}[nthm]{Assumption}
\newtheorem{assumptions}[nthm]{Assumptions}
\newtheorem{ncor}[nthm]{Corollary}
\newtheorem{ncon}[nthm]{Conjecture}
\newtheorem{defi}[nthm]{Definition}
\newtheorem{eg}[nthm]{Example}

\newtheorem{remark}[nthm]{Remark}

\newcommand{\al}{\alpha}
\newcommand{\ga}{\gamma}
\newcommand{\CO}{\mathcal{CO}}
\newcommand{\OC}{\mathcal{OC}}
\newcommand{\MC}{\mathcal{C}}
\newcommand{\MM}{\mathcal{M}}
\newcommand{\qb}{\mathfrak{q}^{b,\ga}}
\newcommand{\mb}{\mathfrak{m}^{b,\ga}}
\newcommand{\m}{\mathfrak{m}}
\newcommand{\mte}{\mathfrak{m}^{t,e}}
\newcommand{\q}{\mathfrak{q}}
\newcommand{\CC}{\mathbb{C}}

\newcommand{\RR}{\mathbb{R}}
\newcommand{\ZZ}{\mathbb{Z}}
\newcommand{\pp}{\widetilde{p}}
\newcommand{\SX}{\mathcal{S}X}
\newcommand{\LX}{\mathcal{L}X}
\newcommand{\E}{\mathcal{E}}
\newcommand{\HH}{\mathcal{H}}
\newcommand{\WE}{\widetilde{\E}}
\newcommand{\WHH}{\widetilde{\HH}}

\newcommand{\K}{\mathbb{K}}
\newcommand{\DS}{QH^*(\Sigma_2)[[u]]}
\newcommand{\DX}{QH^*(X)[[u]]}
\newcommand{\YY}{Y}
\newcommand{\wv}{\widetilde{v}}
\let\stdsection\section
\renewcommand\section{\newpage\stdsection}

\def\st{\bgroup \ULdepth=-.55ex \ULset}

\begin{document}
	
	\title{The cyclic open-closed map, u-connections and R-matrices}
	
	\author{Kai Hugtenburg}
	
	\maketitle
	\begin{abstract}
		This paper considers the (negative) cyclic open-closed map $\OC^{-}$, which maps the cyclic homology of the Fukaya category of a symplectic manifold to its $S^1$-equivariant quantum cohomology. We prove (under simplifying technical hypotheses) that this map respects the respective natural connections in the direction of the equivariant parameter. In the monotone setting this allows us to conclude that $\OC^{-}$ intertwines the decomposition of the Fukaya category by eigenvalues of quantum cup product with the first Chern class, with the Hukuhara-Levelt-Turrittin decomposition of the quantum cohomology. We also explain how our results relate to the Givental-Teleman classification of semisimple cohomological field theories: in particular, how the R-matrix is related to $\OC^{-}$ in the semisimple case; we also consider the non-semisimple case.
	\end{abstract}
	\tableofcontents
	\section{Introduction}
	Kontsevich conjectured \cite{Kon95} that enumerative mirror symmetry, an equality between Gromov-Witten invariants of a space $X$ and period integrals on $Y$ (see \cite{Can}) is a consequence of a homological mirror symmetry: \begin{equation}
		Fuk(X)  \cong D^bCoh(Y).
	\end{equation}
	This paper focusses on the symplectic side of mirror symmetry. Barannikov \cite{Bar} shows that one can extract the Gromov-Witten invariants of $X$ from a variation of semi-infinite Hodge structures (VSHS) associated to the quantum cohomology of $X$, together with a splitting of the Hodge filtration. This goes via the intermediary step of a Frobenius manifold. One approach to obtain enumerative invariants from the Fukaya category is thus to first associate a VSHS to it, and then to specify the correct splitting. It is by now well understood how to construct the structure of a VSHS on the cyclic homology of an $A_\infty$-category (see \cite{Ge}, \cite{KKP},  or \cite{She}). Characterising the splitting has not been done in general, but results have been obtained in various settings.
	
	Ganatra-Perutz-Sheridan \cite{GPS} characterise the splitting when the VSHS is $\mathbb{Z}$-graded and of Hodge-Tate type over a one-dimensional base. The geometric setting one should think of is the Fukaya category of a Calabi-Yau. In this case the splitting is determined by the VSHS itself. Secondly Amorim-Tu \cite{AT} show how the grading operator on quantum cohomology classifies the correct splitting when the Hochschild cohomology ring of the Fukaya category is semi-simple. The grading operator constitutes extra data, so the splitting is not necessarily determined intrinsically by the VSHS. The main examples are all Fano: complex projective space, or quadric hypersurfaces. 
	
	\subsection{Formal TEP-structures}
	Hertling \cite[Section 2.5]{Her} defines TERP-structures. We will only need TEP-structures. Furthermore, rather than working with holomorphic functions, we work with formal power series in the equivariant parameter. Hence we call them formal TEP-structures.
	\begin{defi}[see Definition \ref{formal TEP structures defi}]\text{ }
		Let $\mathbb{K}$ be a field.
		\begin{enumerate}
			\item A formal pre-T-structure over a $\mathbb{K}$-algebra $R$, is a pair $(\mathcal{E}, \nabla)$. Here $\mathcal{E}$ is an $R[[u]]$-module and $\nabla: Der_{\mathbb{K}}R \otimes \mathcal{E} \rightarrow u^{-1} \mathcal{E}$ a flat connection.
			\item If $\mathcal{E}$ is free and finitely-generated, call this a formal T-structure.
			\item A formal TE-structure is a formal T-structure together with an extension of the connection to a flat connection $\nabla: Der_{\mathbb{K}}(R[[u]]) \otimes \mathcal{E} \rightarrow u^{-2} \mathcal{E}$.
			\item A formal TEP-structure is a formal TE-structure equipped with polarisation, i.e.\ a symmetric, sesquilinear, covariantly constant pairing $( \cdot, \cdot ): \mathcal{E} \otimes \mathcal{E} \rightarrow R[[u]]$, which restricts to a non-degenerate pairing $( \cdot, \cdot ): \mathcal{E}/u\mathcal{E} \otimes \mathcal{E}/u\mathcal{E} \rightarrow R$.
		\end{enumerate}
	\end{defi}
	Thus a VSHS in the sense of \cite{Bar} is a formal TP-structure.
	\begin{remark}
		A TEP-structure can be formalised to yield a formal TEP-structure, this process forgets information (the Stokes' data, see \cite[\S II.6]{Sab}). The cyclic homology of an $A_\infty$-category only yields a formal TEP-structure, which is why we will always be talking about formal TEP-structures. For ease of reading, we omit the word `formal' from now on. We hope this doesn't cause any confusion.
	\end{remark}
	\begin{defi}
		The quantum TEP-structure is defined over $R = \Lambda[[H^*(X)]]$, where $\Lambda$ is a Novikov ring. It is given by the $S^1$-equivariant quantum cohomology $QH^*(X;R)[[u]]$. The connection is as defined in \cite{Dub}, or see Section \ref{Quantum u-VSHS}. The pairing is given by the sesquilinear extension of the Poincar\'e pairing.
	\end{defi}
	\begin{defi}
		The TEP-structure $HC_*^-(\mathcal{C})$ associated to an $R$-linear $A_\infty$-category is as defined in \cite{KKP} or see Section \ref{Cyclic homology}.
	\end{defi}
	\begin{remark}
		In general, the TEP-structure associated to an $A_\infty$-category is only a pre-TEP-structure. If the non-commutative Hodge-de Rham spectral sequence degenerates, it is actually a TEP-structure. This is conjectured to hold for smooth and compact $A_\infty$-categories, see \cite{KS1}. In the $\ZZ$-graded setting, Kaledin \cite{Ka} proves this conjectures holds. Our $A_\infty$-categories will only be $\ZZ/2$-graded, and we will always assume the Hodge-de Rham spectral sequence degenerates, and can thus drop the prefix `pre'.
	\end{remark}
	\subsection{Cyclic open-closed map}
	An essential ingredient in proving that the enumerative invariants obtained from the Fukaya category agree with the Gromov-Witten invariants is an isomorphism of TEP-structures called the (negative) cyclic open-closed map. Let $X$ be a closed symplectic manifold. Let $Fuk(X)$ denote the Fukaya category of $X$, which is an $A_\infty$-category over the Novikov ring $\Lambda$. Assume there exists a bulk-deformed Fukaya category $Fuk^t(X)$. By this we mean a Fukaya category which is linear over the ring $R = \Lambda[[H^*(X)]]$.
	\begin{ncon}
		\label{VSHS conjecture}
		There exists a cyclic open-map $\mathcal{OC}^{-}: HC^-_*(Fuk^t(X)) \rightarrow QH^*(X;R)[[u]]$. This is a morphism of TEP-structures over $\Lambda[[H^*(X)]]$.
	\end{ncon}
	Such a morphism has not been constructed in general. Partial results exist: \cite{FOOO} and \cite{Ga19} construct cyclic open-closed maps in a wide range of settings. Ganatra-Perutz-Sheridan \cite{GPS} have announced work proving this is an isomorphism of TP-structures when $X$ is a projective Calabi-Yau manifold. In their case, $R = \Lambda$, so they consider no bulk-deformations. Furthermore, Ohta-Sanda \cite{OS} show that both TE-structures considered come from a new algebraic structure they define, a `CH-structure'. They show that an isomorphism of the CH-structures associated to the Fukaya category and quantum cohomology would imply an isomorphism of associated TE-structures.
	
	We prove a local version of this conjecture, focussing on the cyclic homology of the $A_\infty$-algebra associated to a single Lagrangian. We use the same technical assumptions as used by Solomon and Tukachinsky \cite{ST3}: the moduli spaces of holomorphic disks need to be smooth orbifolds with corners, and the evaluation maps at boundary points are assumed to be submersions, see Assumptions \ref{assumptions}. For us the main example satisfying these conditions is $X = \mathbb{CP}^n$ and $L$ a Lagrangian torus fibre (see Lemma \ref{G action assumptions hold} for a proof of the assumptions). Another class of examples is given by flag varieties and their products, with the Lagrangian given by the real locus (see \cite[Example~1.5]{ST3}). The $A_\infty$-algebra $CF^*(L,L)$ we use is equal up to sign to the $A_\infty$-algebra defined by \cite{ST3}, see Remark \ref{comparison of algebras}.

	\begin{nthm}[see Theorem \ref{cyclic open-closed morphism of u-VSHS}]
		\label{cyclic open closed theorem in intro}
		Let $L\subset X$ be an oriented, relatively-spin Lagrangian submanifold equipped with a $U(\Lambda)$-local system. Suppose there exists a complex structure $J$ such that $(L,J)$ satisfy Assumptions \ref{assumptions}. Then there exists a bulk-deformed Fukaya $A_\infty$-algebra $CF^*(L,L)$. This is an $R=\Lambda[[H^*(X)]]$-linear, curved and filtered $A_\infty$-algebra. Furthermore, there exist a cyclic open-closed map \begin{equation}
			\mathcal{OC}^{-}: HC^-_*(CF^*(L,L)) \rightarrow QH^*(X;R)[[u]],
		\end{equation}
		which is a morphism of pre-TE-structures over $R$.
	\end{nthm}
	The argument we use to show that the cyclic open-closed map is a morphism of T-structures is due to Ganatra-Perutz-Sheridan \cite{GPS2}, as announced in \cite{GPS}. The argument simplifies in our setting as the $A_\infty$-category we use is cyclic and strictly unital. This ensures that our construction of the cyclic open-closed map does not require higher order terms in $u$, as opposed to \cite{Ga19}. This comes at the cost of working over a field containing $\RR$. To shows that the cyclic open-closed map respects the connection in the $u$-directions an extra ingredient is needed, which is that the cyclic open-closed map respects (Euler-)gradings.
	
	As each component of the Fukaya category $Fuk(\mathbb{CP}^n)$ is generated by the Clifford torus (but with different local systems), we can thus construct a bulk-deformed Fukaya category $Fuk(\mathbb{CP}^n;R)$ over $R = \mathbb{C}[[H^*(\mathbb{CP}^n)]]$. We can thus define a global cyclic open-closed map using our setup:
	\begin{cor}
		There exists a cyclic open-closed map \begin{equation}
			\OC^-: HC^-_*(Fuk(\mathbb{CP}^n;R)) \rightarrow QH^*(\mathbb{CP}^n;R)[[u]],
		\end{equation}
		which is an isomorphism of TE-structures over $R$.
	\end{cor}
	
	\begin{remark}
		One reason we adopt the rather restrictive technical assumptions of \cite{ST3} is that we plan follow-up work in which we relate the results of this paper, which concern closed Gromov-Witten invariants, with the open Gromov-Witten invariants defined in \cite{ST2}. Similar to \cite[Remark~4.2]{ST3} we expect that these restrictive technical assumptions can be removed, as their role is purely to simplify the analysis of moduli spaces of holomorphic disks.
	\end{remark}
	\subsection{Image of the cyclic open-closed map for monotone symplectic manifolds}
	\label{monotone setting}
	For the remainder of the introduction, let $X$ be a monotone symplectic manifold. It is then possible to define the Fukaya category and quantum cohomology over $\mathbb{C}$ (rather than over a Novikov ring). For ease of exposition in this introduction, we set all bulk-parameters equal to zero and the Novikov parameter to $1$, so that $R = \K = \mathbb{C}$. Because we then only have a connection in the $u$-direction, we call $\DX$ an \emph{E-structure} (see Section \ref{E-structures}).

	By definition (see \cite{She16} for example), $Fuk(X;\mathbb{C}) = \bigoplus_w Fuk(X)_w$. Here $Fuk(X)_w$ is a $\mathbb{C}$-linear $A_\infty$-category with objects monotone Lagrangians with disk potential $w \in \mathbb{C}$. We consider $Fuk(X)_w$ as a weakly curved $A_\infty$-algebra with curvature $w \cdot 1$. For a monotone symplectic manifold, quantum cohomology can also be defined over $\mathbb{C}$ (see \cite{MS12}). As a vector-space we have $QH^*(X;\mathbb{C}) = H^*(X;\mathbb{C})$. The first Chern class defines a map: \begin{equation}
		c_1 \star : QH^*(X;\mathbb{C}) \rightarrow QH^*(X;\mathbb{C}).
	\end{equation} Decompose quantum cohomology into generalised eigenspaces for this map: \begin{equation}
		QH^*(X;\mathbb{C}) = \bigoplus_w QH^*(X)_w.
	\end{equation} 
	The following was first proved by Ritter-Smith:
	\begin{nthm}[{\cite[Theorem~9.5]{RS}}]
		The open-closed map satisfies $\mathcal{OC}(HH_*(Fuk(X)_w)) \subset QH^*(X)_w$.
	\end{nthm}
	
	A natural question to ask is how this result extends to cyclic homology: \begin{question}
		What is the image of the cyclic open closed map $\OC^{-}(HC^-_*(Fuk(X)_w)) \subset \DX$?
	\end{question}
	One might naively think we would have $\OC^{-}(HC^-_*(Fuk(X)_w)) \subset QH^*(X)_w[[u]]$, but the latter is not necessarily invariant under the connection in the $u$-direction, so this is incompatible with Conjecture \ref{VSHS conjecture}. Instead, taking inspiration from \cite[Section~2.28]{KKP}, we apply the Hukuhara-Levelt-Turrittin theorem (see \cite{Huk},\cite{Lev} and \cite{Tur}) to decompose the quantum E-structure as a direct sum of $\nabla_{\frac{d}{du}}$-invariant submodules indexed by the eigenvalues of $c_1 \star$: 
	\begin{nlemma}[Hukuhara-Levelt-Turrittin]
		\label{decomposition of quantum TE structure}
		There exists a unique decomposition of $\DX$ into $\mathbb{C}[[u]]$ submodules:
		\begin{equation}
			\label{quantum D-module decomposition}
			\DX = \bigoplus_w \DX_w,
		\end{equation}
		such that each $\DX_w$ is invariant under $u^2\nabla_{\frac{d}{du}}$, and $\DX_{w}/u\DX_w = QH^*(X)_w$.
	\end{nlemma}
	Conjecture \ref{VSHS conjecture}, along with a slight extension of the results in \cite{Lev}, then shows that the cyclic open-closed map respects this decomposition:
	\begin{ncor}[see \ref{cyclic OC respects decompositions}]
		\label{image of cyclic open-closed map}
		$\OC^{-}(HC^-_*(Fuk(X)_w)) \subset \DX_w$.
	\end{ncor}
	
	Since we don't actually prove the full conjecture, the corollary that follows from our Theorem \ref{cyclic open closed theorem in intro} is:
	\begin{ncor}
		Let $L \subset X$ be a monotone Lagrangian with disk potential $w$. Suppose $(X,L)$ satisfies Assumptions \ref{assumptions}, then:
		\begin{equation}
			\OC^{-}(HC^-(CF^*(L,L))) \subset \DX_w.
		\end{equation}
	\end{ncor} 

	\begin{remark}
	The Hukuhara-Levelt-Turrittin decomposition has appeared before in the study of mirror symmetry. It was used first in \cite{HS} and later in \cite{KKP} to introduce the notion of a Hodge structure of exponential type.
	\end{remark}

	\subsection{Semi-simple quantum cohomology}
	If we additionally assume that $QH^*(X;\CC)$ is a semi-simple $\CC$-algebra (isomorphic as a ring to a direct sum of copies of $\CC$), we can completely determine the E-structure $\DX$. To this end, for $\phi \in \mathbb{C}[u^{-1}]$, let $\mathcal{E}^{\phi} := (\mathbb{C}[[u]], \nabla_{\frac{d}{du}})$ denote the 1-dimensional TE-structure (over $R = \CC$), with connection given by $\nabla_{\frac{d}{du}} = \frac{d}{du} + \frac{d\phi}{du}$. 
	We show the following, which was already obtained by \cite{Dub}, see also \cite{Te} and \cite{GGI}:
	\begin{nlemma}[see Corollary \ref{w flat basis qcoh}]
		\label{ss quantum coh lemma}
		Assume $QH^*(X)$ is semi-simple, then there exists a basis $v_i \in \DX$ such that $u^2\nabla_{\frac{d}{du}}v_i = w_iv_i$, where the $w_i$ are the eigenvalues of $c_1 \star$. We call the $v_i$ `\emph{$w_i$-flat sections}'. Equivalently, there is an isomorphism of E-structures \begin{equation}
			QH^*(X)[[u]]_w \cong \mathcal{E}^{-\frac{w}{u}}.
		\end{equation}
	\end{nlemma}
	\begin{remark}
		The semi-simplicity assumption is essential; diagonalisability of $c_1$ is insufficient. This is because we need a special property of the grading operator $\mu$ on quantum cohomology (see Lemma \ref{mu property}).
	\end{remark}
	\begin{defi}
		Given an E-structure $(\mathcal{E},\nabla)$ a \emph{splitting} is a $\K$-linear map $s: \mathcal{E}/u\mathcal{E} \rightarrow \mathcal{E}$ splitting the natural projection $\pi: \mathcal{E} \rightarrow \mathcal{E}/u\mathcal{E}$.
	\end{defi}
\begin{eg}
	\label{canonical splitting E}
	The E-structure $\mathcal{E}^{\phi}$ admits a splitting given by:
	\begin{align}
		s: \mathbb{C} = \mathcal{E}^\phi/u\mathcal{E}^\phi &\rightarrow \mathcal{E}^\phi = \CC[[u]]\nonumber\\
		\alpha &\mapsto \alpha.
	\end{align}
\end{eg}
\begin{remark}
	A choice of splitting is equivalent to a choice of opposite subspace as used by Barannikov \cite{Bar} to obtain a Frobenius manifold from a VSHS. See also \cite[Section 2.1.7]{Gr}.
	\end{remark}
	
	The quantum E-structure admits a canonical splitting. This splitting does not respect the decomposition of Lemma \ref{decomposition of quantum TE structure}, but it is the one relevant for Gromov-Witten theory:
	\begin{align}
		s^{GW}: QH^*(X) &\rightarrow \DX\nonumber\\
		\alpha &\mapsto \alpha.
	\end{align}
	When the quantum cohomology is semi-simple, the $w_i$-flat sections define a second splitting $s^{ss}: QH^*(X) \rightarrow \DX$ given by: \begin{equation}
		v_i (\text{mod } u) \mapsto v_i.
	\end{equation} 
	Note that whilst the $v_i$ are not unique, the associated splitting is uniquely determined, as any two choices of the $v_i$ are related by a constant matrix. This splitting preserves the decomposition of the quantum TE-structure: \begin{equation}
		s^{ss}(QH^*(X)_w) \subset QH^*(X)[[u]]_w.
	\end{equation} 

	For a general E-structure, given two splittings $s_1, s_2$, we obtain an element $R \in Aut(\mathcal{E}/u\mathcal{E})[[u]]$ as $R = \sum_{i \geq 0} u^iR_i$, with $R_0 = Id$, and 
	\begin{equation}
		s_1(\alpha) = \sum_{i \geq 0} u^i s_2(R_i(\alpha)) \; \text{for all} \; \alpha \in \mathcal{E}/u\mathcal{E}.
	\end{equation}Such $R$ is called an \emph{R-matrix}.
\begin{remark}
	R-matrices were used by Givental \cite{Giv} and Teleman \cite{Te} to classify semi-simple TFT's. See also \cite[chapter~2]{PanPD} for the definition of R-matrices and their action on cohomological field theories. Their definition of an R-matrix involves an additional `symplectic' property, namely that $R$ preserves the polarisation. The group of such symplectic R-matrices is called the Givental loop group. We do not consider this polarisation, so our R-matrices need not be elements of the Givental loop group.
	\end{remark}

	The two splittings on the quantum E-structure are thus related by an R-matrix $R \in Aut(QH^*(X))[[u]]$. A short computation shows that this is indeed the same R-matrix as defined by Teleman \cite{Te} to recover all (including higher genus) Gromov-Witten invariants of $X$ from its genus 0, 3-point invariants.
	
	By Corollary \ref{image of cyclic open-closed map}, we find the following: \begin{equation}
		\OC^-(HC^-_*(Fuk(X)_w)) \subset R(QH^*(X)_w[[u]]).
	\end{equation} 
	The R-matrix thus tells us how to change the naive/constant decomposition of quantum cohomology to be compatible with the cyclic open-closed map. 
	
	Amorim and Tu show the categorical version of Lemma \ref{ss quantum coh lemma}:
	\begin{nlemma}[{\cite[Corollary~3.8]{AT}}]
		\label{AT theorem}
		Let $\mathcal{C} = \bigoplus_w \mathcal{C}_w$ be a direct sum of strictly unital, smooth, finite-dimensional, cyclic and weakly curved $A_\infty$-categories of curvature $w \cdot 1$. Assume $HH^*(\mathcal{C})$ is semi-simple. Then there exists a splitting: \begin{equation}
			s^{\mathcal{C}}:HC^-_*(\mathcal{C})/uHC^-_*(\mathcal{C}) = HH_*(\mathcal{C}) \rightarrow HC^-_*(\mathcal{C})
		\end{equation} 
		characterised by the equation $u^2\nabla_{\frac{d}{du}}s(\alpha) = ws(\alpha)$ for all $\alpha \in HH_*(\mathcal{C}_w)$.
	\end{nlemma}
	This lemma can be rephrased as the existence of an isomorphism of E-structures \begin{equation}
		HC_*^-(\mathcal{C}) \cong \bigoplus_w \mathcal{E}^{-\frac{w}{u}}.
	\end{equation}
	If $QH^*(X)$ is semi-simple, and the closed-open map is an isomorphism, then $HH^*(Fuk(X))$ is semi-simple. Thus the previous lemma is indeed what was expected from Conjecture \ref{VSHS conjecture} and Lemma \ref{ss quantum coh lemma}.
	
	In Section \ref{semi-simple QH} we explain how our Conjecture \ref{VSHS conjecture}, if proved in appropriate generality, can be used to give an alternative proof of the following theorem of Amorim-Tu.
	\begin{nthm}[{\cite[Theorem~1.3]{AT}}]
		\label{AT reconstruction theorem}
		Let $X$ be a symplectic manifold with $HH^*(Fuk(X))$ semi-simple. Then the category $Fuk(X)$ together with the closed-open map determine the big quantum cohomology as a Frobenius manifold.
	\end{nthm}
	Amorim and Tu prove their theorem under the assumption that $\CO$ is a ring isomorphism, and use the Dubrovin-Teleman reconstruction theorem (\cite{Dub},\cite{Te}) of semi-simple Frobenius manifolds. Our proof instead uses $\OC^-$ and assumes Conjecture \ref{VSHS conjecture}, which allows us to avoid appealing to the reconstruction theorem.
	
	\subsection{Speculations on the general case}
	When the quantum cohomology is not semi-simple, a basis of w-flat sections does not necessarily exist. However, sometimes it is still possible to construct a non-trivial R-matrix. Consider the case when the Fukaya category of $X$ splits as follows:
	\begin{equation}
		Fuk(X) \cong \bigoplus_i Fuk(Y_i),
	\end{equation}
	where the $Y_i$ are (not necessarily monotone) symplectic manifolds. This is expected to hold when $X$ is a blow up (see \cite{Ven} for a proof in certain cases). Another example is the complete intersection of two quadric hypersurfaces in $\mathbb{CP}^5$ (see \cite{Smi}). We conjecture:
	\begin{conjecture}
		When $Fuk(X)$ splits up as above, then the Gromov-Witten invariants of $X$ can be obtained from those of the $Y_i$, together with the genus 0, 3 point invariants of $X$.
	\end{conjecture}
	We will illustrate this conjecture when $X$ is the complete intersection of two quadric hypersurfaces in $\mathbb{CP}^5$. The eigenvalue decomposition of the Fukaya category is as follows: \begin{equation}
		Fuk(X) = Fuk(X)_{-8} \oplus Fuk(X)_0 \oplus Fuk(X)_8.
	\end{equation}
	Smith proves an equivalence:
	\begin{nthm}[{\cite[Theorem~1.1]{Smi}}]
		$D^\pi Fuk(X)_0 \cong D^\pi Fuk(\Sigma_2)$, for $\Sigma_2$ a genus 2 surface.
	\end{nthm}
	Assume that $Fuk(X)_{\pm 8} \cong Fuk(pt)$, which \cite[Section~1.6]{Smi} expects. And note that the $Fuk(pt)$ are considered here with curvature $\pm 8$, so that $HC^-_*(Fuk(X)_{\pm 8}) \cong \mathcal{E}^{\mp 8/u}$ (see Lemma \ref{curved vs uncurved}). Also note that \cite[Chapter~4]{She} proves a natural isomorphism $HC^{-}(D^\pi \mathcal{C}) \cong HC^{-}(\mathcal{C})$. We thus have an isomorphism of E-structures: \begin{equation}
		HC_*^{-}(Fuk(X)) \cong \mathcal{E}^{8/u} \oplus HC_*^{-}(Fuk(\Sigma_2)) \oplus \mathcal{E}^{-8/u}.
	\end{equation}
	The cyclic open-closed map then carries this isomorphism to an isomorphism of E-structures:\begin{align}
		\Phi: \DX &\cong \mathcal{E}^{8/u} \oplus \DS \oplus \mathcal{E}^{-8/u},\\
		x &\mapsto (\Phi_1(x),\Phi_2(x), \Phi_3(x)).
	\end{align}
	An explicit computation shows that $\Phi$ is unique up to rescaling the $\Phi_i$ by constants $\lambda_i \in \CC$. Thus, the following splitting is well-defined: \begin{equation}
		s_1 := \Phi^{-1} \circ (s \oplus s^{GW} \oplus s) \circ \pi \circ \Phi: QH^*(X) \rightarrow \DX.
	\end{equation} Here $s: \CC \rightarrow \mathcal{E}^{\pm8/u}$ is as defined in Example \ref{canonical splitting E}, $s^{GW}$ denotes the canonical splitting on $\DS$, and $\pi$ is the map given by setting $u = 0$. Let $s_2 = s^{GW}: QH^*(X) \rightarrow \DX$ be the canonical splitting. These splittings $s_1$ and $s_2$ are related by an R-matrix. In Appendix \ref{appendix example intersection of quadrics} we show how to compute this R-matrix. We conjecture:
	\begin{ncon}
		This R-matrix recovers all (including higher genus) Gromov-Witten invariants of $X$ from the genus 0, 3 point invariants of $X$ and the all-genus Gromov-Witten invariants of $\Sigma_2$.
	\end{ncon}

	\subsection{Outline of the paper}
	Section \ref{Formal TEP-structures} defines formal TEP-structures and related notations. In section \ref{semi-simple TEP structures} we define semi-simple TEP-structures and interpret results of \cite{AT} using this language. 
	Next in section \ref{Cyclic homology} we endow the cyclic homology of an $A_\infty$-algebra with a TE-structure. Section \ref{outline of general proof} outlines properties of the Fukaya category and the cyclic open-closed map which are sufficient to prove Conjecture \ref{VSHS conjecture} in a general setting. For a Fukaya category with a single Lagrangian we then construct a cyclic open-closed map satisfying these properties in section \ref{closed-open and open-closed maps}. This relies on a structure equation for horocyclic operations, which we prove in section \ref{proof of boundary of horocycle moduli}. In section \ref{Quantum cohomology} we study applications of Conjecture \ref{VSHS conjecture}. In particular we show how Lemma \ref{decomposition of quantum TE structure} and Lemma \ref{ss quantum coh lemma} follow from general considerations about TE-structures. We also explain an alternative proof of Theorem \ref{AT reconstruction theorem}. In Appendix \ref{Euler-grading on Fukaya category} we provide heuristics showing how a `standard' definition of a Fukaya category (with multiple Lagrangians) can be modified to define a $\ZZ$-graded category (but at the cost of enlarging the coefficient ring). We also outline why we expect the properties of section \ref{outline of general proof} (which are sufficient to prove that the cyclic open-closed map is a morphism of TE-structures) to hold for this Fukaya category. In Appendix \ref{appendix example intersection of quadrics} we show there exists a unique R-matrix for the intersection of quadrics in $\mathbb{CP}^5$.  Finally in Appendix \ref{orientation properties} we prove a result which was missing in the literature about the orientation properties of gluing of holomorphic maps.
	
	\subsection{Acknowledgements}
	This paper was mostly written during my PhD at the University of Edinburgh. I could not have wished for a better supervisor than Nick Sheridan. His explanations, suggestions and comments have been invaluable. I would also like to thank Sara Tukachinsky for explaining a wide variety of ideas from her series of joint papers with Jake Solomon. I am also grateful to both Sara Tukachinsky and Jake Solomon for sharing an unpublished draft chapter proving the structure equations for geodesic operators \cite[section~3.2]{ST2}. My chapter \ref{proof of boundary of horocycle moduli} is an adaptation of their proof to the case of horocyclic constraints. I was partly supported by the Royal Society Research Grant for Research Fellows: RGFnR1n181009ERC, the ERC Starting Grant 850713 – HMS and EPSRC Grant EP/W015749/1.

\section{Formal TEP-structures}
\label{Formal TEP-structures}
Let $\mathbb{K}$ be a field of characteristic 0. Let $R$ be a $\ZZ/2$-graded commutative $\K$-algebra.
\begin{defi}[Formal T(EP)-structure]\text{ }
	\label{formal TEP structures defi}
	\begin{enumerate}
		\item A formal pre-T-structure over $\mathcal{M} = Spec(R)$, is a pair $(\mathcal{E}, \nabla)$. Here $\mathcal{E}$ is a $\ZZ/2$-graded $R[[u]]$-module, with $u$ of even degree and $\nabla: Der_{\K}R \otimes \mathcal{E} \rightarrow u^{-1} \mathcal{E}$ a flat connection of even degree.
		\item A formal pre-TE-structure is a formal pre-T-structure together with an extension of the connection to a flat connection $\nabla: Der_{\K}(R[[u]]) \otimes \mathcal{E} \rightarrow u^{-2} \mathcal{E}$.
		\item A formal pre-TP-structure is a formal pre-T-structure equipped with a polarisation, i.e.\ a covariantly constant pairing $$( \cdot, \cdot )_{\E}: \mathcal{E} \otimes \mathcal{E} \rightarrow R[[u]],$$ which is $R$-linear, of even degree and $u$-sesquilinear. That is, for $f(u) \in R[[u]]$, we have: $$f(u)(a,b)_\E = (f(u)a,b)_\E = (-1)^{|f||a|}(a,f(-u)b)_\E.$$
		\item For a formal pre-TEP-structure, we require that the pairing is also covariantly constant with respect to $\nabla_{\partial_u}$. More precisely:$$ (\nabla_{u\partial_u}a,b)_\E + (a,\nabla_{u\partial_u}b)_\E = u\partial_u(a,b)_\E.$$
		\item If additionally $\mathcal{E}$ is free and finitely-generated, we drop the prefix `pre' from T- and TE-structures. Let $\WE = \E/u\E$. For a TP-structure, we additionally require that the restriction of the pairing $( \cdot, \cdot )_{\WE}: \WE \otimes \WE \rightarrow R$ is non-degenerate.
	\end{enumerate}
\end{defi}
As mentioned in the introduction, we will always be taking about formal T-structures, so we will forget about the `formal'. Additionally, we call a (pre-)TE(P)-structure with $R = \K$ a (pre-)E(P)-structure.
\begin{defi}
	A morphism of pre-T(EP)-structures is an $R[[u]]$-module map $F: \E_1 \rightarrow \E_2$ which respects connections and the pairing (if one exists). A morphism of T(EP)-structures is the same as a morphism of pre-T(EP)-structures.
\end{defi}

\begin{defi}
	\label{Euler-grading}
	Let $\E$ be a pre-T-structure over $spec(R)$. An \emph{Euler-grading} on $\E$ consists of an even degree $\K$-linear map: \begin{equation}
		Gr: \E \rightarrow \E,
	\end{equation}
	called the grading and a vector field $E \in Der_\K R$ of even degree, called the Euler vector field, such that for $f \in R$, $a \in \E$ and $X  \in Der_\K R$: \begin{align}
		Gr(fa) &= (2u\partial_u + 2E)(f)a + fGr(a),\\
		[Gr,\nabla_X] &= \nabla_{[2E,X]}.
	\end{align}
	If $\E$ is a pre-TP-structure, we additionally require that $(2u \partial_u + 2E)(a,b)_\E = (Gr(a),b)_\E + (a,Gr(b))_\E$.
\end{defi}
\begin{remark}
	 An Euler-grading differs from a more standard definition of graded in that $\E$ is not required to admit a direct sum decomposition into graded pieces.
\end{remark}
\begin{defi}
	For Euler-graded pre-T-structures $\E_1$, $\E_2$ over $R$ with grading operators $Gr_1$ and $Gr_2$, and Euler-vector field $E_1 = E_2$, a morphism of Euler-graded pre-T-structures is a morphism of pre-T-structures $F$ which additionally satisfies $F \circ Gr_1 = Gr_2 \circ F$.
	\end{defi}
\begin{defi}
	\label{TE-structure on Euler-graded T-structure}
	Given an Euler-graded pre-T-structure, we obtain an associated pre-TE-structure by setting: \begin{equation}\nabla_{\partial_u} := \frac{1}{2u}Gr - \frac{1}{u}\nabla_E.\end{equation}
\end{defi}
A short computation shows the total connection is flat, showing this is a valid definition. As a morphism of Euler-graded pre-T-structures respects the grading and the connection, we find:
\begin{nlemma}
	A morphism of Euler-graded pre-T-structures is a morphism of associated pre-TE-structures.
\end{nlemma}
\begin{defi}
	An Euler-grading on a pre-TE-structure is an Euler-grading on the underlying T-structure, such that $\nabla_{\partial_u} = \frac{1}{2u}Gr - \frac{1}{u}\nabla_E$.
\end{defi}

\subsection{E-structures}
\label{E-structures}
Throughout this subsection we take $\K = \CC$.
\begin{defi}[E(P)-structure]
	An E(P)-structure is a TE(P)-structure $\E$ over the $\K$-algebra $R = \K$, so that $\MM = pt$. We thus only have a connection $\nabla_{\frac{d}{du}}: \E \rightarrow u^{-2}\E$. For ease of notation we will often write $\nabla$ for $\nabla_{\frac{d}{du}}$ for an E(P)-structure.
\end{defi}
\begin{eg}
	Let $\mathcal{E}^{-w/u}$ be the 1-dimensional EP-structure $\mathcal{E} = \mathbb{C}[[u]]$ with connection $\nabla = \frac{d}{du} + \frac{w}{u^2}$ and pairing $(1,1) = 1$.
\end{eg}
\begin{defi}
	A \emph{splitting} of an E-structure is a $\K$-linear map $s: \WE \rightarrow \E$ which splits the canonical map $\pi: \E \rightarrow \WE$. If $\E$ is an EP-structure, we say $s$ is \emph{P-compatible} if $(s(a),s(b))_\E = (a,b)_{\WE}$ for all $a,b \in \WE$.
\end{defi}
As $\E$ is finitely generated and free, there always exists a splitting. A choice of splitting $s$ defines an isomorphism: \begin{align}
	\Phi_s: \WE[[u]] &\cong \E\notag\\
	\sum_{i\geq 0} a_iu^i &\mapsto \sum_i s(a_i)u^i.
\end{align}
Note that the sum on the right side makes sense as $\E$ is finitely generated. We can then write the connection on $\E$ as \begin{equation}
	\nabla =: \frac{d}{du} + A =: \frac{d}{du} + u^{-2}\sum_{i\geq0} A_iu^i, \text{ for some linear maps } A_i: \WE \rightarrow \WE.
\end{equation}
Call $A$ the \emph{connection matrix} and $A_0$ the \emph{residue}. Given two splittings $s_1$ and $s_2$ we obtain an isomorphism \begin{equation}
	R = \Phi_{s_1}^{-1} \circ \Phi_{s_2}: \WE[[u]] \rightarrow \WE[[u]].
\end{equation}
Writing $R = \sum_{i \geq 0} u^i R_i$ for linear maps $R_i: \WE \rightarrow \WE[[u]]$, we find that $R_0 = Id$. The splittings $s_1$ and $s_2$ are then related via: \begin{equation}
	s_2(\cdot) = \sum_{i\geq 0} u^i s_1(R_i(\cdot)).
\end{equation} The connection matrices are related via:
\begin{equation}
	A^{s_2} = R^{-1} A^{s_1} R + R^{-1}\frac{dR}{du},
\end{equation}
which shows that the residue $A_0$ is independent of the choice of splitting. Such a matrix series $R$ is called an \emph{R-matrix}. Usually an extra condition, symplecticity, is imposed on $R$. This condition is satisfied when both splittings are P-compatible.

We now rephrase a theorem by Levelt \cite[Chapter~2]{Lev} in our setup. This theorem is the first step in the Hukuhara-Levelt-Turrittin decomposition. See for example \cite{Mal} for a modern statement.

\begin{nthm}
	\label{decomposition}
	Given an E-structure $\E$ there exists a unique decomposition $\E = \bigoplus_w \E_w$ where the $w$ are the eigenvalues of the residue $A_0: \WE \rightarrow \WE$. This decomposition satisfies: \begin{itemize}
		\item $u^2\nabla (\E_w) \subset \E_w$ for all $w$,
		\item $\pi(\E_w) = \WE_w$, where $\WE_w$ denotes the $w$-generalised eigenspace of the residue.
	\end{itemize}
\end{nthm}
The proof in \cite[Chapter~2]{Lev} is easily seen to apply in our situation. As we will need a specific form of the next term in the connection matrix, $A_1$, we will provide a proof. The main result we need is:
\begin{nlemma}
	\label{decomposition lemma}
	Let $\{e_j\}$ be a basis for $\E$ such that the vectors $\pi(e_j) \in \WE$ are generalised eigenvectors for the residue $A_0$. Write the connection as $\nabla = \frac{d}{du} + u^{-2}\sum_{i\geq0} A_iu^i$ in this basis. Then there exists another basis $\{v_j\}$ for $\E$ such that the following hold: \begin{itemize}
		\item $\pi(v_j) = \pi(e_j)$.
		\item Write the connection as $\nabla = \frac{d}{du} + u^{-2}\sum_{i\geq0} \widetilde{A}_iu^i$ in the basis $\{v_i\}$. Then each $\widetilde{A}_i$ respects the eigenvalue decomposition of $\WE$, that is $\widetilde{A}_i|_{\WE_w}: \WE_w \rightarrow \WE_w$.
		\item $\pi_w \circ \widetilde{A}_1|_{\WE_w} = \pi_ w \circ A_1|_{\WE_w}: \WE_w \rightarrow \WE_w$ for all $w$. Here $\pi_w: \WE \rightarrow \WE_w$.
	\end{itemize}
\end{nlemma}
\begin{proof}[Proof of Lemma \ref{decomposition lemma}]
	Consider a new basis $\lbrace P(e_j) \rbrace $ for some invertible $\mathbb{C}[[u]]$-linear map $P$. The new connection matrix is $\widetilde{A} = P^{-1}AP - P^{-1}\frac{dP}{du}$. If we set $P = id + u^mT_m$, we find (see \cite[Theorem 5.7]{Sab}): $\widetilde{A}_s = A_s$ for $s < m$, \begin{equation}
		\widetilde{A}_m = A_m + [A_0,T_m],
	\end{equation}
	and a more complicated expression for $\widetilde{A}_{>m}$. Let $\WE_w$ for $w \in \mathbb{C}$ be the generalised eigenspaces for $A_0$, then let
	\begin{equation}
		\mathcal{W}_{A_0} = \left\{ \phi: \WE \rightarrow \WE \mid 0 = \pi_w \circ \phi_{|_{\WE_w}}: \WE_w \rightarrow \WE_w \; \text{for all} \; w \right\}.
	\end{equation}
	These are the linear maps which vanish on the diagonal blocks of $A_0$. A short computation shows that the restriction of the adjoint map $ad_{A_0} = [A_0, \_] : \mathcal{W}_{A_0} \rightarrow \mathcal{W}_{A_0}$ is an isomorphism. Thus, there exists a $T_m \in \mathcal{W}_{A_0}$ such that \begin{equation}
		\pi_w \circ [A_0,T_m]|_{\WE_{w'}} = -\pi_w \circ A_m|_{\WE_{w'}}: \WE_{w'} \rightarrow \WE_w.
	\end{equation} That is, all entries of $[A_0,T_m]$ and $-A_m$ which are not in the diagonal blocks of $A_0$ agree. We thus have $\widetilde{A}_m(\WE_w) \subset \WE_w$.
	We then find the $T_m$ successively, starting with $m=1$. Then set $P = \prod_{m \geq 1} (id + u^mT_m)$, noting that this product is well-defined, as for each power of $u$, only finitely many terms in the product contribute. Then set $v_j = P(e_j)$. This shows the first two properties.
	
	For the final property, note that $\widetilde{A}_1 = A_1 + [A_0,T_1]$. As we need that $\widetilde{A}_1|_{\WE_w}: \WE_w \rightarrow \WE_w$, we can choose $T_1$ to only have entries in the off-diagonal blocks. That is, the restriction $T_1: \WE_w \rightarrow \WE_w$ vanishes. But then the same holds for $[A_0,T_1]$.
\end{proof}
\begin{proof}[Proof of Theorem \ref{decomposition}]
	Let $\E_w = \langle v_j | \pi(v_j) \in \WE_w \rangle$. By construction, the $\E_w$ are $u^2 \nabla$-invariant. Uniqueness of the decomposition follows from the following lemma.
\end{proof}
\begin{nlemma}
	\label{uniqueness of decomposition}
	Let $f: \E \rightarrow \E'$ be a morphism of E-structures. Then for any choice of decomposition by eigenvalues $\E = \oplus_w \E_w$ by eigenvalues of $A_0$, and any choice of decomposition of $\E'$ by the eigenvalues of $A_0'$, we have that $f(\E_w) \subset \E'_w$.
\end{nlemma}
\begin{remark}
	Levelt \cite[Chapter 2]{Lev} proves this lemma when $f$ is an isomorphism. Our proof of the general case is very similar.
\end{remark}
Lemma \ref{uniqueness of decomposition} follows directly from:
\begin{nlemma}
	Let $f: \E \rightarrow \E'$ be a morphism of E-structures. Assume that the residues $A_0$ and $A'_0$ have no eigenvalues in common, then $f = 0$.
\end{nlemma}
\begin{proof}
	Expand $f$ in a basis for $\E$ and $\E'$ as a matrix $F = \sum_i u^iF_i$. Expand the connections $\nabla, \nabla'$ as usual with connection matrices $A$ and $A'$. As $f$ respects connection, we obtain the equation: \begin{equation}
		u^2\frac{dF}{du} = FA - A'F.
	\end{equation}
	Expanding in powers of $u$, we find \begin{equation}
		F_0A_0 - A'_0F_0 = 0.
	\end{equation}
	As $A'_0$ and $A_0$ have no eigenvalues in common, this implies $F_0 = 0$ (see \cite[Chapter~8.1]{Gan1}). Next, compare coefficients of $u^{m+1}$. This yields: \begin{equation}
		F_{m+1}A_0 - A'_0 F_{m+1} = L(F_0,F_1, \dots F_{m}),
	\end{equation}
	where $L(F_0, \dots, F_{m-g})$ denotes a linear combination of the $F_{\leq m}$ with vanishing constant term. By induction we can assume $F_0, \dots, F_m$ vanish, which implies $F_{m+1} = 0$.
\end{proof}

\subsection{Semi-simple TEP-structures}
\label{semi-simple TEP structures}
In this section we will interpret results from \cite{AT} in the language of TEP-structures. For simplicity, let $\K = \CC$.

\begin{defi}[semi-simple E(P)-structure]
	\label{Semi-simple EP-structure}
	An E(P)-structure is semi-simple if there exists an isomorphism of E(P)-structures $\E \cong \bigoplus_w \mathcal{E}^{-w/u}$. Here the values $w \in \CC$ are allowed to occur with multiplicity. Let $\xi = u^2\nabla_{\frac{d}{du}}: \WE \rightarrow \WE$ be the residue of the connection. Thus, $\xi$ is given by multiplication by $w$ on each summand $\mathcal{E}^{-w/u}$.
\end{defi}
The following two definitions are inspired by \cite{AT}. 
\begin{defi}[{\cite[Definition~3.9]{AT}}]
	\label{definition of grading operator}
	Let $\E$ be a semi-simple EP-structure with a specified element $\omega \in \WE$. We call a $\CC$-linear map $\mu: \WE \rightarrow \WE$ a \emph{grading operator} and say it is: \begin{enumerate}
		\item P-compatible: if $(\mu(x),y)_{\WE} +(x,\mu(y))_{\WE} = 0$ for all $x,y \in \WE$,
		\item $\xi$-compatible: if the restriction $\pi_w \circ \mu: \WE_{w} \rightarrow \WE_{w}$ vanishes. Here $\pi_w: \WE \rightarrow \WE_w$ denotes the projection onto the $w$-eigenspace of $\xi$,
		\item $\omega$-compatible: if $\mu(\omega) = r \omega$ for some $r \in \CC$ called the \emph{weight} of $\mu$.
		\end{enumerate}
\end{defi}
\begin{defi}[{\cite[Definition~3.7]{AT}}]
	\label{definition ss splitting}
	Let $\E$ be an EP-structure with a specified element $\omega \in \WE$, and $s: \WE \rightarrow \E$ a splitting. We say the splitting is: \begin{enumerate}
		\item P-compatible: if $(s(a),s(b))_\E = (a,b)_{\WE}$ for all $a,b \in \WE$,
		\item Homogeneous: if $\nabla_{\frac{d}{du}} s(a) \in u^{-1} Im(s) + Im(s)$ for all $a \in \WE$,
		\item $\omega$-compatible: if $\nabla_{u\frac{d}{du}} s(\omega) \in r s(\omega) + u^{-1}Im(s)$ for some $r \in \CC$.
	\end{enumerate}
\end{defi}
\begin{eg}
	The EP-structure $\mathcal{E}^{-w/u}$ admits a canonical homogeneous, P- and $\omega$-compatible splitting given by $s^{can}(1) = 1 \in \E$. Here we have not specified the element  $\omega \in \WE$, as the splitting is $\omega$-compatible for any choice of $\omega$. This is because, by definition, $\nabla_{u\frac{d}{du}} s^{can}(a) = u^{-1}ws^{can}(a)$ for all $a \in \WE$.
\end{eg}
\begin{eg}
	\label{splitting on SS E-structure}
	A semi-simple E-structure $\E$ comes with a canonical splitting induced by the isomorphism $\Phi: \E \cong \bigoplus_w \mathcal{E}^{-w/u}$ and the splitting $s^{can}$ on each $\E^{-w/u}$. This splitting is independent of the choice of isomorphism $\Phi$ as any two such isomorphisms are related by an isomorphism $\Psi: \bigoplus_w \mathcal{E}^{-w/u} \rightarrow \bigoplus_w \mathcal{E}^{-w/u}$ and any such $\Psi$ is necessarily independent of $u$. We denote this splitting by $s^{ss}$ (the \emph{semi-simple splitting}) and note that it is homogeneous and $\omega$-compatible, for any $\omega \in \WE$. If $\E$ is a semi-simple EP-structure, $s^{ss}$ is also P-compatible.
\end{eg}
\begin{remark}
	A splitting is homogeneous if and only if the associated connection matrix $A = u^{-2}\sum_{i \geq 0} u^iA_i$ satisfies $A_i = 0$ for $i \geq 2$.
\end{remark}
Amorim and Tu show the following for EP-structures coming from the cyclic homology of an $A_\infty$-category, see {\cite[Theorem~3.10]{AT}}. We state their result in our more general setup. The proof is identical.

\begin{nthm}
	\label{Grading operators and splittings}
	Let $\E$ be a semi-simple EP-structure with a specified element $\omega \in \WE$. Then there exists a bijection between the set of homogeneous, P- and $\omega$-compatible splittings $s: \WE \rightarrow \E$ and the set of P-, $\xi$- and $\omega$-compatible grading operators $\mu: \WE \rightarrow \WE$.
\end{nthm}

We refer the reader to {\cite[Theorem~3.10]{AT}} for the details of the proof, but will say a few words about it. Given a splitting $s$ as in the lemma, there exists a unique series $R = \sum_{i\geq 0} u^iR_i$, where $R_i: \WE \rightarrow \WE$ and $R_0 = Id$ such that \begin{equation}
	\label{Splitting equation}
	s(a) = \sum_{i\geq 0} u^is^{can}(R_i(a)).
\end{equation}
The associated grading operator is then defined by $\mu^s = [\xi,R_1]$, and one obtains the following relation on $R$: \begin{equation}
	\label{R matrix equation}
	[\xi,R_{k+1}] = R_k(\mu^s - k).
\end{equation} 
One then checks all the required properties hold. Conversely, given a grading operator $\mu: \WE \rightarrow \WE$, \cite{AT} show that there exists a unique R-matrix solving Equation \eqref{R matrix equation} and then define the splitting $s^{\mu}$ by Equation \eqref{Splitting equation}.

Now let $\mathcal{H}$ be an Euler-graded TEP-structure over a $\CC$-algebra $R$. Let $\widetilde{\HH} = \HH/u\HH$. The following definition is originally due to Saito \cite[Definition~3.1]{Sai}.

\begin{defi}[{\cite[Definition~4.1]{AT}}]
	An element $\zeta \in \HH$ is called a \emph{primitive form} if it satisfies: \begin{itemize}
		\item (Primitivity) The map defined by \begin{equation}
			\rho^{\zeta} : Der_{\CC}(R) \rightarrow \WHH, \; \rho^{\zeta}(v) = [u \nabla_v \zeta]
		\end{equation}
			is an isomorphism.
		\item (Orthogonality) For any tangent vectors $v_1, v_2 \in Der_\CC(R)$, we have:\begin{equation}
			(u\nabla_{v_1} \zeta, u\nabla_{v_2} \zeta)_{\HH} \in R.
		\end{equation}
	\item (Holonomicity) For any tangent vectors $v_1, v_2, v_3 \in Der_\CC(R)$, we have: \begin{equation}
		(u\nabla_{v_1} u\nabla_{v_2} \zeta, u\nabla_{v_3} \zeta)_{\HH} \in R \oplus u\cdot R.
	\end{equation}
\item (Homogeneity) There exists a constant $r \in \CC$ such that \begin{equation}
	Gr(\zeta) = 2r \zeta.
\end{equation}
	\end{itemize}
If $\zeta$ only satisfies the Primitivity property, we will call $\zeta$ a \emph{primitive element}, and call the TEP-structure $\HH$ \emph{primitive} if such $\zeta$ exists.
\end{defi}

\begin{defi}
	Let $\mathcal{H}$ be a primitive TEP-structure over $R = \CC[[t_1, \dots, t_n]]$. Let $\E$ be the EP-structure $\E := \HH \otimes_R \CC$, where $\CC$ is an $R$-module under the map $t_i \mapsto 0$. For $\omega \in \WE$ say $\omega$ is \emph{primitive} if there exists a primitive element $\zeta \in \HH$ such that $\zeta|_{t= u = 0} = \omega$.
\end{defi}

Amorim and Tu \cite[Theorem~4.2]{AT} also prove a bijection between primitive forms and splittings, which is a bijection originally established in \cite{Sai2}. We rephrase their theorem to apply to our setup. As already observed by \cite[Remark~4.3]{AT}, their proof applies to our more general setup (note that what they call a VSHS corresponds to what we call a TE-structure).
\begin{nthm}
	\label{Primitive forms and splittings}
	Let $\mathcal{H}$ be a Euler-graded, primitive TEP-structure over $R = \CC[[t_1, \dots, t_n]]$ and let $\omega \in \WE$ be primitive. Then there exists a natural bijection between the following two sets: \begin{align}
		\mathcal{P} &:= \{ \zeta \in \HH | \zeta \text{ is a primitive form with } \zeta|_{t= 0, u=0} = \omega\},\\
		\mathcal{S} &:= \{ \text{homogeneous, P- and $\omega$-compatible splittings } s: \WE \rightarrow \E \}.
	\end{align}
\end{nthm}

\begin{defi}
	Let $\mathcal{H}$ be a TEP-structure over $R = \CC[[t_1, \dots, t_n]]$. We say $\HH$ is \emph{semi-simple} if the associated EP-structure $\E := \HH \otimes_R \CC$ is semi-simple.
\end{defi}

For a semi-simple, Euler graded and primitive TEP-structure $\HH$ as above, with a choice of primitive $\omega \in \WE$, Theorems \ref{Grading operators and splittings} and \ref{Primitive forms and splittings} thus combine to give a bijection between the set of P-, $\xi$- and $\omega$-compatible grading operators $\mu: \WE \rightarrow \WE$ and the set of primitive forms $\zeta \in \HH$ with $\zeta|_{t= 0, u=0} = \omega$.
\begin{ncor}
	\label{Primitive forms and grading operators}
	Let $\mathcal{H}$ be a Euler-graded, primitive, semi-simple TEP-structure over $R = \CC[[t_1, \dots, t_n]]$ and let $\omega \in \WE$ be primitive. Then there exists a natural bijection between the following two sets: \begin{align}
		\mathcal{P} &:= \{ \zeta \in \HH | \zeta \text{ is a primitive form with } \zeta|_{t= 0, u=0} = \omega\},\\
		\mathcal{G} &:= \{\text{P-, }\xi\text{- and }\omega\text{-compatible grading operators }\mu: \WE \rightarrow \WE \}.
	\end{align}
\end{ncor}

The relevance of this bijection is that given a primitive form $\zeta$ as above, Saito and Takahashi \cite{SaiTak} endow $Spec(R)$ with the structure of a Frobenius manifold. A grading operator $\mu$ on a semi-simple TEP-structure over a ring $R$ thus gives rise to a Frobenius manifold $\MM_{\mu}$. In chapter \ref{Quantum cohomology} we will come back to this construction.

\section{TE-structure on the cyclic homology of an $A_\infty$-algebra}
\label{Cyclic homology}
In this section we will define a TE-structure on the cyclic homology of an $A_\infty$-algebra. All of the definitions can easily be extended to $A_\infty$-categories.
Let $S_n[k]$ be the set of all partitions of $\lbrace 1, \dots k\rbrace$ into $n$ ordered sets of the form $(1, 2, \dots, k_1)$, $(k_1 + 1, \dots, k_1 + k_2), \dots ,(k_1 + \dots + k_{n-1} + 1, \dots, k_1 + \dots + k_n)$. Let $(i:n)$ denote the $i$th set of the partition. The size of $(i:n)$ is $k_i$. We allow for the case $k_i = 0$.

Let $\K$ be a field, and assume $\K$ is complete when equipped with the trivial valuation. Let $R$ be a $\mathbb{Z}/2$-graded $\K$-algebra with a complete valuation $\zeta_R: R \rightarrow \RR_{\geq 0} \cup \{ \infty\}$. Let $A$ be a $\mathbb{Z}/2$-graded $R$-module with a complete valuation $\zeta_A: A \rightarrow \RR_{\geq 0}\cup \{ \infty\}$. Let $|\alpha|$ denote the degree of $\alpha \in A$, and $|\alpha|' := |\alpha| - 1$ the shifted degree. For $\alpha = (\alpha_1,\dots,\alpha_k)$ let $\epsilon(\alpha) = \sum_{j=1}^k |\alpha_j|'$. For a partition $P \in S_n[k]$ let $\epsilon_i = \epsilon(\alpha^{(i:n)})$.

Let $I \coprod J = [l]$ be a partition of $[l]$ in the usual sense, not respecting the order of $[l]$.
Equip the subsets $I$ and $J$ with the order induced from $[l]$. 
\begin{defi}[{\cite[Definition~1.1]{ST3}}]
	An $n$-dimensional, strictly unital, curved, cyclic, $\mathbb{Z}/2$-graded $A_\infty$-structure on $A$ consists of:
	\begin{itemize}
		\item $R$-multilinear maps $\mathfrak{m}_k: A^{\otimes{k}} \rightarrow A[2-k]$. This means that for $t \in R$ and $\alpha_1, \dots, \alpha_k \in A$: \begin{equation}
		\mathfrak{m}_k(\alpha_1, \dots, t\alpha_i, \dots, \alpha_k) = (-1)^{|t|(|\alpha_1|' + \dots + |\alpha_{i-1}|' + 1)}t\mathfrak{m}_k(\alpha_1, \dots, \alpha_i, \dots, \alpha_k).
		\end{equation}
		\item A pairing $\langle \; , \; \rangle: A \otimes A \rightarrow R[n]$.
		\item An element $\mathbf{e} \in A$ with $|\mathbf{e}| = 0$.
	\end{itemize} 
	These satisfy the following relations:
	\begin{enumerate}
		\item The $A_\infty$ relations hold: \begin{equation}
		\sum_{P \in S_3[k]} (-1)^{\epsilon_1}\mathfrak{m}_{k_1 + 1 + k_3}(\alpha^{(1:3)}, \mathfrak{m}_{k_2}(\alpha^{(2:3)}),\alpha^{(3:3)}) = 0.\end{equation}
		\item The pairing $\langle \; , \; \rangle$ is graded $R$-bilinear: \begin{equation}
		\langle a \alpha_1, \alpha_2 \rangle = a \langle \alpha_1, \alpha_2 \rangle = (-1)^{|a||\alpha_1|'}\langle \alpha_1, a \alpha_2 \rangle.
		\end{equation}
		\item The pairing is graded anti-symmetric: \begin{equation}
			\label{graded symmetric pairing}
			\langle \alpha_1, \alpha_2 \rangle = (-1)^{|\alpha_1|'|\alpha_2|'+1}\langle \alpha_2, \alpha_1 \rangle.
		\end{equation}
		\item The pairing is cyclic: for $\alpha = (\alpha_1, \dots, \alpha_k)$, we have
		\begin{equation}
		\langle \mathfrak{m}_k(\alpha_1, \dots, \alpha_k), \beta \rangle = (-1)^{|\beta|'\epsilon(\alpha)} \langle \mathfrak{m}_k(\beta, \alpha_1, \dots, \alpha_{k-1}), \alpha_k \rangle.
		\end{equation}
		\item The unit is strict: \begin{enumerate}
			\item $\mathfrak{m}_k(\alpha_1, \dots, \alpha_{i-1}, \mathbf{e}, \alpha_{i+1}, \dots, \alpha_k) = 0, \; \forall k \neq 0,2,$
			\item $\mathfrak{m}_2(\mathbf{e},\alpha) = \alpha = (-1)^{|\alpha|}\mathfrak{m}_2(\alpha, \mathbf{e})$.
		\end{enumerate}
		\item The $A_\infty$-operations respect the valuation: \begin{enumerate}
		\item $\zeta_A(\m_k(\al)) \geq \zeta_A(\al)$,
		\item $\m_0 = w \cdot \mathbf{e} + \overline{\m}_0$, with $\zeta_A(\overline{\m}_0) > 0$,
		\item $\zeta_R(\langle \al_1, \al_2 \rangle) \geq \zeta_A(\al_1) + \zeta_A(\al_2)$.
		\end{enumerate}
	\end{enumerate}
\end{defi}

Given an $A_\infty$-algebra $A$, recall the opposite $A_\infty$-algebra as defined by \cite[definition~3.5]{She}. As an $R$-module we have $A^{op} = A$. But now: \begin{equation}
	\mathfrak{m}_k^{op}(\al_1, \dots, \al_k) := (-1)^{\dagger(\al)}\mathfrak{m}_k(\al_k, \dots, \al_1).
\end{equation}
Here, for $\al = (\al_1, \dots, \al_k)$, we set $\dagger(\al) = \sum_{1 \leq i <j \leq k} |\al_i|'|\al_j|'$. If $A$ is strictly unital with unit $\mathbf{e} \in A$, then $A^{op}$ is strictly unital with unit $\mathbf{e}^{op} := -\mathbf{e}$.

We can also define:
\begin{defi}[negative $A_\infty$-algebra]
	Let the \emph{negative} $A_\infty$-algebra be given by $A^{-} := A$ as R-modules, but $\mathfrak{m}_k^{-}(\al_1, \dots, \al_k) = (-1)^{k-1} \mathfrak{m}_k(\al_1, \dots, \al_k)$.
	\end{defi}
Note that there is an isomorphism of $A_\infty$-algebras $A \cong A^{-}$ given by $\al \mapsto -\al$. Combining both definitions we have \begin{defi}[negative-opposite $A_\infty$-algebra]
	\label{negative-opposite algebra}
	The \emph{negative-opposite} algebra associated to $A$ is given by $(A^{op})^{-} =: A^{-op}$. This is then an n-dimensional, strictly unital, cyclic, $\mathbb{Z}/2$-graded $A_\infty$-algebra with unit $\mathbf{e} \in A$.
\end{defi}
\subsection{Hochschild invariants}
\label{Hochschild invariants}
Let $(A,\mathfrak{m})$ be an n-dimensional, cyclic, strictly unital, curved, $\mathbb{Z}/2$-graded $A_{\infty}$-algebra over $R$. Let $\mathfrak{r} \subset R$ be the maximal ideal of elements with positive valuation. Define the (reduced) Hochschild cochains of $A$:\begin{equation}
	CC^*(A) := \prod_{k=0}^\infty Hom_R \left(  \left(\frac{A}{R\cdot \mathbf{e}}[1]\right)^{\otimes k}, A[1] \right).
\end{equation}
Also define the uncompleted (reduced) Hochschild chains of $A$ to be: \begin{equation}
CC^{unc}_*(A) := \bigoplus_{k=0}^\infty A \otimes \left(\frac{A}{R \cdot \mathbf{e}}\right)^{\otimes k}.
\end{equation}
Following \cite[Section~3.5]{CLT2018}, we define:
\begin{defi}
	The completed reduced Hochschild chains and cochains are given by:
	\begin{align}
		CC_*(A) = \varprojlim CC^{unc}_*(A/\mathfrak{r}^k).
	\end{align}
\end{defi}
\begin{remark}
	In the remainder of this chapter, we will recall and define various operations on Hochschild (co)chains. For simplicity, we will often define these operations only for the uncompleted chains. They descend to operations on the completed Hochschild chains as the $A_\infty$-operations and the pairing are assumed to respect the valuation.
\end{remark}
\begin{remark}
	We need to be careful about the $R$-linearity of Hochschild cochains. For $\phi \in CC^*(A)$ and $t\in R$ this means that: \begin{equation}
		\phi(\al_1, \dots, t\al_i, \dots, \al_k) = (-1)^{|t|(|\al_1|' + \dots + |\al_{i-1}|' + |\phi|')} t\phi(\al_1, \dots, \al_i, \dots, \al_k).
	\end{equation}
\end{remark}

Denote an element $\alpha \in CC_*(A)$ by $\alpha = \alpha_0[\alpha_1|\dots|\alpha_k]$, and for a subset $I\subset \lbrace 1, \dots, k \rbrace$, write $\alpha^I$ for $\bigotimes_{j \in I} \alpha_j$. 
\begin{defi}
	Hochschild homology is defined as: $HH_*(A) := H^*(CC_*(A), b)$. Here the differential $b$ is given by:
	\begin{align}
		b(\alpha) &= \sum_{\substack{ P \in S_3[k]}} (-1)^{\epsilon_3(\epsilon_2 + \epsilon_1 + |\alpha_0|')} \mathfrak{m}_{k_3 + 1 + k_1}(\alpha^{(3:3)}\otimes \alpha_0 \otimes \alpha^{(1:3)})[\alpha^{(2:3)}] \nonumber\\
		&+\sum_{\substack{ P \in S_3[k]}} (-1)^{\epsilon_1 + |\alpha_0|'}\alpha_0[\alpha^{(1:3)}|\mathfrak{m}_{k_2}(\alpha^{(2:3)})|\alpha^{(3:3)}].
	\end{align}
	Note here that in the second sum, terms with $k_2 = 0$ are allowed.
\end{defi} 

\begin{defi}
	The negative cyclic chain complex is given by \begin{equation}
		CC_*^-(A) = \varprojlim (CC^{unc}_*(A/\mathfrak{r}^k)[[u]], b+uB).
	\end{equation}
	Here the second differential $B$ is defined by: \begin{equation}
		B(\alpha) = \sum_{\substack{ P \in S_2[k]}} (-1)^{\epsilon_2(|\alpha_0|'+\epsilon_1)} 1[\alpha^{(2:2)}|\alpha_0|\alpha^{(1:2)}].
	\end{equation}
	The homology of the negative cyclic chain complex is called the \emph{negative cyclic homology}, denoted $HC^-_*(A)$.
\end{defi}

Hochschild cochains admit a differential too. First introduce the Gerstenhaber product, defined by:
\begin{equation}
\phi \circ \psi(\alpha_1, \dots, \alpha_k) = \sum_{P \in S_3[k]} (-1)^{|\psi|' \epsilon_1} \phi(\alpha^{(1:3)} \otimes \psi(\alpha^{(2:3)}) \otimes \alpha^{(3:3)}).
\end{equation}
The Gerstenhaber bracket is then defined by: $[\phi, \psi] = \phi \circ \psi - (-1)^{|\phi|'|\psi|'} \psi \circ \phi$. The $A_\infty$-structure maps $\mathfrak{m}_k$ define a Hochschild chain $\mathfrak{m} \in CC^2(A)$. The differential on $CC^*(A)$ is then given by $[\mathfrak{m}, \_]$. Hochschild cohomology is defined as $HH^*(A) := H^*(CC^*(A),[\mathfrak{m}, \_])$.

Finally we note that $CC^*(A)$ admits an $A_\infty$-structure $M^k$, defined by \cite{Ge}. $M^1$ is the differential $[\mathfrak{m}, \_]$. We will also need the $M^2$ part of these operations.
\begin{defi}
	The cup product on $CC^*(A)$ is given by $\psi \cup \phi := (-1)^{|\psi|}M^2(\psi,\phi)$. Here \begin{equation}
	M^2(\psi,\phi)(\alpha) := \sum_{P\in S_5[k]} (-1)^{|\psi|'\epsilon_1 + |\phi|'(\epsilon_1+\epsilon_2+\epsilon_3)}\mathfrak{m}(\alpha^{(1:5)}\otimes \psi(\alpha^{(2:5)})\otimes \alpha^{(3:5)} \otimes \phi(\alpha^{(4:5)})\otimes\alpha^{(5:5)}).
	\end{equation} 
\end{defi}

Since $A$ is cyclic, we can consider the pairing:
\begin{align}
	\label{pairing on hochschild chains}
	(\;, \;): CC^*(A) &\otimes CC_*(A) \rightarrow R\nonumber\\
	\phi & \otimes \alpha \mapsto (-1)^{|\alpha_0|(\epsilon(\widetilde{\alpha}) + 1)}\langle \phi(\alpha_1, \dots, \alpha_k), \alpha_0 \rangle.
\end{align}
Here, for $\alpha \in CC_*(A)$ we use $\widetilde{\alpha}$ to denote the tuple $(\alpha_1, \dots, \alpha_k)$.
\begin{nlemma}
	\label{lem: pairing hochschil (co)chains}
	The pairing above descends to a pairing $HH^*(A) \otimes HH_*(A) \rightarrow R$. Concretely, we have $([\mathfrak{m},\phi],\alpha) + (-1)^{|\phi|}(\phi,b(\alpha)) = 0$.
\end{nlemma}
\begin{proof}
	We first write out:
	\begin{align}
		([\mathfrak{m},\phi],\alpha) &= (\mathfrak{m} \circ \phi - (-1)^{|\phi|'|\mathfrak{m}|'}\phi \circ \mathfrak{m}, \alpha)\\
		\label{hochschild codifferential pairing first term}
		&= \sum_{P \in S_3[k]} (-1)^{|\alpha_0|(\epsilon(\widetilde{\alpha})+ 1) + |\phi|' \epsilon_1} \langle \mathfrak{m}(\alpha^{(1:3)} \otimes \phi(\alpha^{(2:3)}) \otimes \alpha^{(3:3)}),\alpha_0\rangle\\
		\label{hochschild codifferential pairing second term}
		&\qquad + \sum_{P \in S_3[k]} (-1)^{|\phi|'+1+|\alpha_0|(\epsilon(\widetilde{\alpha}) + 1) + \epsilon_1} \langle \phi(\alpha^{(1:3)} \otimes \mathfrak{m}(\alpha^{(2:3)}) \otimes \alpha^{(3:3)}),\alpha_0\rangle.
	\end{align}
	Next, we write out:
	\begin{align}
		(\phi,b(\alpha)) &= \sum_{\substack{ P \in S_3[k]}} (-1)^{\epsilon_3(\epsilon_2 + \epsilon_1 + |\alpha_0|')} (\phi,\mathfrak{m}_{k_3 + 1 + k_1}(\alpha^{(3:3)}\otimes \alpha_0 \otimes \alpha^{(1:3)})[\alpha^{(2:3)}])\\
		&\qquad +\sum_{\substack{ P \in S_3[k]}} (-1)^{\epsilon_1 + |\alpha_0|'} (\phi,\alpha_0[\alpha^{(1:3)}|\mathfrak{m}_{k_2}(\alpha^{(2:3)})|\alpha^{(3:3)}])\\
		\label{hochschild differential pairing first term}
		&= \sum_{\substack{ P \in S_3[k]}} (-1)^{\epsilon_3(\epsilon_2 + \epsilon_1 + |\alpha_0|')+ A_{1}} \langle \phi(\alpha^{(2:3)}),\mathfrak{m}_{k_3 + 1 + k_1}(\alpha^{(3:3)}\otimes \alpha_0 \otimes \alpha^{(1:3)}) \rangle\\
		\label{hochschild differential pairing second term}
		&\qquad+\sum_{\substack{ P \in S_3[k]}} (-1)^{\epsilon_1 + |\alpha_0|' + |\alpha_0|\epsilon(\widetilde{\alpha})} \langle \phi(\alpha^{(1:3)}\otimes \mathfrak{m}_{k_2}(\alpha^{(2:3)})\otimes\alpha^{(3:3)}),\alpha_0\rangle.
	\end{align}
	Where $A_1 = (\epsilon_2+1)(|\alpha_0|'+\epsilon_1+\epsilon_3)$. The terms \ref{hochschild codifferential pairing second term} and \ref{hochschild differential pairing second term} agree up to a factor $(-1)^{|\phi|'}$, as required. We then use cyclic symmetry to make \ref{hochschild differential pairing first term} agree with \ref{hochschild codifferential pairing first term}, up to the same factor.
	We compute \begin{align}
		&\qquad\sum_{\substack{ P \in S_3[k]}} (-1)^{\epsilon_3(\epsilon_2 + \epsilon_1 + |\alpha_0|')+A_{1}} \langle \phi(\alpha^{(2:3)}),\mathfrak{m}_{k_3 + 1 + k_1}(\alpha^{(3:3)}\otimes \alpha_0 \otimes \alpha^{(1:3)}) \rangle\\
		&=\sum_{\substack{ P \in S_3[k]}} (-1)^{\epsilon_3(\epsilon_2 + \epsilon_1 + |\alpha_0|')+A_{1} +A_2} \langle \mathfrak{m}_{k_3 + 1 + k_1}(\alpha^{(3:3)}\otimes \alpha_0 \otimes \alpha^{(1:3)}),\phi(\alpha^{(2:3)}) \rangle,
		\intertext{by \ref{graded symmetric pairing}, where  $A_2 = (\epsilon_2 + |\phi|')(|\alpha_0|'+\epsilon_1+\epsilon_3+1)$. And finally:}
		&=\sum_{\substack{ P \in S_3[k]}} (-1)^{\epsilon_3(\epsilon_2 + \epsilon_1 + |\alpha_0|')+A_{1} +A_2 + A_3} \langle \mathfrak{m}(\alpha^{(1:3)} \otimes \phi(\alpha^{(2:3)}) \otimes \alpha^{(3:3)}),\alpha_0\rangle,
	\end{align}
	where $A_3 = (|\alpha_0|'+\epsilon_3)(|\phi|'+\epsilon_2+\epsilon_1)$.
	Combining all the signs, we get the required equality.
\end{proof}

We will also need the bilinear maps $b^{1,1}, B^{1,1} : CC^*(A) \otimes CC_*(A) \rightarrow CC_*(A)$, defined for a Hochschild cochain $\phi \in CC^*(A)$, and $\alpha = \alpha_0[\alpha_1|\dots|\alpha_k] \in CC_*(A)$ by: \begin{equation}
	\label{b11}
b^{1,1}(\phi;\alpha) = \sum_{\substack{P \in S_5[k]}} (-1)^{\dagger} \mathfrak{m}_{k_1+k_3+k_5+2}(\alpha^{(3:5)}\otimes\phi(\alpha^{(4:5)})\otimes\alpha^{(5:5)}\otimes \alpha_0 \otimes \alpha^{(1:5)})[\alpha^{(2:5)}],
\end{equation}
where $\dagger = (\epsilon_3 + \epsilon_4 + \epsilon_5)(|\alpha_0|' + \epsilon_1 +\epsilon_2) +|\phi|'\epsilon_3$,
and
\begin{equation}
	\label{B11}
B^{1,1}(\phi;\alpha) = \sum_{\substack{P \in S_4[k]}} (-1)^{|\phi|'\epsilon_2 + (\epsilon_1 + |\alpha_0|')(\epsilon_2 + \epsilon_3 + \epsilon_4)} \mathbf{e}[\alpha^{(2:4)}|\phi(\alpha^{(3:4)})|\alpha^{(4:4)}|\alpha_0|\alpha^{(1:4)}].
\end{equation}
There is a nice interplay between the cup product on $CC^*(A)$, the $A_\infty$-module structure $b^{1,1}$, and the pairing. \begin{nlemma}Define the cap product $\phi \cap \alpha = (-1)^{|\phi|}b^{1,1}(\phi,\alpha)$. We then have:\begin{equation}
	(\phi \cup \psi,\alpha) = (\phi, \psi \cap \alpha).\end{equation}
\end{nlemma}
\begin{proof}
	This is an easy verification of signs using the cyclic symmetry of the pairing.
\end{proof}
\begin{defi}
	Set $i\lbrace \phi \rbrace = b^{1,1}(\phi, \_ ) + u B^{1,1}(\phi, \_ ): CC^{-}_*(A) \rightarrow CC^-_*(A)$.
\end{defi} 
The curvature, or Lie derivative, $\mathcal{L}: CC^*(A) \otimes CC_*(A) \rightarrow CC_*(A)$ defined in \cite{Ge} can be written as:
\begin{align}
	\mathcal{L}_\phi(\alpha) &= \sum_{\substack{ P \in S_3[k]}} (-1)^{\epsilon_3(\epsilon_2 + \epsilon_1 + |\alpha_0|')}\phi_{k_3 + 1 + k_1}(\alpha^{(3:3)}\otimes \alpha_0 \otimes \alpha^{(1:3)})[\alpha^{(2:3)}]\nonumber\\
	&+\sum_{\substack{ P \in S_3[k]}} (-1)^{|\phi|'(\epsilon_1 + |\alpha_0|')} \alpha_0[\alpha^{(1:3)}|\phi_{k_2}(\alpha^{(2:3)})|\alpha^{(3:3)}].
\end{align}
Observe that $\mathcal{L}_{\mathfrak{m}} = b$. An easy computation shows:
\begin{nlemma}
	\label{L properties}
	For any $\phi, \psi \in CC^*(A)$, we have: $[\mathcal{L}_\psi,\mathcal{L}_\phi] = \mathcal{L}_{[\psi,\phi]}$. In particular $[b,\mathcal{L}_\phi] = \mathcal{L}_{[\m,\phi]}$. Furthermore $[B,\mathcal{L}_\phi] = 0$.
\end{nlemma}
\begin{remark}
	\label{commutator}
	For any linear maps $A$ and $B$ of homogeneous degrees the commutator $[A,B]$ is defined as the supercommutator: $[A,B] = AB - (-1)^{|A||B|}BA$. By extending linearly, this defines the commutator for all linear maps.
\end{remark} Getzler shows:
\begin{nprop}(\cite[Theorem~2.2]{Ge})
	\label{Cartan Homotopy}
	\begin{equation}
	\lbrack i\lbrace \phi \rbrace, b+uB \rbrack = u \mathcal{L}_{\phi} + i\lbrace \lbrack \mathfrak{m}, \phi \rbrack \rbrace.
	\end{equation}
\end{nprop}

Getzler furthermore defines a connection in the base directions. We extend his definition to allow for the case when $R$ is $\ZZ/2$-graded.
\begin{defi}

	The Getzler-Gauss-Manin connection is defined on the chain level by:
	\begin{align}
		\label{GGM on cyclic homology}
		\nabla^{GGM}: Der_{\mathbb{C}}R \otimes_\mathbb{K} CC^-_*(A) &\rightarrow u^{-1}CC^{-}_*(A),\nonumber\\
		\nabla^{GGM}_v(\alpha) &:= v(\alpha) + (-1)^{|v|+1}u^{-1}i\{ v(\mathfrak{m}) \}(\alpha).
	\end{align}
\end{defi}
Here for a Hochschild cochain $\phi \in CC^*(A)$ and a derivation $v \in Der_\CC R$, the Hochschild cochain $v(\phi)$ is defined as \begin{equation}
	v(\phi)(\al) := v(\phi(\al)) - (-1)^{|\phi|'|v|} \phi(v(\al)).
\end{equation}
Getzler shows the connection descends to the level of cohomology and is flat. This endows $HC^{-}_*(A)$ with a $\ZZ/2$-graded pre-T-structure over $Spec(R)$. 

Sheridan proves the following holds over a field $\K$. Nothing in their proof breaks down if we work over a general ring. We thus have: \begin{nthm}[{\cite[Theorem~B.2]{She}}]
	\label{Sheridan GGM invariant up to homotopy}
	Let $F: \mathcal{C} \rightarrow \mathcal{D}$ be an $R$-linear $A_\infty$-morphism. Then \begin{equation}F_*: HC_*^-(\MC) \rightarrow HC^-_*(\mathcal{D})\end{equation} is a morphism of pre-T-structures.
\end{nthm}

\subsection{The u-connection}
The pre-T-structure $HC^-_*(A)$ has been extended to a pre-TE-structure by \cite{KKP}. We give another interpretation of this definition. First recall the notion of an Euler-grading on an $A_\infty$-algebra:
\begin{defi}
	\label{Euler-grading A infinity category}
	An Euler-grading on an $n$-dimensional, strictly unital, cyclic, $\mathbb{Z}/2$-graded $A_\infty$-algebra $A$ consists of an Euler vector field $E \in Der_\K R$ of even degree and an even degree map $Gr: A \rightarrow A$ such that \begin{equation}
		 Gr \circ \m_k = \m_k \circ Gr + (2-k) \m_k.
	\end{equation}
	and \begin{equation}
	Gr(f\al) = 2E(f)\al + fGr(\al) \text{ for } f\in R \text{ and } \al \in hom(L_1,L_2). 
\end{equation}
Furthermore, we require that $Gr(\mathbf{e}) = 0$. An $R$-linear $A_\infty$-morphism $F: A \rightarrow B$ is said to be Euler-graded if $E_A = E_B$ and $F^k \circ Gr_A = Gr_B \circ F^k + (k-1)F^k$.
\end{defi}
Now suppose that $A$ is Euler-graded. Consider $Gr: A \rightarrow A$ as a length-1 Hochschild cochain. Then define the operator $Gr^{-}: CC^-_*(A) \rightarrow CC_*^-(A)$ by \begin{equation}
	Gr^{-} := \mathcal{L}_{Gr} + \Gamma + 2 u \frac{\partial}{\partial u},
\end{equation}
where $\Gamma(\alpha_0[\alpha_1|\dots|\alpha_k]) = -k \alpha_0[\alpha_1|\dots|\alpha_k]$ is the length operator on cyclic chains. 
\begin{nlemma}
	The grading $Gr^{-}$ descends to cyclic homology, and endows $HC^-_*(A)$ with an Euler-graded T-structure.
\end{nlemma}
\begin{proof}
	Let $\m' = \prod_k (2-k)\m_k \in CC^*(A)$. A short computation shows $[\Gamma,b] = b - \mathcal{L}_{\m'}$ and $[\Gamma,B] = -B$. We thus have: \begin{align}
		[Gr^-,b+uB] &=  [\mathcal{L}_{Gr},b+uB] + [\Gamma,b+uB] + 2 [u \frac{d}{du},b+uB]\\
		&= \mathcal{L}_{[Gr,\m]} + b - \mathcal{L}_{m'} - uB + 2uB\\
		&= \mathcal{L}_{\m'} + b -\mathcal{L}_{\m'} +uB\\
		&= b+uB.
	\end{align}
The second equality follows by Lemma \ref{L properties}. This shows $Gr^-$ descends to cyclic homology. Next, observe that for $f \in R[[u]]$ and $\al \in CC_*^-(A)$, we have $\mathcal{L}_{Gr}(f\alpha) = 2E(f)\al + fGr(\al)$. This shows that \begin{equation}
	Gr^-(f\al) = (2u\partial_u + 2E)(f)\al + f Gr^- \al
\end{equation}
holds on the chain level. Next, for $v \in Der_\mathbb{C} R$, we want to compute $[Gr^-,\nabla_v]$. To this end, first observe that, after picking a basis for $A$, we have: $[\mathcal{L}_{Gr},v](\al) = [2E,v](\al)$. Furthermore, a direct computation shows: \begin{align}
	[\mathcal{L}_{Gr} + \Gamma,b^{1,1}(v(\m),\cdot)] &= b^{1,1}([2E,v](\m),\cdot) + 2b^{1,1}(v(\m),\cdot)\\
	[\mathcal{L}_{Gr} + \Gamma,B^{1,1}(v(\m),\cdot)] &= B^{1,1}([2E,v](\m),\cdot).
\end{align}
We thus have: \begin{align}
	[Gr^-,\nabla_v] &= [\mathcal{L}_{Gr} + \Gamma + 2u\partial_u,v - B^{1,1}(v(\m),\cdot) - u^{-1}b^{1,1}(v(\m),\cdot)]\\
	&= [\mathcal{L}_{Gr},v] - [\mathcal{L}_{Gr} + \Gamma,b^{1,1}(v(\m),\cdot)] - [\mathcal{L}_{Gr} + \Gamma,B^{1,1}(v(\m),\cdot)] + 2u^{-1}b^{1,1}(v(\m),\cdot)\\
	&= [2E,v] - B^{1,1}([2E,v](\m),\cdot) -u^{-1} b^{1,1}([2E,v](\m),\cdot)\\
	&= \nabla_{[2E,v]}.
\end{align}
\end{proof}
Thus, as an Euler-graded pre-T-structure naturally admits an extension to a pre-TE-structure, any Euler-graded $A_\infty$-algebra naturally admits a pre-TE-structure on $HC^-_*(A)$. For an arbitrary $A_\infty$-algebra $\MC$, we will now define an Euler-graded deformation, and use this to define a u-connection on $HC^-_*(\MC)$.
\begin{defi}
	Let $\MC$ be any $R$-linear $A_\infty$-algebra. Define the $R[s,s^{-1}]$-linear $A_\infty$-algebra $\MC^s := \MC \otimes_R R[s,s^{-1}]$, where $s$ is of odd degree. The operations are defined by: \begin{equation}
		\mathfrak{m}_k^s(\alpha_1, \dots, \alpha_k) := s^{2-k}\mathfrak{m}_k(\alpha_1, \dots, \alpha_k),
	\end{equation}
	and extending s-linearly.
\end{defi}
\begin{nlemma}
	Define $Gr: hom(\MC^s, \MC^s) \rightarrow hom(\MC^s, \MC^s)$ by setting $Gr(s^ka) := ks^ka$ for $a \in hom(\MC,\MC)$. This makes $\MC^s$ a $\ZZ$-graded algebra. In particular, by defining $E = \frac{s}{2}\frac{d}{ds} \in Der_{\K} R[s,s^{-1}]$, $\MC^s$ is an Euler-graded $A_\infty$-algebra.
\end{nlemma}
\begin{remark}
	The $A_\infty$-algebra $\MC^s$ is also used in \cite[section~3.1]{CLT2018} to define the connection in the $u$ direction.
\end{remark}
\begin{remark}
	The deformation $\MC^s$ is canonical in the following sense: an $A_\infty$-morphism $F: \MC \rightarrow \mathcal{D}$ induces an $A_\infty$-morphism $F^s: \MC^s \rightarrow \mathcal{D}^s$ given by \begin{equation}
		F^s_k(\alpha_1, \dots, \alpha_k) = s^{1-k}F_k(\alpha_1, \dots, \alpha_k).
	\end{equation}
This morphism is Euler-graded. 
\end{remark}

$\MC^s$ is Euler-graded, so naturally comes with a $u$-connection $\nabla^s_{\frac{\partial}{\partial u}} = \frac{1}{2u}Gr - \frac{1}{u}\nabla^{GGM}_E$ (see Definition \ref{TE-structure on Euler-graded T-structure}). Define the $u$-connection on $HC^-(\MC)$ to be the restriction to $s=1$ of $\nabla^s_{\frac{\partial}{\partial u}}$. One can check that indeed: \begin{equation}
	\label{u connection cyclic homology}
	\nabla_{\frac{\partial}{\partial u}} = \frac{\partial}{\partial u} + \frac{\Gamma}{2u} + \frac{i\lbrace \mathfrak{m}' \rbrace}{2u^2},
\end{equation} 
where $\mathfrak{m}' = \sum_k (2-k)\mathfrak{m}_k$. Call this the \emph{canonical} u-connection associated to an $A_\infty$-algebra. This makes $HC^-_*(\MC)$ into a pre-TE-structure. 

\begin{remark}
	\label{Euler-graded deformation of cat}
	In the deformation $\MC^s$, $s$ has odd degree. We can also define the $R[e,e^{-1}]$-linear $A_\infty$-algebra $\mathcal{C}^e := \MC \otimes_R R[e,e^{-1}]$, where $e$ is of even degree. The operations are defined by: \begin{equation}
		\m_k^e(a_1,\dots, a_k) = e^\frac{2 - k - |\m_k(a_1, \dots, a_k)| + \sum_i |a_i|}{2} \m_k(a_1,\dots,a_k).
	\end{equation}
	Here $|a|$ is $0$ if $a$ has even degree or $1$ if $a$ has odd degree. Note that we can divide by $2$ because $\m$ is $\mathbb{Z}/2$-graded. This is Euler-graded with $E = e\partial_e$, grading operator $Gr(e^ka) = (2k+|a|)e^ka$.
\end{remark}
\begin{nlemma}
	\label{Euler-graded connection agrees with u-connection}
	Let $\MC$ be an Euler-graded $A_\infty$-algebra over $R$ with grading $Gr$ and Euler-vector field $E$. Then the canonical u-connection agrees up to homotopy with the u-connection coming from the Euler-grading.
\end{nlemma}
\begin{proof}
	Choose an $R$-basis for all morphism spaces. This defines an operator \begin{equation}
		deg := Gr - 2E:hom(\MC,\MC) \rightarrow hom(\MC,\MC),
	\end{equation}
	Let $\nabla_{\frac{d}{du}}$ denote the canonical u-connection. The u-connection defined using the Euler grading is given by \begin{equation}
		\widetilde{\nabla}_{\frac{d}{du}} = \frac{Gr^{-}}{2u} - \nabla^{GGM}_E.	
	\end{equation}
	Using the definition of $deg$ we can rewrite this as: \begin{equation}
		\widetilde{\nabla}_{\frac{d}{du}} = \frac{d}{du} + \frac{\Gamma+\mathcal{L}_{deg}}{2u} + \frac{i\{ E(\mathfrak{m}) \}}{u^2}.
	\end{equation}
	The properties of $E$ and $Gr$ show: \begin{equation}
		[\m,deg] = 2E(\m)-\m'.
	\end{equation}
	By the Cartan homotopy formula \ref{Cartan Homotopy}, we thus have: \begin{equation}
		\widetilde{\nabla}_{\frac{d}{du}}  = \nabla_{\frac{d}{du}} + u^{-2}[i\{deg\},b+uB].
	\end{equation}
	\end{proof}
In particular, if we define a u-connection $\widetilde{\nabla}_{\frac{d}{du}}$ on $HC^-(\MC)$ by restricting the connection $\nabla^e$ coming from the Euler-grading on $\MC^e$ to $e = 1$, then $\widetilde{\nabla}_{\frac{d}{du}}$ agrees with the canonical u-connection.
\begin{nlemma}
	\label{u connection is A infty invariant}
	Let $F: \mathcal{C} \rightarrow \mathcal{D}$ be an $R$-linear $A_\infty$-morphism. Then $F_*: HC_*^-(\MC) \rightarrow HC^-_*(\mathcal{D})$ is a morphism of pre-TE-structures.
\end{nlemma}
\begin{proof}
	Let $F^s: \mathcal{C}^s \rightarrow \mathcal{D}^s$ be the induced Euler-graded morphism. Now apply Theorem \ref{Sheridan GGM invariant up to homotopy} to $F^s$ to find that it respects $\nabla^{GGM}_E$ up to homotopy. As we also have that $[F^s,Gr] = 0$, we find that $F^s$ respects $\nabla_{\partial_u}$ up to homotopy. Restriction to $s=1$ shows the result.
\end{proof}

We finish this section with a comparison between the E-structures associated to a weakly curved $A_\infty$-algebra and its uncurved associated $A_\infty$-algebra. We use this to conclude that the eigenvalue decomposition of the negative cyclic homology is trivial. For simplicity, here we assume $A$ is a $\mathbb{C}$-linear $A_\infty$-algebra. Suppose that $(A,\mathfrak{m})$ is strictly unital and \emph{weakly} curved, i.e.\ $\mathfrak{m}_0 = w \cdot \mathbf{e}$ for some $w \in \mathbb{C}$. From $(A,\mathfrak{m})$ we can then obtain an uncurved $A_\infty$-algebra by setting $\overline{\mathfrak{m}}_k = \mathfrak{m}_k$ for $k\geq 1$, and $\overline{\mathfrak{m}}_0 = 0$. 

\begin{nlemma}
	\label{curved vs uncurved}
	$(A,\overline{\mathfrak{m}})$ is a unital, uncurved $A_\infty$-algebra with: \begin{equation}
	HH_*(A,\mathfrak{m}) \cong HH_*(A,\overline{\mathfrak{m}}),\; HH^*(A,\mathfrak{m}) \cong HH^*(A,\overline{\mathfrak{m}}) \text{ and } HC^{-}_*(A,\mathfrak{m}) \cong HC^{-}_*(A,\overline{\mathfrak{m}}).
	\end{equation}
	Furthermore, there exist an isomorphism of pre-E-structures: \begin{equation}
	(HC^-_*(A,\mathfrak{m}),\nabla^{\mathfrak{m}}) \cong (HC^-_*(A,\overline{\mathfrak{m}}),\nabla^{\overline{\mathfrak{m}}}) \otimes \mathcal{E}^{-\frac{w}{u}}.
	\end{equation}
	Here on both sides the connection $\nabla$ denotes the canonical connection defined above.
\end{nlemma}
\begin{proof}
	The Hochschild differentials satisfy $b = \overline{b}$ as we are working with reduced chains. $B = \overline{B}$ by definition. $\mathfrak{m}' = \overline{\mathfrak{m}}' + 2w\cdot\mathbf{e}$, and then from the fact that $\mathbf{e}$ is a strict unit, we get that $b^{1,1}(\mathfrak{m}',\_) = \overline{b}^{1,1}(\overline{\mathfrak{m}},\_) + 2w\cdot Id$. Furthermore $B^{1,1}(\mathfrak{m}',\_) = \overline{B}^{1,1}(\overline{\mathfrak{m}}',\_)$ by definition of the reduced chains. The result then follows.
\end{proof}
The residue of the connection $\nabla_{\frac{d}{du}}$ is the map $b^{1,1}(\mathfrak{m}', \_): HC^-(A)/uHC^-(A) = HH_*(A) \rightarrow HH_*(A)$. The following lemma shows that this decomposition is trivial, with the only eigenvalue being $w$. See also \cite[Lemma~2.4]{RiSmi} and \cite[Section~2.2.7]{KKP}.
\begin{nlemma}
	Let $(A,\mathfrak{m})$ be a weakly curved $A_\infty$-algebra with finite dimensional Hochschild homology. Then the operator $b^{1,1}(\mathfrak{m}',\_) - w\cdot Id: HH_*(A) \rightarrow HH_*(A)$ is nilpotent.
\end{nlemma}
\begin{proof}
	Lemma \ref{curved vs uncurved} allows us to reduce this to the uncurved case. Then, on the chain level, $b^{1,1}(\mathfrak{m}',\_): CC_*(A) \rightarrow CC_*(A)$ reduces the length of the chain by at least 1. Take a basis for $HH_*(A)$ and pick representatives $\alpha_i \in CC_*(A)$. Let $N = \max_i length(\alpha_i)$, then $(b^{1,1}(\mathfrak{m}',\_))^{N+1} = 0 $ on $HH_*(A)$.
\end{proof}
\begin{ncor}
	\label{single summand decomposition}
	Let $A$ be a $\CC$-linear, strictly unital and weakly curved $A_\infty$-algebra with curvature $w\cdot \mathbf{e}$. Assume the Hodge-de Rham spectral sequence of $A$ degenerates. Then, in the eigenvalue decomposition of Theorem \ref{decomposition}, the E-structure $HC^-_*(A)$ has just a single summand associated to the eigenvalue $w$.
\end{ncor}
For future reference, we rephrase a result by Amorim and Tu, \cite[Corollary~3.8]{AT} in the language of TEP-structures:
\begin{nthm}
	\label{semi simple cyclic homology}
	If $A$ is an $n$-dimensional, strictly unital, cyclic, $\mathbb{Z}/2$-graded, smooth and finite dimensional $A_\infty$-algebra with $HH^*(A)$ semi-simple, then $HC_*^-(A)$ is a semi-simple TEP-structure.
\end{nthm}

\section{Cyclic open-closed map respects connections}
\label{Cyclic open-closed map respects connections}
\subsection{Coefficient rings}
\label{coefficient rings}
Consider the Novikov ring \begin{align}
	\Lambda &= \left\{  \sum_{i = 0}^{\infty} a_iQ^{\lambda_i} | a_i \in \mathbb{C}, \; \lambda_i \in \mathbb{R}_{\geq 0},\; \lim_{i\to \infty } \lambda_i = \infty \right\}.
\end{align}
Let $Q$ have degree $0$. For a $\ZZ$-graded $\CC$-vector space $U$, let $\CC[[U]]$ be the ring of formal functions on the completion of $U$ at the origin. Explicitly, let $\{v_i\}_{i \in I}$ be a homogeneous basis for $U$, and $\{v_i^*\}_{i \in I}$ the dual basis for $U^*$. Let $\{t_i\}_{i \in I}$ be formal variables of degree $-|v_i|$, then we have an isomorphism: \begin{align}
	\label{isomorphism coefficient ring}
	\CC[[t_i]]_{i \in I} &\cong \CC[[U]],\notag \\
	t_i &\mapsto v_i^*.
\end{align}
Each formal vector field $v \in \CC[[U]] \otimes U$ on $U$ can be viewed as a derivation $\partial_v: \CC[[U]] \rightarrow \CC[[U]]$. In coordinates, if $v = \sum_i f_i v_i$, for some $f_i \in \CC[[U]]$, then $\partial_v = \sum_i f_i \partial_{t_i}$. Define the vector fields \begin{equation}
	\Gamma_U = \sum_i t_i \partial_{t_i} \text{ and } E_U = \sum_i \frac{deg(t_i)}{2}t_i\partial_{t_i}.
\end{equation}
These are independent of the chosen basis.

 For $l \in \mathbb{Z}$, let $U[l]$ denote the graded vector space with $U[l]^i = U^{i+l}$. Then set:
\begin{equation}
	Q_U :=  \Lambda[[U[2]]].
\end{equation}

Following \cite{ST3}, define the valuation $\zeta_Q: Q_U \rightarrow \RR_{\geq 0}$ by: \begin{equation}
	\label{valuation on coefficient ring}
	\zeta_Q\left(\sum_{j = 0}^{\infty} a_jQ^{\lambda_j} \prod_{i \in I} t_i^{l_{ij}}\right) = \inf_{\substack{j \\ a_j\neq 0}} (\lambda_j + \sum_{i \in I} l_{ij}).
\end{equation} 
Let $\mathcal{I}_U = \{ f \in Q_U | \zeta_Q(f) > 0 \} \subset Q_U$.

To account for gradings, we will also make use of the `universal Novikov ring': $\Lambda^{e} := \Lambda[e,e^{-1}]$, where $e$ has degree $2$. Let $Q^{e}_U$ be defined using $\Lambda^e$ instead of $\Lambda$. 

\begin{remark}
	A lot of our work is based on \cite{ST3}. They use a different Novikov ring, more commonly used in Gromov-Witten theory. Instead of taking series in $Q^{\mathbb{R}}$ they take series with terms $T^{\beta}$ for $\beta \in H_2(X,L)$. For them the monomial $T^\beta$ has degree $\mu(\beta)$, where $\mu: H_2(X,L) \rightarrow \mathbb{Z}$ is the Maslov index. The graded map $T^{\beta} \mapsto Q^{\omega(\beta)}e^{\mu(\beta)/2}$, allows us to compare their Novikov ring with the universal Novikov ring $\Lambda^e$. Note that $\mu(\beta) \in 2\mathbb{Z}$ as we assume our Lagrangian is orientable.
\end{remark}
\subsection{Quantum TE-structure}
\label{Quantum u-VSHS}
Let $(X,\omega)$ be a symplectic manifold and let $U \subset H^*(X;\mathbb{C})$ be a graded $\CC$-vector subspace. For any ring $R$, let $A^*(X;R)$ denote the space of differential forms on $X$ with coefficients in $R$.

\begin{defi}
	A \emph{bulk-deformation parameter} over $U$ is an element $\gamma \in A^*(X;\mathcal{I}_U)$ with $d\gamma = 0$, $|\gamma| = 2$ and $[\gamma] = \Gamma_U \in Q_U \otimes U$.
\end{defi}
In section \ref{outline of general proof} we will outline how a general proof that $\OC^-$ respects T-structures can be upgraded to a result about TE-structures. In order for this to work, the T-structure needs to know about deformation in the direction of $c_1(X)$, the first Chern class of $X$. To that end we make the following assumption:
\begin{assumption}
	\label{bulk-deformation assumption}
	We assume there exists a $\YY \in Der_{\Lambda^e} Q^e_U$ be such that $[\YY(\gamma)] = c_1 \in H^*(X;Q_U^e)$. In this case, note that as $|\ga| = 2$, this implies $|\YY| = 0$.
\end{assumption}

Let $\gamma$ be a bulk-deformation parameter over $U$. We now consider the quantum cohomology $QH^*(X;Q^e_U)$. As a vector space this is just $H^*(X;Q^e_U)$, but the product is given by the bulk-deformed quantum cup product $\eta_1 \star_\gamma \eta_2$. A general reference for the construction of the quantum cup product is \cite{MS12}, however, our coefficient ring includes the universal Novikov parameter $e$, so we sketch how to modify the definition. See also Definition \ref{Quantum cup product} for a construction in our specific setup. Recall from \cite{MS} that the quantum cup product is defined as a sum over curve classes $\beta \in H_2(X)$: \begin{equation}
	\eta_1 \star \eta_2 = \sum_\beta Q^{\omega(\beta)}(\eta_1 \star \eta_2)_{\beta}.
\end{equation}
Here $(\eta_1 \star \eta_2)_{\beta}$ is defined by the equation \begin{equation}
	\int_X \eta \cup (\eta_1 \star \eta_2)_{\beta} = GW^\beta_{0,3}(\eta,\eta_1,\eta_2) \text{ for all } \eta \in H^*(X),
\end{equation}
where $GW^\beta_{0,3}$ denotes the genus $0$, $3$ point Gromov-Witten invariant in curve class $\beta$. One can then extend this definition to take into account bulk deformations $\ga$, to obtain the product $\star_\ga$ on quantum cohomology $QH^*(X;Q_U)$. We then define (see also Definition \ref{Quantum cup product}) the product on $QH^*(X;Q_U^e)$ by:
\begin{equation}
	\label{product on Euler graded qcoh}
	\eta_1 \star_\gamma \eta_2 = \sum_\beta Q^{\omega(\beta)}e^{c_1(\beta)}(\eta_1 \star_\gamma \eta_2)_{\beta},
\end{equation}
where $c_1 = c_1(TX)$ is the first Chern class. 
\begin{defi}
	The quantum T-structure over $Q^e_U \supset \Lambda^e$ is given as a $Q_U^e[[u]]$-module by:\begin{equation}
		QH^*(X;Q^e_U)[[u]].
	\end{equation}
	For $v \in Der_{\Lambda}Q_U$ the quantum connection is defined by: \begin{equation}
		\label{connection in bulk direction quantum coh}
		\nabla_v \eta = v(\eta) - u^{-1}v(\gamma) \star_\gamma \eta.
	\end{equation}
\end{defi}
We now wish to extend the quantum T-structure to be defined over $Q^e_U \supset \Lambda$. To this end observe that $Der_\Lambda (Q_U^e) = Q_U^e \otimes_{Q_U} Der_\Lambda Q_U \oplus Q^e_U\langle \partial_e \rangle$. Extend the connection by setting: 
	\begin{equation}
		\label{e connection on qcoh}
		\nabla_{e\partial_e} \eta = e\partial_e (\eta) - u^{-1} c_1 \star_\gamma \eta.
	\end{equation}
\begin{nlemma}
	These above definitions make $QH^*(X;Q_U^e)[[u]]$ into a T-structure over $Q_U^e \supset \Lambda$.
\end{nlemma}
\begin{proof}
	The verification that $[\nabla_v,\nabla_w] = \nabla_{[v,w]}$ for $v,w \in Der_{\Lambda} Q_U$ is standard, so we will not do it here. Instead we verify that $[\nabla_v,\nabla_{e\partial_e]} = 0$.
	The divisor equation for closed Gromov-Witten invariants shows that for $v \in Der_\Lambda Q_U$ we have:\begin{equation}
		e\partial_e (v(\gamma) \star_\gamma \eta) = c_1 \star_\gamma v(\gamma) \star_\gamma \eta + v(\ga) \star_\gamma e\partial_e (\eta).
	\end{equation}
	We also find that: \begin{equation}
		v(c_1 \star_\gamma \eta) = c_1 \star_\gamma v(\ga) \star_\gamma \eta + c_1 \star_\gamma v(\eta).
	\end{equation}
	A direct verification then shows that $[\nabla_{e\partial_e},\nabla_v] = 0$ holds.
\end{proof}
We will now define endow $QH^*(X;Q_U^e)[[u]]$ with an Euler-graded T-structure.
\begin{defi}
	Define the Euler vector field by $E = e\partial_e + E_U$. Define the grading operator \begin{equation}Gr^-: QH^*(X;Q_U^e)[[u]] \rightarrow QH^*(X;Q_U^e)[[u]]\end{equation} by taking into account the cohomological degrees, the grading on the coefficient rings and the degree of $u$, but with the grading shifted down by $n$ so that, for $\eta \in H^*(X;\CC)$ and $f \in Q_U^e[[u]]$, we have: \begin{equation}
		Gr^-(f\eta) = (|f| + |\eta| - n)f\eta = (2u\partial_u + 2E)(f)\al + 2f\mu(\al).
	\end{equation}
	where $\mu: H^p(X;Q_U^e) \rightarrow H^p(X;Q_U^e)$ is given by $\mu(\eta) = \frac{p-n}{2}\eta$.
\end{defi}
A short computation then shows:
\begin{nlemma}
	The above definitions make $QH^*(X;Q_U^e)[[u]]$ into an Euler-graded T-structure.
\end{nlemma}
As the quantum T-structure is Euler-graded, Definition \ref{TE-structure on Euler-graded T-structure} endows it with a connection in the $u$-direction:
\begin{equation}
	\nabla_{\partial_u} = \frac{Gr^-}{2u} - u^{-1}\nabla_E.
\end{equation}
This makes $QH^*(X;Q_U^e)[[u]]$ a TE-structure. Writing out the definitions of $Gr^-$ and $E$ yields the formula:
\begin{equation}
	\label{formula for u connection qcoh}
	\nabla_{\partial_u} \eta = \partial_u(\eta) + u^{-1}\mu(\eta) + u^{-2}(c_1 + E_U(\gamma))\star_\gamma \eta.
\end{equation}
\begin{remark}
	Defining $\nabla_{Q\partial_Q}\eta := Q\partial_Q (\eta) - u^{-1} [\omega] \star_\gamma \eta$ extends the connection to include the $Q$-direction, to obtain a TE-structure over $Q_U^e \supset \CC$.  We do not use this connection.
\end{remark}
\begin{remark}
	\label{comparison of quantum T-structures}
	We can use the same formula \ref{formula for u connection qcoh} to define a TE-structure on $QH^*(X;Q_U)[[u]]$. There then is a natural isomorphism of TE-structures over $Q_U \supset \Lambda$: \begin{equation}
		QH^*(X;Q_U^e)[[u]] \otimes_{\Lambda^e} \Lambda \cong QH^*(X;Q_U)[[u]].
	\end{equation}
	Here $\Lambda$ is considered as a $\Lambda^e$ module via the homomorphism $\Lambda^e \rightarrow \Lambda$ given by evaluation at $e=1$.
\end{remark}
There are alternative definitions of the quantum connections on $QH^*(X;Q_U^e)[[u]]$ given by changing the signs:
\begin{defi}
\begin{align}
	\label{nabla * on qcoh}
	\nabla^*_v \eta &= v(\eta) + u^{-1}v(\gamma) \star_\gamma \eta\\
	\nabla^*_{e\partial_e} \eta &= e\partial_e (\eta) + u^{-1}c_1 \star_\gamma \eta
\end{align}
The alternative connection in the $u$-direction is given by $\nabla^*_{\partial_u} = \frac{Gr^-}{2u} -u^{-1}\nabla^*_E$. Writing the formulae out we find: \begin{equation}
	\nabla^*_{\partial_u}\eta = \partial_u(\eta) + u^{-1} \mu(\eta) - u^{-2} (c_1 + E_U(\ga)) \star_\gamma \eta.
\end{equation}
\end{defi}
Define the Poincar\'e pairing \begin{align}
	\langle \cdot , \cdot \rangle_X: QH^*(X;Q_U^e) \otimes QH^*(X;Q_U^e) \rightarrow Q_U^e,
\end{align}
by $\langle \eta_1, \eta_2 \rangle_X = \int_X \eta_1 \wedge \eta_2$.
Now extend the Poincar\'e pairing \emph{u-linearly} to a pairing \begin{equation}
	\langle \cdot, \cdot \rangle_X: QH^*(X;Q_U^e)[[u]]\otimes QH^*(X;Q_U^e)[[u]] \rightarrow Q_U^e[[u]].
\end{equation} We then have:
\begin{equation}
	\label{poincare pairing GGM connection}
	\langle \nabla^*_{v} \eta_1, \eta_2 \rangle_X + (-1)^{|\eta_1||v|} \langle  \eta_1, \nabla_{v}\eta_2 \rangle_X = v (\langle \eta_1, \eta_2 \rangle_X),
\end{equation}
for all $v \in Der_\Lambda Q_U^e[[u]]$.
\begin{remark}
	It is customary to extend the Poincar\'e pairing sesquilinearly to the quantum TE-structure, to obtain a TEP-structure where the polarisation can be matched up with the higher residue pairing on cyclic homology. However, since we don't mention the polarisation in this paper, we use the $u$-linear extension as it simplifies the proof that the cyclic open-closed map is a morphism of TE-structures.
\end{remark}

\subsection{Outline of proof of Theorem \ref{cyclic open closed theorem in intro}}
\label{outline of general proof}
In this section we give an outline of the proof of Theorem \ref{cyclic open closed theorem in intro}. We will state sufficient conditions which imply that the cyclic open-closed map respects the connection $\nabla_{\partial_u}$. We state these conditions in such a way that they should be easy to generalise to different geometric setups.

Let $U \subset H^*(X;\mathbb{C})$ be a graded vector space, and $\ga$ a bulk-deformation parameter over $U$ satisfying Assumption \ref{bulk-deformation assumption}. Let $L \subset X$ be a Lagrangian submanifold. We define an Euler-graded $A_\infty$-algebra $CF^*(L,L)[e]$ over $Q^e := Q_U[e,e^{-1}]$ in Section \ref{Fukaya algebra}. The Euler vector field is given by $E = e\partial_e + E_U$, where $E_U$ is as in Section \ref{coefficient rings} and $e$ is of degree $2$. The Floer cochain complex $CF^*(L,L)$ is then defined by restricting to $e=1$: $CF^*(L,L) := CF^*(L,L)[e] \otimes_{Q_U^e} Q_U$.

More generally, suppose there exists a bulk-deformed Fukaya category $Fuk^t(X)$ defined over $Q_U$. By using $e$ to take into account the Maslov index of holomorphic disks, it should be possible to construct an Euler-graded Fukaya category $Fuk^t(X)[e]$ over $Q^e$. In Appendix \ref{Euler-grading on Fukaya category}, we construct such an Euler-graded Fukaya category geometrically. In the appendix, $U$ will be the 1-dimensional vector-space spanned by the first Chern class.

Assume there exists a cyclic open-closed map 
\begin{equation}
	\OC^-: HC^-_*(Fuk^t(X)) \rightarrow QH^*(X;Q_U)[[u]],
\end{equation}
 which is the restriction to $e=1$ of a map 
 \begin{equation} 
 	\OC_e^-: HC^-_*(Fuk^t(X)[e]) \rightarrow QH^*(X;Q_U^e)[[u]]. 
 \end{equation} 
In Section \ref{cyclic open-closed map}, we will construct a cyclic open-closed map by defining a chain level pairing (which we call the cyclic open-closed pairing): \begin{equation}
	\langle \cdot, \OC_e^-(\cdot) \rangle:  (C^*(X;Q_U^e)[[u]]) \otimes CC_*^-(Fuk^t(X)[e]) \rightarrow Q_U^e[[u]].
\end{equation}
We show that it satisfies: \begin{equation}\langle d\eta, \OC_e^-(\alpha) \rangle + (-1)^{|\eta|} \langle \eta, \OC_e^-((b+uB)(\alpha)) \rangle = 0.
\end{equation}
This ensures the chain level pairing descends to a pairing: \begin{equation}
	\langle \cdot, \OC_e^-(\cdot) \rangle: QH^*(X;Q_U^e)[[u]] \otimes HC_*^-(Fuk^t(X)[e]) \rightarrow Q_U^e[[u]].
\end{equation}
We then apply Poincar\'e duality to the $QH^*(X;Q^e)$ factor to obtain the map $\OC_e^-$. It is uniquely determined by the property: 
\begin{equation}
	\langle \eta,\OC_e^-(\al) \rangle_X = \langle \eta, \OC_e^-(\alpha) \rangle.
\end{equation}
On the left, the pairing is the Poincar\'e pairing on $X$, and on the right the pairing is the open-closed pairing.

We expect more generally that given a construction of a chain level cyclic open-closed map, it can be extended to $CC_*^-(Fuk^t(X)[e])$ by taking the Maslov index of the holomorphic disks into account. 

Suppose the cyclic open-closed pairing satisfies the following properties.
\begin{assumptions}
	\label{OC assumptions}
	$\text{ }$
	\begin{enumerate}
		\item For $v \in Der_{\Lambda^e} Q_U^e$, there exists a pairing $\langle \_, G_{v}(\_) \rangle : (C^*(X;Q_U^e)[[u]]) \otimes CC_*^-(Fuk^t(X)[e]) \rightarrow Q_U^e[[u]]$ such that for all $\eta$ and $\al$ we have: 	\begin{multline}
			\langle \nabla^*_{v} \eta, \OC_e^- (\alpha) \rangle + (-1)^{|\eta||v|}\langle \eta, \OC_e^- (\nabla_{v} \alpha) \rangle = v \left(\langle \eta, \OC_e^-(\alpha)\rangle\right) \\+ u^{-1}\left(\langle d\eta, G_{v} (\alpha) \rangle + (-1)^{|\eta| + |v|}\langle \eta, G_{v}\left((b+uB)(\alpha)\right)\rangle \right).
		\end{multline}
		Here $\nabla^*$ is the sign-changed connection from \eqref{nabla * on qcoh}. By \eqref{poincare pairing GGM connection}, this implies that $\OC_e^-$ is a morphism of T-structures over $Q_U^e \supset \Lambda^e$.
		\label{OC assumptions 1}
		\item There exists a $\phi \in CC^1(Fuk^t(X)[e])$, such that $\YY(\mathfrak{m}^e) = e\partial_e(\m^e) + [\m^e, \phi]$.
		\label{OC assumptions 2}
		\item For any $\eta$ and $\al$ we have: \begin{equation}
			\langle \eta, \YY(\OC_e^-)(\alpha) \rangle = \langle \eta, e\partial_e(\OC_e^-)(\al) \rangle + \langle \al, \OC_e^-(\mathcal{L}_\phi(\eta)) \rangle.
		\end{equation}
		Here, for any $v \in Der_{\Lambda} Q_U^e$, we define: \begin{equation}
			\label{eq: definition of deriviative of open closed map}
			\langle \eta, v(\OC_e^-)(\al) \rangle := v(\langle \eta, \OC_e^-(\alpha) \rangle) - \langle v(\eta), \OC_e^-(\alpha) \rangle - (-1)^{|\eta||v|} \langle \eta, \OC_e^-(v(\alpha)) \rangle.
		\end{equation} \label{OC assumptions 3}
		\item $\OC_e^-$ respects the Euler-grading on cyclic invariants: $Gr^{-} \circ \OC_e^- = \OC_e^- \circ Gr^{-}$.
		\label{OC assumptions 4}
	\end{enumerate} 
\end{assumptions}
\begin{remark}
	Since we construct a cyclic open-closed map via a cyclic open-closed pairing, we have stated the required properties for the latter. An easy modification of the assumptions would allow them to be applied to an open-closed map defined directly as a chain map $CC_*^-(Fuk^t(X)[e]) \rightarrow C^*(X;Q_U^e)[[u]]$.
\end{remark}
\begin{remark}
	To define the connections $\nabla_v$ and $\nabla_{e\partial_e}$ on cyclic homology, we need to choose a basis for all of the morphism spaces in $Fuk^t(X)$. Assumption \ref{OC assumptions 2} is required to hold with respect to the same bases as used to define the connections. The same holds for assumption \ref{OC assumptions 3}. On quantum cohomology, we take the derivatives with respect to the standard constant basis (i.e.\ one in $H^*(X;\CC)$).
\end{remark}
\begin{nthm}
	\label{assumptions imply TE isom}
	Suppose assumptions \ref{OC assumptions} hold, then $\mathcal{OC}_e^-$, and hence $\OC^-$, respects $\nabla_{\partial_u}$ on homology.
\end{nthm}
First we will show that Assumptions \ref{OC assumptions 1}, \ref{OC assumptions 2} and \ref{OC assumptions 3} show that $\OC^e$ respects the connection $\nabla_{e\partial_e}$. To this end define:
\begin{defi}
	\label{defi: G e}
	Define the pairing $\langle \eta, G_{e\partial_e}(\al) \rangle$ by $$\langle \eta, G_{e\partial_e}(\al) \rangle = \langle \eta, G_{\YY}(\al)\rangle + (-1)^{{|\eta|}}\langle \eta, \OC_e^-(i\{\phi\}(\al)) \rangle$$
\end{defi}
\begin{nprop}
	The pairing $\langle \eta, G_{e\partial_e}(\al) \rangle$ satisfies \begin{multline*}
		\langle \nabla^*_{e\partial_e} \eta, \OC_e^- (\alpha) \rangle + \langle \eta, \OC_e^- (\nabla_{e\partial_e} \alpha) \rangle = e\partial_e \left(\langle \eta, \OC_e^-(\alpha)\rangle\right) \\+ u^{-1}\left(\langle d\eta, G_{e\partial_e} (\alpha) \rangle + (-1)^{|\eta|}\langle \eta, G_{e\partial_e}\left((b+uB)(\alpha)\right)\rangle \right).
	\end{multline*}
\end{nprop}
\begin{proof}
	Applying Assumption \ref{OC assumptions 1}, with $v = \YY$ yields: \begin{align}
		&\begin{multlined}[t]u^{-1} \langle c_1 \star \eta, \OC_e^-(\alpha) \rangle - u^{-1}\langle \eta, \OC_e^-(i\{\YY(\m^e)\}\alpha) \rangle = \langle \eta, \YY(\OC_e^-)(\alpha) \rangle + u^{-1}\langle d\eta, G_{\YY} (\alpha) \rangle \\ + u^{-1}(-1)^{|\eta|}\langle \eta, G_{\YY}\left((b+uB)(\alpha)\right)\rangle.\end{multlined}\\
		&\qquad \text{using Assumption \ref{OC assumptions 2}, this gives:}\notag\\
		&\begin{multlined}[t] u^{-1} \langle c_1 \star \eta, \OC_e^-(\alpha) \rangle - u^{-1}\langle \eta, \OC_e^-(i\{e\partial_e(\m^e)\}\alpha) \rangle -u^{-1}\langle \eta, \OC_e^-(i\{[\m^e,\phi]\}\alpha) \rangle =\\ \langle \eta, \YY(\OC_e^-)(\alpha) \rangle + u^{-1}\langle d\eta, G_{\YY} (\alpha) \rangle + u^{-1}(-1)^{|\eta|}\langle \eta, G_{\YY}\left((b+uB)(\alpha)\right)\rangle.\end{multlined}\\
		&\qquad \text{By the Cartan homotopy formula (Proposition \ref{Cartan Homotopy}) we can rewrite this as:}\notag\\
		&\begin{multlined}[t] u^{-1} \langle c_1 \star \eta, \OC_e^-(\alpha) \rangle - u^{-1}\langle \eta, \OC_e^-(i\{e\partial_e(\m^e)\}\alpha) \rangle = \langle \eta, \YY(\OC_e^-)(\alpha) \rangle - \langle \eta, \OC_e^-(\mathcal{L}_{\phi}\alpha) \rangle \\+ u^{-1}\langle d\eta, G_{e\partial_e} (\alpha) \rangle + u^{-1}(-1)^{|\eta|}\langle \eta, G_{e\partial_e}\left((b+uB)(\alpha)\right)\rangle.\end{multlined}\\
		&\qquad \text{Then apply Assumption \ref{OC assumptions 3} to obtain:}\notag\\
		&\begin{multlined}[t] u^{-1} \langle c_1 \star \eta, \OC_e^-(\alpha) \rangle - u^{-1}\langle \eta, \OC_e^-(i\{e\partial_e(\m^e)\}\alpha) \rangle = \langle \eta, e\partial_e(\OC_e^-)(\alpha) \rangle + u^{-1}\langle d\eta, G_{e\partial_e} (\alpha) \rangle \\ + u^{-1}(-1)^{|\eta|}\langle \eta, G_{e\partial_e}\left((b+uB)(\alpha)\right)\rangle.\end{multlined}
	\end{align}
	This implies the result, by rewriting $\langle \eta, e\partial_e(\OC_e^-)(\alpha) \rangle$ using \eqref{eq: definition of deriviative of open closed map}.
\end{proof}
\begin{proof}[Proof of Theorem \ref{assumptions imply TE isom}]
 	The previous proposition implies $\mathcal{OC}^e_-$ respects $\nabla_{e\partial_e}$ on homology. We will spell this out.
 	By the properties of the Poincar\'e pairing (Equation \eqref{poincare pairing GGM connection}):
 	\begin{equation}
 		\langle \eta,\nabla_{e\partial_e}(\OC_e^-(\al)) \rangle_X = e\partial_e(\langle \eta, \OC_e^-(\al) \rangle) - \langle \nabla^*_{e\partial_e} \eta, \OC_e^-(\al) \rangle.
 	\end{equation}
 	The above shows that on homology: \begin{equation}
 		\langle \eta,\nabla_{e\partial_e}(\OC_e^-(\al)) \rangle_X = \langle \eta, \OC_e^-(\nabla_{e\partial_e}(\al)) \rangle,
 	\end{equation}
 	which shows that $\OC_e^-(\nabla_{e\partial_e}(\al)) = \nabla_{e\partial_e}(\OC_e^-(\al))$ on homology. As $E = e\partial_e + E_U$, we combine this with Assumption \ref{OC assumptions 1}, applied to $v = E_U$ to find that the open-closed map respects $\nabla_{E}$. Then, as $\nabla_{\partial_u} = \frac{Gr^{-}}{2u} - u^{-1}\nabla_E$, and using the fact that $\OC_e^-$ respects the Euler-grading (Assumption \ref{OC assumptions 4}), we find that $\OC_e^-$ respects the connection $\nabla_{\partial_u}$. The statement about $\OC^-$ follows by restriction to $e = 1$.
\end{proof}

\subsection{Regularity assumptions}
Let $X$ be a $2n$-dimensional symplectic manifold and $J$ be an $\omega$-tame almost complex structure on $X$. Let $L \subset X$ be an oriented Lagrangian equipped with a $U(\Lambda)$-local system and a relative spin structure $\mathfrak{s}$. For us a relative spin structure comes with a choice of element $w_{\mathfrak{s}} \in H^2(X;\ZZ/2)$ such that $w_{\mathfrak{s}}|_L = w_2(TL) \in H^2(L;\ZZ/2)$.

For $l \geq 0$, let $\mathcal{M}_{l+1}(\beta)$ be the moduli space of stable $J$-holomorphic spheres with $l+1$ marked points in homology class $\beta \in H_2(X,\mathbb{Z})$. Let \begin{equation}
ev_j^\beta: \mathcal{M}_{l+1}(\beta) \rightarrow X
\end{equation}
be the evaluation map at the $j$'th marked point.
For $k\geq -1$, $l\geq 0$, let $\mathcal{M}_{k+1,l}(\beta)$ be the moduli space of $J$-holomorphic stable maps $(\mathbb{D},S^1) \rightarrow (X,L)$ in homology class $\beta \in H_2(X,L)$ with one boundary component, $k+1$ anti-clockwise ordered boundary marked points, and $l$ interior marked points. Let \begin{equation}
evb_i^\beta: \mathcal{M}_{k+1,l}(\beta) \rightarrow L \text{ and } evi_j^\beta: \mathcal{M}_{k+1,l}(\beta) \rightarrow X
\end{equation}
be the evaluation maps at the $i$'th boundary and $j$'th interior marked points respectively. The relative spin structure determines an orientation on the moduli spaces $\mathcal{M}_{k+1,l}(\beta)$, see \cite[Chapter~8]{FOOO}.

We will also need a moduli space of disks with a horocyclic constraint. Recall that a horocycle in a disk is given by a circle tangent to the boundary. These moduli spaces are similar to the ones used in \cite[Chapter~3]{ST2}, where some of the marked points are constrained to lie on a geodesic in $\mathbb{D}$. Our definition is entirely analogous, except that we replace `geodesic' with `horocycle'. Let the smooth locus of $\mathcal{M}_{k+1,l; \perp_i}(\beta) \subset \mathcal{M}_{k+1,l}(\beta)$ be the subset defined by requiring the first and second interior marked points $w_1$ and $w_2$ to lie at $-t$ and $t$ respectively for $t \in (-1,1)$ and fixing the i'th boundary point $z_i$ at $-i$. Equivalently, we require that $z_i$, $w_1$, $w_2$ lie on a horocycle in anti-clockwise ordering. This moduli space also appeared in \cite{Ga12}, where it was used to show that the closed-open map is an algebra homomorphism.

We now give a more formal definition of the moduli space $\mathcal{M}_{k+1,l; \perp_i}(\beta)$ as a fibre product of known spaces. Consider the forgetful map $\mathcal{M}_{k+1,l}(\beta) \rightarrow \MM_{1,2} = D^2$, only remembering the i'th boundary marked point, and the first interior marked point. Here the identification $\MM_{1,2} \cong D^2$ is achieved by using an automorphism of the disk to map the boundary marked point to $-i$, and the interior marked point to $0$. Consider the inclusion $I \hookrightarrow D^2$ given by the arc of the horocycle through $-i$ and $0$ with negative real part. This is a circle of radius $\frac{1}{2}$ centred at $-\frac{i}{2}$. The condition on the order of the marked points means that second interior lies on the semi-circle with negative real part. We then define: \begin{equation}
	\mathcal{M}_{k+1,l; \perp_i}(\beta) = I \times_{D^2} \mathcal{M}_{k+1,l}(\beta).
\end{equation}

Take the orientation on $I$ to be the positive orientation, so that $\partial I = \{1\} -\{ 0\}$. The orientation on $\mathcal{M}_{k+1,l; \perp_i}(\beta)$ is then defined by the fibre-product orientation, as in \cite[Section 2.2]{ST4}.

We assume the following:
\begin{assumptions}
	$\text{ }$
	\label{assumptions}
	\begin{enumerate}
		\item $\mathcal{M}_{l+1}(\beta)$ is a smooth orbifold with corners.
		\label{assumptions 1}
		\item $ev_0$ is a proper submersion.
		\label{assumptions 2}
		\item $\mathcal{M}_{k+1,l}(\beta)$ is a smooth orbifold with corners.
		\label{assumptions 3}
		\item $evb_0^\beta$ is a proper submersion.
		\label{assumptions 4}
		\item $\mathcal{M}_{k+1,l; \perp_i}(X,\beta)$ is a smooth orbifold with corners.
		\label{assumptions 5}
		\item $evb_0^\beta|_{\mathcal{M}_{k+1,l; \perp_i}(X,\beta)}$ is a proper submersion.
		\label{assumptions 6}
	\end{enumerate}
\end{assumptions}
We will now show these assumptions hold in the following setup:

\begin{nlemma}
	\label{G action assumptions hold}
	The above assumptions hold for $L \subset X$ a Lagrangian and a complex structure $J$ with the following properties:
	\begin{itemize}
		\item $J$ is integrable.
		\item There exists a Lie group $G_X$ acting J-holomorphically and transitively on $X$.
		\item There exist a Lie subgroup $G_L \subset G_X$ whose action restricts to a transitive action on $L$.
		\end{itemize}
\end{nlemma}
\begin{proof}
	This argument is the same as \cite[Section~1.3.12]{ST2}, but for horocyclic rather than geodesic constraints. For assumptions \ref{assumptions}.\ref{assumptions 1} and \ref{assumptions}.\ref{assumptions 2}, \cite[Proposition~7.4.3]{MS12} show that if the above properties hold, all stable holomorphic maps in $\mathcal{M}_{l+1}(\beta)$ are regular, it then follows from \cite{RRS} that this space is a smooth orbifold with corners. As $G_X$ acts on $\mathcal{M}_{l+1}(\beta)$, $G_X$ acts transitively on $X$, and $ev_0$ is equivariant with respect to this action, $ev_0$ is a proper submersion. 
	
	Solomon and Tukachinsky show assumptions \ref{assumptions}.\ref{assumptions 3} and \ref{assumptions}.\ref{assumptions 4} hold in this situation by adapting the arguments for closed Riemann surfaces to Riemann surfaces with boundary (see \cite[Remark~1.6]{ST3}). 
	
	Furthermore, $\mathcal{M}_{k+1,l; \perp_i}(\beta)$ is a smooth orbifold with corners as the maps $I \hookrightarrow D^2$ and $\mathcal{M}_{k+1,l}(\beta) \rightarrow D^2$ are transverse (in the sense of \cite[Definition~6.1]{Joy}). This is clear for any curve mapping to the part of the horocycle not intersecting $S^1$. The remaining case, for a point in the pre-image of $-i \in D^2$ is more involved.
	
	To this end, let $u \in \MM_{k+1,l}(\beta)$ be a pre-image of the point $-i \in D^2$. Let $z_0$ denote the boundary marked point, and $w_1, w_2$ the interior marked points partaking in the horocycle. The stable curve underlying $u$ has at least 3 components, with one `middle' disk component which has the marked point $z_0$. There is then additionally a `right' configuration which has the point $w_1$, and a `left' one which has $w_2$. Both the left and the right configurations either are disks connected to the middle disks, or have a disk component which connects them to the middle disk. Let $p_1$ and $p_2$ denote the boundary points where the right and left configurations are connected to the middle disk. Now use a gluing argument to open up the point $p_1$. By stability, the new middle disk (which might have $w_1$ on it) must now have at least $4$ boundary marked points, or $2$ boundary and $1$ interior marked point. A one parameter family can then be constructed by moving the point $z_0$ towards $p_2$. The original map $u$ arises from the case where $z_0$ merges with $p_2$ and a new disk bubble is formed. The map $u$ lies in the interior of this one parameter family (as we can move $z_0$ to either side of $p_2$), so it lies in the same stratum as the family. The forgetful map to $D^2$ is given by the location of $p_2$ in this case, so we see that the map at the level of tangent spaces is indeed surjective onto $T_{-i}S^1$ as required.
	
	Finally note that $G_L$ acts on $\mathcal{M}_{k+1,l; \perp_i}(X,\beta)$, and the evaluations maps are equivariant. As $G_L$ acts transitively on $L$, $evb_0^\beta|_{\mathcal{M}_{k+1,l; \perp_i}(X,\beta)}$ is a proper submersion. 
\end{proof}

\begin{eg}
	The simplest example is $(\mathbb{CP}^n, T_{cl})$, where $T_{cl}$ denotes the Clifford torus. Here the groups acting are $G_X = SU(n+1)$ and $G_L = T^{n}$. Other examples (see {\cite[Example~1.5]{ST3}}) are $(\mathbb{CP}^n, \mathbb{RP}^n)$, or more generally flag varieties and Grassmannians with $L$ being the real locus. In these cases $G_X$ is a complex matrix group and $G_L$ the real subgroup. Another class of examples are the quadric hypersurfaces with real locus $S^n$: \begin{equation}
		X_{2,n} = \big\{ \sum_{i=0}^n z_i^2 = z_{n+1}^2 \big\} \subset \mathbb{CP}^{n+1}.
		\end{equation}
	In this case, $G_X = SO(n+1,1, \CC)$ and $G_L = SO(n+1,1, \RR)$, with the group acting in the obvious manner.  
\end{eg}

\subsection{$\mathfrak{q}$-operations}
This section follows \cite{ST3} and \cite{ST2} closely. Let $L \subset X$ be as in the previous section. Let \begin{equation} hol: H_1(L,\mathbb{Z}) \rightarrow U(\Lambda)\end{equation} denote the monodromy representation of the local system on $L$.

Let $A^*(L)$ denote differential forms on $L$ with coefficients in $\mathbb{C}$, and similarly for $X$.
For $\al \in A^*(L)$, let $|\al|$ denote its degree as a differential form, and similarly for differential forms on $X$. Also, let $|\al|' := |\al| - 1$, and for an ordered set $\al = (\al_1, \dots, \al_k)$, write $\epsilon(\al) := \sum_i |\al_i|' \in \mathbb{Z}$.

 For $k,l \geq 0$ and $\beta \in H_2(X,L)$ with $(k,l,\beta) \notin \lbrace (1,0, \beta_0), (0,0,\beta_0) \rbrace$,\cite{ST3} define operations: \begin{equation}
	\mathfrak{q}_{k,l}^{ST,\beta}:A^*(L)^{\otimes k} \otimes A^*(X)^{\otimes l} \rightarrow A^*(L).
\end{equation}
We extend their definition to take into account the local system and set: \begin{equation}
	\mathfrak{q}_{k,l}^{ST,\beta}(\alpha_1 \otimes \dots \otimes \alpha_k; \gamma_1 \otimes \dots \otimes \gamma_l) := (-1)^{\zeta(\alpha)} hol(\partial \beta) (evb_0^\beta)_* \bigg{(} \bigwedge_{j=1}^l (evi^\beta_j)^* \gamma_j \wedge \bigwedge_{i=1}^k (evb_i^\beta)^* \alpha_j \bigg{)}. 
\end{equation}
Here $\zeta(\alpha) = 1 + \sum_{j=1}^{k} j|\alpha_j|'$. The special cases are as follows: 
\begin{align}
	&\mathfrak{q}_{0,0}^{ST,\beta} := - (evb_0^\beta)_* 1 \in A^*(L),\\
	&\mathfrak{q}_{1,0}^{ST,\beta_0} := d\alpha,\\
	&\mathfrak{q}_{0,0}^{ST,\beta_0} := 0.
\end{align}

For the cleanest staments, we will use a sign convention differing from \cite{ST3}.
\begin{defi}
	Let the operations $\q_{k,l}^\beta: A^*(L)^{\otimes k} \otimes A^*(X)^{\otimes l} \rightarrow A^*(L)$ be defined by: \begin{equation}
		\q_{k,l}^\beta(\al_1, \dots, \al_k;\ga_1, \dots, \ga_l) = (-1)^{\dagger(\al) + k-1} \q_{k,l}^{ST,\beta}(\al_k, \dots, \al_1;\ga_1, \dots, \ga_l).
	\end{equation}
	Here, for $\al = (\al_1, \dots, \al_k)$, we set $\dagger(\al) = \sum_{1 \leq i <j \leq k} |\al_i|'|\al_j|'$. This is the sign coming from reversing the order of $\al$. 
\end{defi}

\cite{ST2} also define closed operations:
\begin{equation}
	\mathfrak{q}_{\emptyset, l}^{ST,\beta}: A^*(X)^{\otimes l} \rightarrow A^*(X),
\end{equation}
by
\begin{equation}
	\mathfrak{q}_{\emptyset, l}^{ST,\beta}(\gamma_1, \dots, \gamma_l) := (-1)^{w_{\mathfrak{s}}(\beta)} (ev_0^\beta)_*(\bigwedge_{j=1}^l (ev_j^\beta)^* \gamma_j),
\end{equation}
with special cases: \begin{equation}
	\mathfrak{q}_{\emptyset, 1}^{ST,\beta_0} := 0, \; \mathfrak{q}_{\emptyset, 0}^{ST, \beta_0} := 0.
\end{equation}
We use these operations, without any sign change, so that $\q_{\emptyset,l} = \q^{ST}_{\emptyset,l}$. The quantum product $\_ \star \_: A^*(X) \otimes A^*(X) \rightarrow A^*(X)$, is then given by \begin{equation}
	\gamma_1 \star \gamma_2 = \mathfrak{q}_{\emptyset,2}(\gamma_1, \gamma_2).
\end{equation}

We also define new operations coming from the moduli spaces with horocyclic constraints $\MM_{k+1,l,\perp_i}(\beta)$. We first define these using sign conventions similar to \cite{ST3}.
\begin{defi}
	Let $\mathfrak{q}^{ST,\beta}_{k,l;\perp_i}: A^*(L)^{\otimes k} \otimes A^*(X)^{\otimes l} \rightarrow A^*(L),$
	be defined by 
	\begin{equation}
		\mathfrak{q}^{ST,\beta}_{k,l;\perp_i}(\alpha_1, \dots, \alpha_k;\gamma_1, \dots, \gamma_l) = (-1)^{\zeta(\alpha) + {\zeta_{\perp}(\alpha;\gamma)}}hol(\partial \beta) (evb_0^\beta)_*\big{(}\bigwedge_{j=1}^{l} (evi_j^{\beta})^* \gamma_j \wedge \bigwedge_{j=1}^{k} (evb_j^{\beta})^*\alpha_j\big{)},
	\end{equation}
	where $\zeta_{\perp}(\alpha;\gamma) = |\alpha|' + |\gamma| + n$. Then, we modify the sign convention as before. We set
	\begin{equation}
		\q_{k,l,\perp_i}^\beta(\al_1, \dots, \al_k;\ga_1, \dots, \ga_l) = (-1)^{\dagger(\al) + k-1} \q_{k,l,\perp_{k+1-i}}^{ST,\beta}(\al_k, \dots, \al_1;\ga_1, \dots, \ga_l).
	\end{equation}
		The sign $\dagger(\al)$ is as before. When $i = 0$, $\perp_{k+1}$ should be interpreted as $\perp_0$.
\end{defi}

For all of the above $\mathfrak{q}^\beta$ operations, set \begin{equation}
	\mathfrak{q}_{*} = \sum_\beta Q^{\omega(\beta)}e^{\frac{\mu(\beta)}{2}} \mathfrak{q}^\beta_{*}.
\end{equation}
Here $\mu: H_2(X,L) \rightarrow \mathbb{Z}$ is the Maslov-class, and $e$ is of degree $2$. We thus have operations \begin{equation}
	\mathfrak{q}_{k,l}:A^*(L;\Lambda^e)^{\otimes k} \otimes A^*(X;\Lambda^e)^{\otimes l} \rightarrow A^*(L;\Lambda^e).
\end{equation}

Let $\langle \alpha_1, \alpha_2 \rangle_L = (-1)^{|\alpha_2|}\int_L \alpha_1 \wedge \alpha_2$ be the Poincar\'e pairing on $L$. \cite{ST3} prove results about the operations $\q^{ST}$, we state the analogous results for our operations $\q$. These follows by a direct verification of signs from the results in \cite{ST3}.
\begin{nprop}[cyclic symmetry, see {\cite[Proposition~3.3]{ST3}}]
	\label{cyclic symmetry}
	For any $\alpha = (\alpha_1, \dots, \alpha_{k+1})$ and $\gamma = (\gamma_1, \dots, \gamma_l)$:
	\begin{equation}
		\langle \mathfrak{q}_{k,l}(\alpha_1, \dots, \alpha_k; \gamma_1, \dots, \gamma_l),\alpha_{k+1} \rangle_L = (-1)^{|\alpha_{k+1}|'\epsilon_k(\alpha)} \langle  \mathfrak{q}_{k,l}(\alpha_{k+1}, \alpha_1, \dots, \alpha_{k-1}; \gamma_1, \dots, \gamma_l),\alpha_{k} \rangle_L
	\end{equation}
\end{nprop}
\begin{nprop}[degree property, see {\cite[Proposition~3.4]{ST3}}]
	\label{degree property}
	For any $\alpha = (\alpha_1, \dots, \alpha_{k})$ and $\gamma = (\gamma_1, \dots, \gamma_l)$:
	\begin{align}
		|\mathfrak{q}^\beta_{k,l}(\alpha_1, \dots, \alpha_k; \gamma_1, \dots, \gamma_l)| &= 2 + \epsilon(\alpha) -\mu(\beta) + \sum_{j=1}^l (|\gamma_j| - 2)  \\
		&\equiv \epsilon(\alpha) + \sum_{j=1}^l |\gamma_j| \; (\text{mod } 2)
	\end{align}
	The last equality holds as $L$ is orientable so the Maslov-index of any disk is even. By construction, we then have: \begin{equation}
		|\q_{k,l}(\al,\ga)| = 2 + \epsilon(\al) + |\ga| -2l. 
	\end{equation}
\end{nprop}
\begin{nprop}[unit property, see {\cite[Proposition~3.2]{ST3}}]
	\label{unit property}
	For $f \in A^0(L)$, $\alpha_1, \dots, \alpha_k \in A^*(L;R)$ and $\gamma \in A^*(X;Q)^{\otimes l}$, we have:
	\begin{equation}
		\mathfrak{q}^\beta_{k,l}(\alpha_1, \dots, \alpha_{i-1},f,\alpha_i, \dots, \alpha_k; \gamma) =
		\begin{cases}
			df  &(k,l,\beta) = (1,0,\beta_0)\\
			(-1)^{|f|}f\alpha_1 &(k,l,\beta) = (2,0,\beta_0), \; i=1\\
			(-1)^{|\alpha_1||f|'}f\alpha_1 & (k,l,\beta) = (2,0,\beta_0), \; i=2\\
			0 & \text{otherwise}
	\end{cases}\end{equation}
\end{nprop}
\begin{nprop}[top degree property, see {\cite[Proposition~3.12]{ST3}}]
	\label{top degree property}
	We have $(\mathfrak{q}^\beta_{k,l}(\alpha;\gamma))_n = 0$ for all lists $\alpha$, $\gamma$ for all $(k,l,\beta) \notin \{ (1,0,\beta_0), (0,1,\beta_0), (2,0,\beta_0)  \}$. Here $\delta_n$ denotes the degree $n$ part of a differential form $\delta \in A^*(L;R)$.
\end{nprop}
\begin{nprop}[divisor property, see {\cite[Proposition~3.9]{ST3}}]
	\label{divisor property}
	Assume $\gamma_1 \in A^2(X,L)$, $d\gamma_1 = 0$, then \begin{equation}
		\mathfrak{q}^\beta_{k,l}(\alpha, \gamma) = (\int_\beta \gamma_1)\cdot \mathfrak{q}^\beta_{k,l-1}(\alpha; \bigotimes_{j \geq 2} \gamma_j)
	\end{equation}
\end{nprop}
\begin{nprop}[energy-zero property, see {\cite[Proposition~3.8]{ST3}}]
	\label{energy zero property}
	For $k \geq 0$, 
	\begin{equation}
		\mathfrak{q}_{k,l}^{\beta_0}(\alpha_1, \dots, \alpha_k; \gamma_1, \dots, \gamma_l) = \begin{cases*}
			d\alpha_1, &$(k,l) = (1,0)$,\\
			(-1)^{|\alpha_1|}\alpha_1 \wedge \alpha_2, &$(k,l) = (2,0)$,\\
			\gamma_1|_L, &$(k,l) = (0,1)$,\\
			0, &\text{otherwise}.
		\end{cases*}
	\end{equation}
\end{nprop}
\begin{nprop}[fundamental class property, see {\cite[Proposition~3.7]{ST3}}]
	\label{fundamental class property}
	For $k \geq 0$, \begin{equation}
		\mathfrak{q}_{k,l}^{\beta}(\alpha_1, \dots, \alpha_k; 1, \gamma_1, \dots, \gamma_{l-1}) = 
		\begin{cases}
			1, &(k,l,\beta) = (0,1,\beta_0),\\
			0, & \;\text{otherwise}
	\end{cases}\end{equation}
\end{nprop}
Let  $\gamma = (\gamma_1, \dots,\gamma_l)$ be a list of differential forms on $X$. For subsets $I \sqcup J = \{ 1, \dots, l\}$, define $sign^\gamma(I,J)$ by the equation \begin{equation}
	\bigwedge_{i\in I} \gamma_i \wedge \bigwedge_{j\in J} \gamma_j = (-1)^{sign^\gamma(I,J)}\bigwedge_{s\in [l]} \gamma_s,
\end{equation}
or explicitly \begin{equation}
	sign^\gamma(I,J) = \sum_{i \in I, j \in J, j<i} |\gamma_i||\gamma_j|.
\end{equation}
\begin{nprop}[structure equation for $\q_{k,l}$, see{\cite[Proposition~2.4]{ST3}}]
	\label{boundary of q operation}
	For any $\alpha = (\alpha_1, \dots, \alpha_{k})$ and $\gamma = (\gamma_1, \dots, \gamma_l)$:
	\begin{align}
		0 &= \sum_{\substack{ P \in S_3[k] \\ (2:3) = \{ j \}}}  (-1)^{|\gamma^{(1:3)}| + 1}\mathfrak{q}_{k,l}(\alpha;\gamma^{(1:3)} \otimes d\gamma_j \otimes \gamma^{(3:3)})\\
		&+ \sum_{\substack{ P \in S_3[k] \\ I \sqcup J = [l]}} (-1)^{i(\alpha, \gamma, P, I)} \mathfrak{q}_{k_1+1+k_3,|I|}(\alpha^{(1:3)}\otimes \mathfrak{q}_{k_2,|J|}(\alpha^{(2:3)};\gamma^{J})\otimes \alpha^{(3:3)};\gamma^I),
	\end{align}
	where \begin{equation}i(\alpha,\gamma, P, I) = (|\gamma_J| + 1)\epsilon_1 + |\gamma_I| + sign^\gamma(I,J).\end{equation}
\end{nprop}
The new result we prove concerns the boundary of the moduli spaces $\mathcal{M}_{k+1,l; \perp_0}(X,\beta)$. The proof is given in chapter \ref{proof of boundary of horocycle moduli}.
\begin{nprop}[structure equation for $\q_{k,l,\perp_0}$]
	\label{boundary of horocycle moduli}
	\begin{align}
		0
		\label{boundary of horocycle moduli line 1}
		& = \sum_{\substack{S_3[l]\\ (2:3) = \{ i \}}} (-1)^{1+|\gamma^{(1:3)}|} \mathfrak{q}_{k,l;\perp_0}(\otimes_{j=1}^{k} \alpha_j; \gamma^{(1:3)} \otimes d\gamma_i \otimes \gamma^{(3:3)})\\
		\label{boundary of horocycle moduli line 2}
		&+ \sum_{\substack{J_1 \cup J_2 =[l] \\1,2 \in J_2}} (-1)^{sign^\gamma(J_1,J_2)}\mathfrak{q}_{k, j_1}(\alpha; \mathfrak{q}_{\emptyset, j_2}(\gamma^{J_2}) \otimes \gamma^{J_1})\\
		\label{boundary of horocycle moduli line 3}
		&+ \sum_{\substack{J_1 \cup J_2 =[l] \\1,2 \in J_2\\ P \in S_3[k]}} (-1)^{sign^\gamma(J_1,J_2) + \epsilon_1|\gamma^{J_2}| + 1} \mathfrak{q}_{k_1+k_3+1, j_1} (\alpha^{(1:3)} \otimes \mathfrak{q}_{k_2,j_2;\perp_0}(\alpha^{(2:3)}; \gamma^{J_2})\otimes \alpha^{(3:3)}; \gamma^{J_1})\\
		\label{boundary of horocycle moduli line 4}
		&+ \sum_{\substack{J_1 \cup J_2=[l]\\ 1,2 \in J_1\\ P \in S_3[k]}} (-1)^{sign^\gamma(J_1,J_2) + \epsilon_1(|\gamma^{J_2}|+1) + |\gamma^{J_1}|} \mathfrak{q}_{k_1+k_3+1, j_1;\perp_0} (\alpha^{(1:3)} \otimes \mathfrak{q}_{k_2,j_2}(\alpha^{(2:3)};\gamma^{J_2})\otimes \alpha^{(3:3)}; \gamma^{J_1})\\
		\label{boundary of horocycle moduli line 5}
		&+\sum_{\substack{J_1 \cup J_2 \cup J_3 = [l]\\1\in J_2,\;2\in J_3\\P \in S_5[k]}} (-1)^{A_5}\mathfrak{q}_{k_1+k_3 + k_5+2, j_1} (\alpha^{(1:5)} \otimes \mathfrak{q}_{k_2,j_2}(\alpha^{(2:5)};\gamma^{J_2})\otimes \alpha^{(3:5)} \otimes \mathfrak{q}_{k_4,j_3}(\alpha^{(4:5)};\gamma^{J_3}) \otimes \alpha^{(5:5)};\gamma^{J_1}).
	\end{align}
	Where \begin{equation}
		A_5 = sign^{\gamma}(J_1,J_2,J_3)+ |\gamma^{J_2}|+ (|\gamma^{J_2}| + 1)\epsilon_1 + (|\gamma^{J_3}| + 1)(\epsilon_1 + \epsilon_2 + \epsilon_3) + 1
	\end{equation}
\end{nprop}
Additionally we need the following proposition, which is similar to \cite[Lemma~3.10]{ST2} and is proven in section \ref{sec: Proof of unit on the horocycle}.
\begin{nprop}[Unit on the horocycle]
	\label{unit on the horocycle}
	For $\alpha = (\alpha_0, \dots, \alpha_k) \in A^*(L)^{\otimes{k+1}}$ and $\gamma \in A^*(X)^{\otimes l}$ we have:
	\begin{align}
		\langle \mathfrak{q}_{k,l}(\widetilde{\alpha};\gamma),\alpha_0 \rangle_L &= \sum_{P \in S_2[k]} (-1)^{\epsilon_1+|\gamma|} \langle \mathfrak{q}_{k+1,l,\perp_i}(\alpha^{(1:2)},1,\alpha^{(2:2)};\gamma), \alpha_0 \rangle_L\\
		&= \sum_{P \in S_2[k]} (-1)^{\epsilon_1+(\epsilon_1 + 1)(\epsilon_2 + |\alpha_0|')+|\gamma|} \langle \mathfrak{q}_{k+1,l,\perp_0}(\alpha^{(2:2)},\alpha_0,\alpha^{(1:2)};\gamma), 1 \rangle_L.
	\end{align}
\end{nprop}

\subsection{Bulk-deformed $\mathfrak{q}$-operations}
Let $U \subset H^*(X;\CC)$ be a graded vector subspace, and recall the definition of $Q_U$ from Section \ref{coefficient rings}. The same formulae as before then define $\q$-operations for differential forms with coefficients in $Q_U^e$. Thus, for example we have: \begin{equation}
	\mathfrak{q}_{k,l}:A^*(L;Q_U^e)^{\otimes k} \otimes A^*(X;Q_U^e)^{\otimes l} \rightarrow A^*(L;Q_U^e).
\end{equation}
We then have: 
\begin{nprop}[Linearity, see {\cite[Proposition~3.1]{ST3}}]
		The $\q$-operations are multilinear, in the sense that for $f \in Q_U^e$, $\al = (\al_1, \dots, \al_k)$ and $\ga = (\ga_1, \dots, \ga_l)$, we have: \begin{equation}
			\q_{k,l}(\al_1, \dots, \al_{i-1}, f \al_i, \dots, \al_k; \ga) = (-1)^{|f|(i + \sum_{j=1}^{i-1}|\al_j| + |\ga|)}f\q_{k,l}(\al, \ga),
		\end{equation}
		and \begin{equation}
			\q_{k,l}(\al; \ga_1, \dots, f\ga_i, \dots, \ga_l) = (-1)^{|f|\sum_{j=1}^{i-1}|\ga_j|}f\q_{k,l}(\al, \ga).
		\end{equation}
\end{nprop}

\begin{defi}
	A \emph{bulk-deformation pair} over $U$ is a pair $(b,\ga)$. Here $\gamma \in \mathcal{I}_UA^*(X;Q_U)$ is a bulk-deformation parameter over $U$ and $b \in \mathcal{I}_U A^*(L,Q_U^e)$ with $|b| = 1$.
\end{defi}
For a bulk-deformation pair $(b,\ga)$ define the bulk-deformed operations:
 \begin{equation}
\mathfrak{q}^{b,\gamma}_{k,l}(\al_1, \dots, \al_k; \ga_1, \dots, \ga_l) := \sum_{\substack{s,t \geq 0\\s_0 + \dots + s_k = s}} \frac{1}{t!}\mathfrak{q}_{k+s,l+t}(b^{\otimes s_0} \otimes \al_1 \otimes \dots \otimes \al_k \otimes b^{\otimes s_k}; \ga_1 \otimes \dots \otimes \ga_l \otimes \gamma^{\otimes t}).
\end{equation}
Similarly define:\begin{equation}
\mathfrak{q}^{\gamma}_{\emptyset,l}(\ga_1, \dots, \ga_l) = \sum_t \frac{1}{t!}\mathfrak{q}_{\emptyset,l+t}(\ga_1 \otimes \dots \otimes \ga_l \otimes \ga^{\otimes t}).
\end{equation}
Finally we define the bulk-deformed horocyclic $\mathfrak{q}$-operations by:\begin{equation}
\mathfrak{q}^{b,\gamma}_{k,l,\perp_i}(\al_1, \dots, \al_k; \ga_1, \dots, \ga_l) := \sum_{\substack{s,t \geq 0\\s_0 + \dots + s_k = s}} \frac{1}{t!}\mathfrak{q}_{k+s,l+t,\perp_{i + \sum_j^{i-1} s_j}}(b^{\otimes s_0} \otimes \al_1 \otimes \dots \otimes \al_k \otimes b^{\otimes s_k}; \ga_1 \otimes \dots \otimes \ga_l \otimes \gamma^{\otimes t}).
\end{equation}
We also define:
\begin{defi}
	\label{Quantum cup product}
	For a bulk-deformation parameter $\ga_U$, the bulk-deformed quantum cup product is defined by:
	\begin{align}
		\star_{\ga_U}: H^*(X;Q_U^e) \otimes H^*(X;Q_U^e) \rightarrow H^*(X;Q_U^e),
	\end{align}
	by setting $\eta_1 \star_{\ga_U} \eta_2 = \q^{\ga_U}_{\emptyset,2}(\eta_1,\eta_2)$.
\end{defi}
\begin{remark}
	When $(\ga,b)$ is a bulk-deformation pair, the degree assumptions on $\gamma$ and $b$ imply that the properties \ref{cyclic symmetry} - \ref{unit on the horocycle}, with the exception of the energy-zero property \ref{energy zero property}, all hold for the operations $\mathfrak{q}^{b,\gamma}$ with the same signs as before.
\end{remark}

We will need the following lemma later on; it follows from an easy verification of signs.
\begin{nlemma}
	\label{derivative of q}
	For $v \in Der_{\Lambda^e} Q^e$, $\alpha \in A^*(L;Q^e)^{\otimes k}$ and $\eta \in A^*(X;Q^e)^{\otimes l}$ all of homogeneous degrees, we have:
	\begin{align}v(\qb_{k,l}(\alpha;\eta)) &= \sum_{\substack{P \in S_3[k]\\(2:3) = i}} (-1)^{|\eta^{(1:3)}||v|} \qb_{k,l}(\al;\eta^{(1:3)}\otimes v(\eta_i) \otimes \eta^{(3:3)})\\
		&+ (-1)^{|\eta||v|}\qb_{k,l+1}(\al;\eta \otimes v(\gamma))\\
		&+\sum_{P \in S_2[k]} (-1)^{(|\eta|+\epsilon_1 + 1)|v| + 1} \qb_{k+1,l}(\al^{(1:2)}\otimes v(b) \otimes \alpha^{(2:2)};\eta)\\
		&+ (-1)^{(|\eta|+ 1)|v|} \qb_{k,l}(v(\al);\eta).
	\end{align}
	Here \begin{equation} v(\al) = \sum_{\substack{ P \in S_3[k] \\ (2:3) = \{ j \}}} (-1)^{\epsilon_1|v|} \al^{(1:3)} \otimes v(\al_j) \otimes \al^{(3:3)}.\end{equation}
\end{nlemma}
A similar lemma holds for $v = e \partial_e$, here one gets an additional term, as the $\q$ operations depend on $e$. Also note that $\partial_e(\ga) = 0$ by definition of a bulk-parameter. First, define operations weighted by the Maslov index $\mu$: $\widetilde{\mathfrak{q}}_{k,l}^{b,\ga} = \sum_{\beta} \mu(\beta)\mathfrak{q}_{k,l}^{b,\ga,\beta}$. We then have: \begin{nlemma}
	\label{derivative of q in e direction}
	For $\alpha \in A^*(L;Q^e)^{\otimes k}$ and $\eta \in A^*(X;Q^e)^{\otimes l}$ all of homogeneous degrees, we have:
	\begin{align}e\partial_e(\qb_{k,l}(\alpha;\eta)) &= \sum_{\substack{P \in S_3[k]\\(2:3) = i}} \qb_{k,l}(\al;\eta^{(1:3)}\otimes e\partial_e(\eta_i) \otimes \eta^{(3:3)})\\
		&-\sum_{P \in S_2[k]} \qb_{k+1,l}(\al^{(1:2)}\otimes e\partial_e(b) \otimes \alpha^{(2:2)};\eta)\\
		&+ \qb_{k,l}(e\partial_e(\al);\eta)\\
		&+ \frac{1}{2}\widetilde{\mathfrak{q}}_{k,l}^{b,\ga}(\al;\eta).
	\end{align}
\end{nlemma}

\subsection{Fukaya $A_\infty$-algebra}
\label{Fukaya algebra}
Let $U \subset H^*(X;\CC)$ and $(\ga,b)$ be a bulk-deformation pair over $U$. Assume $\ga$ satisfies Assumption \ref{bulk-deformation assumption}, fixing the derivation $\YY \in Der_{\Lambda^e} Q^e_U$ such that $[\YY(\ga)] = c_1$. For ease of notation, write $Q = Q_U$.

Solomon and Tukachinsky \cite[Theorem~1]{ST3} construct an $A_\infty$-algebra $A^{ST}$ using the operations $\q^{ST}_{k,0}$. We have different sign conventions for our operations $\q$, but the following still holds.
\begin{defi}
	Let $(A := CF^*(L,L)[e] := A^*(L;Q^e),\mathfrak{m}_k := \qb_{k,0},\langle \;, \; \rangle_L, 1)$. It follows directly from the properties of the $\q$ operations that this forms an $n$-dimensional, strictly unital and cyclic $A_\infty$-algebra. It follows from the degree property \ref{degree property}, and the definition that $|e| = 2$ that this $A_\infty$-algebra is Euler-graded with Euler vector field $E = e\partial_e + E_U$. The grading operator is defined by $Gr(f\al) = (|f| + |\al|)f\al$ for $f \in Q^e$ and $\al \in A^*(L)$. Furthermore, this $A_\infty$-algebra is (possibly) curved. The valuation $\zeta_A$ is induced by the valuation $\zeta_Q$, defined in \eqref{valuation on coefficient ring}.
\end{defi}
\begin{defi}
	Let $CF^*(L,L) := CF^*(L,L)[e] \otimes_{Q^e} Q$ be the $A_\infty$-algebra obtained by setting $e=1$.
\end{defi}
\begin{remark}
	\label{comparison of algebras}
	 Recall the definition of the negative-opposite of an $A_\infty$-algebra (Definition \ref{negative-opposite algebra}).  Then $CF^*(L,L)$ is related to $A^{ST}$ as $CF^*(L,L)^{-op} = A^{ST}$.
\end{remark}

Whenever we have to pick a basis for $A$ in order to compute derivatives, we will always pick a constant basis, i.e.\ one in $A^*(L;\mathbb{C})$.

Recall the connection \begin{equation}
	\nabla^{GGM}: Der_\Lambda^e Q^e \otimes HC_*^-(A) \rightarrow u^{-1}HC_*^-(A).
\end{equation}
We will now define a connection $\widetilde{\nabla}$ which agrees with $\nabla^{GGM}$ up to homotopy. First, for $v \in Der_\mathbb{C} Q^e$ define the length zero Hochschild cochain $\phi_v := v(b) \in A$. Also let $\mathfrak{m}_{v}(\alpha) := \mathfrak{q}^{b,\ga}_{k,1}(\alpha,v(\gamma))$. Lemma \ref{derivative of q} then shows: \begin{ncor}
	\label{derivative of m}
	For $v \in Der_{\Lambda^e} Q^e$ we have:
	\[v(\mathfrak{m}) = \mathfrak{m}_v + [\phi_v,\mathfrak{m}].\]
\end{ncor}
\begin{defi}
	\label{defi: homotopy Jv}
	For $v \in Der_{\Lambda^e Q^e}$, define $J_v = i\{ \phi_v \}$.
\end{defi}
The Cartan homotopy formula \ref{Cartan Homotopy} then shows:
\begin{ncor}
	\label{modified connection}
	The connection defined by:
	\begin{equation}
	\widetilde{\nabla}_v(\alpha) := v(\alpha) - \mathcal{L}_{\phi_v}(\alpha) + (-1)^{|v|+1}u^{-1}i\{ \mb_v \}(\al)
	\end{equation} 
	satisfies: \begin{equation}
		\nabla_v = \widetilde{\nabla}_v + u^{-1}[J_v,b+uB].
	\end{equation}
\end{ncor}

We finish this section by showing that the property \ref{OC assumptions}.\ref{OC assumptions 2} holds. To this end, set $\widetilde{\mathfrak{m}}_k = \widetilde{\q}^{b,\ga}_{k,0}$. Lemma \ref{derivative of q in e direction} then shows:
\begin{nlemma}
	$2e\partial_e(\mathfrak{m}^{b,\ga}) = \widetilde{\mathfrak{m}}^{b,\ga} - [\mathfrak{m}^{b,\ga},\phi_{2e\partial_e}]$.
\end{nlemma}
Applying lemma \ref{derivative of m} to $v= \YY$, and using the divisor equation \eqref{divisor property} to rewrite $2\m^{b,\ga}_{\YY} = \widetilde{\m}^{b,\ga}$, we find: \begin{nlemma}
	\label{derivative of m in e direction}
	The Fukaya $A_\infty$-algebra thus defined satisfies property \ref{OC assumptions}.\ref{OC assumptions 2}.
	\begin{equation}
		\YY(\mb) = e\partial_e (\mb) + [\mb, \phi_{e\partial_e} - \phi_{\YY}].
	\end{equation}
\end{nlemma}

\subsection{Closed-open and open-closed maps}
\label{closed-open and open-closed maps}
In \cite{FOOO} an open-closed map is defined directly. In our setup one would want to define it using operations $\mathfrak{p}_k: CC_*(A) \rightarrow A^*(X)$ given by: \begin{equation}
		\mathfrak{p}_k(\alpha) = (evi_1)_*(\bigwedge_{j=0}^k evb_{j}^*\alpha_j),
	\end{equation}
	where now the push forward is along the interior evaluation $evi_1: \mathcal{M}_{k+1,1} \rightarrow X$. However, in the present setup the push-forward along interior evaluation is not well defined, as $evi_1:\mathcal{M}_{k+1,l}(\beta) \rightarrow X$ need not be a submersion. 
	
The approach we thus take is to first construct a closed-open map, and then dualise to obtain an open-closed map. The fact that the closed-open map is an algebra homomorphism (a result first shown by \cite{FOOOToric1} and \cite{Ga12}), will then translate into the result that the open-closed map is a morphism of $QH^*(X)$ modules (see also \cite{Ga12}). We will then construct a cyclic open-closed map, which was also done in \cite{Ga19}. We then upgrade the result that the open-closed map is a morphism of $QH^*(X)$-module to show that the cyclic open-closed map is a morphism of T-structures, using an argument due to \cite{GPS2}. We then show Assumptions \ref{OC assumptions} hold for the cyclic open-closed map in order to conclude that it respects the connection in the $u$-direction.

\label{Closed-open and open-closed}
\subsubsection{The closed-open map}
We define the closed-open map $\CO_e: A^*(X;Q^e) \rightarrow CC^*(A)$ on the chain level. Set \begin{equation}
	\CO_e(\eta)(\alpha_1, \dots, \alpha_k) = \qb_{k,1}(\alpha_1, \dots, \alpha_k; \eta).
\end{equation}
It follows from the unit property, Proposition \ref{unit property}, that the Hochschild cochain $\CO_e(\eta)$ is reduced. Furthermore, it follows from the degree property, Proposition \ref{degree property}, that $|\CO_e(\eta)| = |\eta|\; (mod\;2)$, so that the closed-open map is a $\mathbb{Z}/2$-graded map.
\begin{nlemma}
	\label{closed-open is a chain map}
	The closed-open map is a chain map. That is: $[\mathfrak{m},\CO_e(\eta)] = \CO_e(d\eta)$.
\end{nlemma}
\begin{proof}
	By definition: \begin{align}
		[\mathfrak{m},\CO_e(\eta)](\alpha_1, \dots, \alpha_k) &= \sum_{P \in S_3[k]} (-1)^{\epsilon_1(1+|\eta|)}\qb_{k_1+1+k_3,0}(\alpha^{(1:3)}\otimes \qb_{k_2,1}(\alpha^{(2:3)};\eta)\otimes \alpha^{(3:3)})\nonumber\\
		&\qquad + \sum_{P \in S_3[k]} (-1)^{\epsilon_1 +|\eta|}\qb_{k_1+1+k_3,1}(\alpha^{(1:3)}\otimes \qb_{k_2,0}(\alpha^{(2:3)})\otimes \alpha^{(3:3)};\eta),
	\end{align}
	which, by Proposition \ref{boundary of q operation} equals $\CO_e(d\eta)$.
\end{proof}
Let $\CO_e: H^*(X;Q^e) \rightarrow HH^*(A)$ be the induced map on cohomology. Next up we will prove the following, which is originally due to \cite{FOOOToric1} and \cite{Ga12} in different setups:

\begin{nprop}
	\label{closed-open is a algebra homomorphism}
	The closed-open map induces a unital algebra homomorphism on cohomology.
\end{nprop}
To this end, we first define:
\begin{defi}
	\label{closed-open homotopy}
	Let the homotopy operator $H:A^*(X;Q^e)^{\otimes 2} \rightarrow CC^*(A)$ be given by \begin{equation}
		H(\gamma_1, \gamma_2)(\alpha_1, \dots, \alpha_k) = \qb_{k,2,\perp_0}(\alpha;\gamma_1\otimes\gamma_2).
	\end{equation}
\end{defi}
Unitality follows from the fundamental class property \ref{fundamental class property}. The following lemma immediately implies that the closed-open map respects the product.
\begin{nlemma}
	\label{homotopy for CO algebra hom}
	The homotopy operator $H$ satisfies:
	\begin{equation}
		\CO_e(\gamma_1 \star \gamma_2) = \CO_e(\gamma_1) \cup \CO_e(\gamma_2) + H(d(\gamma_1\otimes \gamma_2)) + [\mathfrak{m},H(\gamma_1,\gamma_2)].
	\end{equation}
\end{nlemma}
\begin{proof}
We write down the terms one by one. Firstly:
	\begin{equation}
		\CO_e(\gamma_1 \star \gamma_2)(\alpha) = \qb_{k,1}(\alpha;\qb_{\emptyset,2}(\gamma_1,\gamma_2)).
	\end{equation}
	Furthermore:
	\begin{multline}
		\CO_e(\gamma_1) \cup \CO_e(\gamma_2)(\alpha) = (-1)^{|\gamma_1|}M^2(\CO_e(\gamma_1),\CO_e(\gamma_2))\\
		= \sum_{P\in S_5[k]} (-1)^{\star}\qb_{k_1+k_3+k_5+2}(\alpha^{(1:5)}\otimes \qb_{k_2,1}(\alpha^{(2:5)};\gamma_1)\otimes \alpha^{(3:5)} \otimes \qb_{k_4,1}(\alpha^{(4:5)};\gamma_2) \otimes \alpha^{(5:5)}),
	\end{multline}
	where $\star = |\gamma_1| + (|\gamma_1|+1)\epsilon_1 + (|\gamma_2|+1)(\epsilon_1+\epsilon_2+\epsilon_3)$.
	
	We then compute the homotopy terms: \begin{align}
		H(d(\gamma_1\otimes \gamma_2)) = \qb_{k,2,\perp_0}(\alpha;d\gamma_1\otimes\gamma_2) + (-1)^{|\gamma_1|}\qb_{k,2,\perp_0}(\alpha;\gamma_1\otimes d\gamma_2).
	\end{align}
	Finally we find: 
	\begin{align}
		[\mathfrak{m},H(\gamma_1,\gamma_2)] &= \mathfrak{m} \circ H(\gamma_1,\gamma_2) + (-1)^{|\gamma_1| + |\gamma_2| + 1}H(\gamma_1,\gamma_2) \circ \mathfrak{m}\\
		&= \sum_{ P \in S_3[k]} (-1)^{\epsilon_1(|\gamma_1| + |\gamma_2|)} \qb_{k_1+k_3+1,0} (\alpha^{(1:3)} \otimes \qb_{k_2,2;\perp_0}(\alpha^{(2:3)}; \gamma_1 \otimes \gamma_2)\otimes \alpha^{(3:3)})\\
		&\qquad + \sum_{P \in S_3[k]} (-1)^{\epsilon_1 + |\gamma_1| + |\gamma_2| + 1} \qb_{k_1+k_3+1,2;\perp_0} (\alpha^{(1:3)} \otimes \qb_{k_2,0}(\alpha^{(2:3)})\otimes \alpha^{(3:3)}; \gamma_1 \otimes \gamma_2).
	\end{align}
	Lemma \eqref{homotopy for CO algebra hom} then follows by applying the structure equation \eqref{boundary of horocycle moduli} with interior inputs $\gamma_1 \otimes \gamma_2$.
\end{proof}
We conclude this section with the following observations:
\begin{nlemma}
	\label{CO of c1}
	$\CO_e(c_1) = \widetilde{\mathfrak{m}}/2$.
\end{nlemma}
\begin{proof}
	 Pick a representative $\eta \in A^2(X,L)$ for the Maslov class $\mu$. Then $\eta$ also represents $2c_1 \in H^*(X)$. The divisor property \ref{divisor property} shows that $\CO_e(\ga_1) = \widetilde{\mathfrak{m}}$.
\end{proof}
Furthermore, by definition of $\m_v$ we have:
\begin{nlemma}
	\label{CO of derivatives}
	For $v \in Der_\mathbb{C} Q^e$ we have: $\CO_e(v(\gamma)) = \mathfrak{m}_{v}$. It follows that $\CO_e(E_U(\gamma)) = \mathfrak{m}_{E_U}$.
\end{nlemma}
\subsubsection{The open-closed map}
The open-closed map will take the form $\OC: HH_*(A) \rightarrow QH^*(X;Q^e)$. To this end, we first define the open-closed pairing $A^*(X;Q^e) \otimes CC_*(A) \rightarrow Q^e$. We then show this descends to a pairing on (co)homology. Finally, by dualising the first factor, and using Poincar\'e duality, we obtain the open-closed map.

\begin{defi}
	\label{defi: open-closed pairing}
	The open-closed pairing $\langle \_,\OC_e(\_) \rangle: A^*(X;Q^e) \otimes CC_*(A) \rightarrow Q^e$ is defined as \begin{equation}
		\langle \eta, \OC_e(\alpha) \rangle := (\CO_e(\eta),\alpha) = (-1)^{|\alpha_0|(\epsilon(\widetilde{\alpha}) + 1)}\langle \qb_{k,1}(\widetilde{\alpha};\eta),\alpha_0 \rangle_L.
	\end{equation}
Here the pairing $(\cdot,\cdot)$ is as in Equation \eqref{pairing on hochschild chains}.
\end{defi}
We will show (Lemma \ref{OC respects b differential}) that the open-closed pairing descends to homology, so that the following makes sense.
\begin{defi}
	The open-closed map $\OC: HH_*(A) \rightarrow QH^*(X;Q^e)$ is defined by requiring that \begin{equation}
		\langle \eta,\OC(\al) \rangle_X = \langle \eta, \OC_e(\alpha) \rangle.
	\end{equation}
On the left, the pairing is the Poincar\'e pairing on $X$, and on the right the pairing is the open-closed pairing.
\end{defi}
Since the closed-open map is a chain map (Lemma \ref{closed-open is a chain map}), and the pairing respects differentials (Lemma \ref{lem: pairing hochschil (co)chains}), the following is immediate. 
\begin{nlemma}
	\label{OC respects b differential}
	We have:
	\begin{equation}\langle d\eta, \OC_e(\alpha) \rangle + (-1)^{|\eta|} \langle \eta, \OC_e(b(\alpha)) \rangle = 0.
	\end{equation}
	The open-closed pairing thus descends to (co)homology.
\end{nlemma}
Ganatra \cite{Ga12} shows the closed-open map makes $HH_*(A)$ into a $QH^*(X)$-module. We prove this in our setup.
\begin{nprop}
	\label{open-closed map respects module structures}
	The open-closed map is a map of $QH^*(X;Q^e)$-modules.
\end{nprop}
To this end, first define: \begin{defi}
	\label{homotopy pairing}
Recall the map $H$ from Definition \ref{closed-open homotopy}. Then let \begin{align}
		\langle \_,G(\_) \rangle :A^*(X)^{\otimes 2} &\otimes CC_*(A) \rightarrow R^e\\
		\gamma_1 \otimes \gamma_2 &\otimes \alpha \mapsto \langle \gamma_1 \otimes \gamma_2, G(\alpha) \rangle  = (H(\gamma_1, \gamma_2),\alpha).
	\end{align}
\end{defi}
The following lemma follows directly from Lemma \ref{homotopy for CO algebra hom}.
\begin{nlemma}
	\label{module structure homotopy}
	The pairing $G$ satisfies: 
	\begin{multline}
		\langle \gamma_1 \star \gamma_2, \OC_e(\alpha) \rangle = \langle \gamma_1, \OC_e(\CO_e(\gamma_2) \cap \alpha)\rangle + \langle d(\gamma_1 \otimes \gamma_2), G(\alpha) \rangle + (-1)^{|\ga_1| + |\ga_2|}\langle \gamma_1 \otimes \gamma_2, G(b(\alpha)) \rangle.
	\end{multline}
\end{nlemma}
\begin{proof}[Proof of Proposition \ref{open-closed map respects module structures}]
		Lemma \ref{module structure homotopy} shows that on homology we have:: \begin{align}
		\langle \gamma_1, \OC_e(\CO_e(\gamma_2) \cap \alpha)\rangle_X &= \langle \ga_1 \star \ga_2,\OC(\al) \rangle_X\\
		&= \langle \ga_1, \ga_2 \star \OC(\al) \rangle_X.
	\end{align}
	Thus $\ga \star \OC(\al) = \OC_e(\CO_e(\gamma) \cap \alpha)$.
\end{proof}

\subsubsection{The cyclic open-closed map}
\label{cyclic open-closed map}
In order to define the cyclic open-closed map, we first show the open-closed pairing descends to cyclic homology.
\begin{defi}
	Extend the open-closed pairing (Definition \ref{defi: open-closed pairing}) $u$-linearly to a map: \begin{equation}
		\langle \_,\OC_e^-(\_) \rangle: (A^{*}(X;Q^e)[[u]])  \otimes CC^{-}(A) \rightarrow Q^e[[u]].
	\end{equation}
\end{defi}
We then have:
\begin{nlemma}
	\label{B lemma}
	The open-closed pairing descends to a pairing $QH_{-}^{*}(X;R^e) \otimes HC^{-}(A) \rightarrow Q^e[[u]]$.
\end{nlemma}
\begin{proof}
	We have already shown the open-closed pairing respects the first differential $b$. We are then done if $\langle \eta,\mathcal{OC}(B(\alpha)) \rangle = 0$.
	Now \begin{align}
		\langle \eta,\OC_e(B(\alpha)) \rangle &= \sum_{P \in S_2[k]} (-1)^{\epsilon_2(|\alpha_0|' + \epsilon_1)} ( \CO_e(\eta), 1[\alpha^{(2:2)}|\alpha_0|\alpha^{(1:2)}] )\\
		&= \sum_{P \in S_2[k]} (-1)^{\epsilon_2(|\alpha_0|' + \epsilon_1)+ \epsilon(\alpha) + |\eta|} \langle \qb_{k+1,1}(\alpha^{(2:2)}\otimes\alpha_0 \otimes \alpha^{(1:2)};\eta),1 \rangle_L\\
		&= 0.
	\end{align}
	The first equality is by definition of $B$ and $\OC$. The second equality holds by definition of $\CO_e$. The last equality follows by the top degree property \ref{top degree property}.
\end{proof}
Extend the Poincar\'e pairing on $X$ $u$-linearly to a pairing $QH^*(X;Q^e)[[u]] \otimes QH^*(X;Q^e)[[u]] \rightarrow Q^e[[u]]$. 
\begin{defi}
	The cyclic open-closed map $\OC_e^-: HC^-_*(A) \rightarrow QH^*(X;Q^e)[[u]]$ is defined by requiring that \begin{equation}
		\langle \eta,\OC_e^-(\al) \rangle_X = \langle \eta, \OC_e^-(\alpha) \rangle.
	\end{equation}
	On the left, the pairing is the Poincar\'e pairing on $X$, and on the right the pairing is the open-closed pairing.
\end{defi}
We now prove:
\begin{nthm}
	\label{open-closed map respects GGM connections}
	The cyclic open-closed map $\OC_e^-: HC^{-}(A) \rightarrow (QH^*(X;Q_U)[[u]],\nabla)$ is a morphism of T-structures over $Q^e \supset \Lambda^e$.
\end{nthm}
Our proof follows the same ideas as outlined by Ganatra-Perutz-Sheridan in talks. First we observe that the same reasoning as for Lemma \ref{B lemma} shows that: \begin{nlemma}
	\label{B11 lemma}
	$\langle \eta , \mathcal{OC}(B^{1,1}(\phi, \alpha)) \rangle  = 0$ for any $\eta \in A^*(X;Q^e)$, $\alpha \in CC_*(A)$ and $\phi \in CC^*(A)$.  
\end{nlemma}
Next we define: \begin{defi}
	For $v \in Der_{\Lambda^e} Q^e$, set \begin{equation}
		\langle \eta,\widetilde{G}_{v}(\al)\rangle := (-1)^{|v||\eta|}\langle \eta \otimes v(\ga), G(\al) \rangle,
	\end{equation}
	where we have extended $G$ from Definition \ref{homotopy pairing} $u$-linearly.
\end{defi}
We then have: \begin{nlemma}
	\label{lem: homotopy pairing modified connection}
	The pairing $\widetilde{G}_v$ satisfies:
	\begin{multline*}
	\langle \nabla^*_{v} \eta, \OC_e^- (\alpha) \rangle + (-1)^{|\eta||v|}\langle \eta, \OC_e^- (\widetilde{\nabla}_{v} \alpha) \rangle = v \left(\langle \eta, \OC_e^-(\alpha)\rangle\right) \\+ u^{-1}\left(\langle d\eta, \widetilde{G}_{v} (\alpha) \rangle + (-1)^{|\eta| + |v|}\langle \eta, \widetilde{G}_{v}\left((b+uB)(\alpha)\right)\rangle \right).
\end{multline*}
\end{nlemma}
In order to prove this, we first show the following: \begin{nlemma}
	\label{derivative of open-closed pairing} We have:
	\begin{equation}
		v \left(\langle \eta, \OC_e^-(\alpha)\rangle\right) + (-1)^{|\eta|+|v|}\langle \eta, \widetilde{G}_{v}(B(\alpha)) \rangle = \langle v(\eta), \OC_e^-(\alpha )\rangle + (-1)^{|\eta||v|}\langle \eta, \OC_e^- \left(v(\alpha) - \mathcal{L}_{\phi_v}(\al)\right) \rangle.
	\end{equation} 
\end{nlemma}
\begin{proof}
	Using Lemma \ref{unit on the horocycle}, we find that: \begin{equation}
		\langle \eta, \widetilde{G}_{v}(B(\alpha)) \rangle = (-1)^{|\eta||v| + |\alpha_0|(\epsilon(\widetilde{\alpha})+1) + |v| + |\eta| + 1}\langle \qb_{k,l+1}(\widetilde{\al};\eta \otimes v(\gamma)),\alpha_0 \rangle_L.
	\end{equation}
	We also write out the other terms. 
	\begin{align}
		\langle v(\eta), \OC_e^-(\alpha )\rangle &= (-1)^{|\alpha_0|(\epsilon(\widetilde{\alpha}) + 1)}\langle \qb_{k,1}(\widetilde{\alpha};v(\eta)),\alpha_0 \rangle_L,\\
		\langle \eta, \OC_e^- (v(\alpha)) \rangle &=  (-1)^{(|v|+|\alpha_0|)(\epsilon(\widetilde{\alpha}) + 1)}\langle \qb_{k,1}(\widetilde{\alpha};\eta),v(\alpha_0) \rangle_L +(-1)^{|\alpha_0|(\epsilon(\widetilde{\alpha}) + 1)}\langle \qb_{k,1}(v(\widetilde{\alpha});\eta),\alpha_0 \rangle_L,\\
		\langle \eta, \OC_e^-(\mathcal{L}_{\phi_v}(\al)) \rangle &= \sum_{P \in S_2[k]} (-1)^{|\alpha_0|(\epsilon(\widetilde{\alpha}) + 1) + |v|(\epsilon_1 + 1)}\langle \qb_{k+1,1}(\alpha^{(1:2)} \otimes v(b) \otimes \alpha^{(2:2)};\eta),\alpha_0 \rangle_L.
	\end{align}
	By definition of the open-closed pairing, we also have: \begin{equation}
		v \left( \langle \eta, \OC_e^-(\alpha)\rangle \right) = (-1)^{|\alpha_0|(\epsilon(\widetilde{\alpha}) + 1)}\langle v(\qb_{k,1}(\widetilde{\alpha};\eta)),\alpha_0 \rangle_L + (-1)^{|\alpha_0|(\epsilon(\widetilde{\alpha}) + 1) + |v|(\epsilon(\widetilde{\al}) +|\eta|+ 1)}\langle \qb_{k,1}(\widetilde{\alpha};\eta),v(\alpha_0) \rangle_L.
	\end{equation}
	Then, apply Lemma \ref{derivative of q} to compute $v(\qb_{k,1}(\widetilde{\alpha};\eta))$. Keeping track of all the signs shows the result.
\end{proof}
\begin{proof}[Proof of Lemma \ref{lem: homotopy pairing modified connection}]
	As all the terms in the above equation are $u$-linear, we may assume $\al$ and $\eta$ are independent of $u$, and then prove this order by order in $u$.
	
	To verify the $u^{-1}$ term, apply Lemma \ref{module structure homotopy} with $\gamma_1 = \eta$ and $\gamma_2 = v(\gamma)$ and use Lemma \ref{CO of derivatives} to compute $\CO_e(v(\gamma)) = \mathfrak{m}_{v}$.
	Equality of the $u^0$ terms is shown by Lemma \ref{derivative of open-closed pairing}.
\end{proof}
\begin{defi}
	Recall the definition of $J_v$ from Definition \ref{defi: homotopy Jv}. Then define the pairing: \begin{equation}
		\langle \eta,G_{v}(\al)\rangle := \langle \eta,\widetilde{G}_{v}(\al)\rangle + (-1)^{|\eta|'|v|' + 1} \langle \eta, \OC^-_e(J_v(\al)) \rangle.
	\end{equation}
\end{defi}
\begin{nprop}
	\label{homotopy pairing connection}
	The pairing $G_v$ satisfies:
	\begin{multline*}
		\langle \nabla^*_{v} \eta, \OC_e^- (\alpha) \rangle + (-1)^{|\eta||v|}\langle \eta, \OC_e^- (\nabla_{v} \alpha) \rangle = v \left(\langle \eta, \OC_e^-(\alpha)\rangle\right) \\+ u^{-1}\left(\langle d\eta, G_{v} (\alpha) \rangle + (-1)^{|\eta| + |v|}\langle \eta, G_{v}\left((b+uB)(\alpha)\right)\rangle \right).
	\end{multline*}
\end{nprop}
\begin{proof}
	This follows directly from Lemma \ref{lem: homotopy pairing modified connection}, Corollary \ref{modified connection}, and a verification of signs.
\end{proof}

\begin{proof}[Proof of Theorem \ref{open-closed map respects GGM connections}]
	By the properties of the Poincar\'e pairing \ref{poincare pairing GGM connection}:
	\begin{equation}
		\langle \eta,\nabla_v(\OC_e^-(\al)) \rangle_X = (-1)^{|\eta||v|} v(\langle \eta, \OC_e^-(\al) \rangle) - (-1)^{|\eta||v|}\langle \nabla^*_v \eta, \OC_e^-(\al) \rangle.
	\end{equation}
	Lemma \ref{homotopy pairing connection} then shows that on homology: \begin{equation}
		\langle \eta,\nabla_v(\OC_e^-(\al)) \rangle_X = \langle \eta, \OC_e^-(\nabla_v(\al)) \rangle,
	\end{equation}
	which shows that $\OC_e^-(\nabla_v(\al)) = \nabla_v(\OC_e^-(\al))$.
\end{proof}

We next show that the Assumptions \ref{OC assumptions} hold in our setup, so that the open-closed map respects u-connections. Assumption \ref{OC assumptions}(\ref{OC assumptions 1}) is Lemma \ref{homotopy pairing connection}. Assumption \ref{OC assumptions}(\ref{OC assumptions 2}) is lemma \ref{derivative of m in e direction} with $\phi = \phi_{e\partial_e} - \phi_{\YY}$.

The following lemma shows Assumption \ref{OC assumptions}(\ref{OC assumptions 3}) holds in our setup.
\begin{nlemma}
	\label{assumption 3 holds}
	For any $\eta$ and $\al$ we have: \begin{equation}
		\langle \eta, \YY(\OC_e^-)(\alpha) \rangle = \langle \eta, e\partial_e(\OC_e^-)(\al) \rangle + \langle \al, \OC_e^-(\mathcal{L}_\phi(\eta)) \rangle,
	\end{equation}
	Where $\phi$ is as above.
\end{nlemma}
\begin{proof}
	First note that a computation similar to Lemma \ref{derivative of open-closed pairing} shows that 	\begin{multline}
	e\partial_e \left(\langle \eta, \OC_e^-(\alpha)\rangle\right) = \langle e\partial_e(\eta), \OC_e^-(\alpha )\rangle + \langle \eta, \OC_e^- \left(v(\alpha) -\mathcal{L}_{\phi_{e\partial_e}}(\al)\right) \rangle \\+ \sum_\beta (-1)^{|\alpha_0|(\epsilon(\widetilde{\al}) +1)} \frac{\mu(\beta)}{2}\langle \qb_{k,1}(\widetilde{\al};\eta),\al_0 \rangle_L.
\end{multline}
We thus find that: \begin{equation}
	\langle \eta, e\partial_e(\OC_e^-)(\alpha)\rangle + \langle \eta, \OC_e^- (\mathcal{L}_{\phi_{e\partial_e}}(\al)) \rangle = \sum_\beta (-1)^{|\alpha_0|(\epsilon(\widetilde{\al}) +1)} \frac{\mu(\beta)}{2}\langle \qb_{k,1}(\widetilde{\al};\eta),\al_0 \rangle_L.
\end{equation}
Similarly, we have: \begin{align}
	\langle \eta, \YY(\OC_e^-)(\alpha)\rangle + \langle \eta, \OC_e^- (\mathcal{L}_{\phi_{\YY}}(\al)) \rangle &= (-1)^{|\alpha_0|(\epsilon(\widetilde{\al}) +1)} \langle \qb_{k,2}(\widetilde{\al};\eta\otimes c_1),\al_0 \rangle_L\\ &= \sum_\beta (-1)^{|\alpha_0|(\epsilon(\widetilde{\al}) +1)} \frac{\mu(\beta)}{2}\langle \qb_{k,1}(\widetilde{\al};\eta),\al_0 \rangle_L.
\end{align}
The last equality follows from the divisor property \ref{divisor property}. The result follows.
\end{proof}
Finally we show that assumption \ref{OC assumptions}(\ref{OC assumptions 4}) holds:
\begin{nlemma}
	\label{assumption 4 holds}
	We have: $Gr^- \circ \OC_e^- = \OC_e^- \circ Gr^-$.
\end{nlemma}
\begin{proof}
	First observe that for $\al = \al_0[\al_1, \dots, \al_k]$, with $\al_i \in A^*(L;\CC)$, we have: \begin{equation}
		Gr^-(\al) = (1 + \epsilon(\al))\al
	\end{equation}
	We then compute \begin{align}
		|\eta| + |\OC_e^-(\al)| - 2n &= |\langle \eta,\OC_e^-(\al) \rangle_X| \\
		&= |\langle \qb_{k,1}(\widetilde{\al};\eta),\al_0 \rangle_L|\\
		&= |\qb_{k,1}(\widetilde{\al};\eta)| + |\al_0| - n\\
		&= \epsilon(\widetilde{\al}) + |\eta| + |\al_0| - n.
	\end{align}
	We thus have $|\OC_e^-(\al)| = 1 +\epsilon(\al) + n$. As the grading $Gr^-$ on $QH^*(X;Q^e_U)[[u]]$ is shifted down by $n$ compared to the cohomological grading, we have \begin{equation}
		Gr^-(\OC_e^-(\al)) = (1+\epsilon(\al))\OC_e^-(\al) = \OC_e^-(Gr^-(\al)).
	\end{equation}
\end{proof}
Theorem \ref{assumptions imply TE isom} thus implies that the cyclic open-closed map $\OC_e^-$ is a morphism of TE-structures.
Now recall that $CF^*(L,L) = A \otimes_{Q^e} Q$. The cyclic open-closed map then restricts at $e=1$ to a map \begin{equation}
	\OC^-: HC^-_*(CF^*(L,L)) \rightarrow QH^*(X;Q).
\end{equation}
We have thus proved Theorem \ref{cyclic open closed theorem in intro}.
\begin{nthm}[Theorem \ref{cyclic open closed theorem in intro}]
	\label{cyclic open-closed morphism of u-VSHS}
	Under assumptions \ref{assumptions} the cyclic open-closed map $$\mathcal{OC}^-: HC^{-}(CF^*(L,L)) \rightarrow QH^*(X;Q_U)[[u]]$$ is a morphism of TE-structures.
\end{nthm}

\section{Analysis on the horocyclic moduli space}
\label{proof of boundary of horocycle moduli}
The main goal of this chapter is to prove Proposition \ref{boundary of horocycle moduli}. Finally, in the last section we proof Proposition \ref{unit on the horocycle}.

The following sections use the method of proof explained to the author in an unpublished draft by Jake Solomon and Sara Tukachinsky. We prove the following result for the operations $\q^{ST}_{k,l,\perp}$, which were defined using the sign convention similar to \cite{ST3}. 
\begin{nprop}
	\label{boundary of ST horocycle moduli}
	\begin{align}
		0
		& = \sum_{\substack{S_3[l]\\ (2:3) = \{ i \}}} (-1)^{1+|\gamma^{(1:3)}|} \mathfrak{q}^{ST}_{k,l;\perp_0}(\otimes_{j=1}^{k} \alpha_j; \gamma^{(1:3)} \otimes d\gamma_i \otimes \gamma^{(3:3)})\\
		&+ \sum_{\substack{J_1 \cup J_2 =[l] \\1,2 \in J_2}} (-1)^{sign^\gamma(J_1,J_2)}\mathfrak{q}^{ST}_{k, j_1}(\alpha; \mathfrak{q}^{ST}_{\emptyset, j_2}(\gamma^{J_2}) \otimes \gamma^{J_1})\\
		&+ \sum_{\substack{J_1 \cup J_2 =[l] \\1,2 \in J_2\\ P \in S_3[k]}} (-1)^{sign^\gamma(J_1,J_2) + \epsilon_1|\gamma^{J_2}| + 1} \mathfrak{q}^{ST}_{k_1+k_3+1, j_1} (\alpha^{(1:3)} \otimes \mathfrak{q}^{ST}_{k_2,j_2;\perp_0}(\alpha^{(2:3)}; \gamma^{J_2})\otimes \alpha^{(3:3)}; \gamma^{J_1})\\
		&+ \sum_{\substack{J_1 \cup J_2=[l]\\ 1,2 \in J_1\\ P \in S_3[k]}} (-1)^{sign^\gamma(J_1,J_2) + \epsilon_1(|\gamma^{J_2}|+1) + |\gamma^{J_1}|} \mathfrak{q}^{ST}_{k_1+k_3+1, j_1;\perp_0} (\alpha^{(1:3)} \otimes \mathfrak{q}^{ST}_{k_2,j_2}(\alpha^{(2:3)};\gamma^{J_2})\otimes \alpha^{(3:3)}; \gamma^{J_1})\\
		&+\sum_{\substack{J_1 \cup J_2 \cup J_3 = [l]\\1\in J_3,\;2\in J_2\\P \in S_5[k]}} (-1)^{A_5}\mathfrak{q}^{ST}_{k_1+k_3 + k_5+2, j_1} (\alpha^{(1:5)} \otimes \mathfrak{q}^{ST}_{k_2,j_2}(\alpha^{(2:5)};\gamma^{J_2})\otimes \alpha^{(3:5)} \otimes \mathfrak{q}^{ST}_{k_4,j_3}(\alpha^{(4:5)};\gamma^{J_3}) \otimes \alpha^{(5:5)};\gamma^{J_1}).
	\end{align}
	Here \begin{equation}
		A_5 = sign^{\gamma}(J_1,J_2,J_3)+ |\gamma^{J_2}|+ (|\gamma^{J_2}| + 1)\epsilon_1 + (|\gamma^{J_3}| + 1)(\epsilon_1 + \epsilon_2 + \epsilon_3)
	\end{equation}
\end{nprop}
Proposition \ref{boundary of horocycle moduli} then follows from the above by a direct verification of signs. In the following, we prove Proposition \ref{boundary of ST horocycle moduli} assuming that the $U(\Lambda)$-local system is trivial. The general result then follows easily.

Recall that the orientation on $\mathcal{M}_{k+1,l; \perp_0}(\beta)$ is defined by the fibre-product orientation, as defined in \cite[Section 2.2]{ST4}. We take the orientation on $I$ to be the positive orientation, so that $\partial I = \{1\} -\{ 0\}$. The boundary is then identified as: \begin{equation}
	\label{decomposing boundary of horocyclic}
	\partial\mathcal{M}_{k+1,l; \perp_0}(\beta) \cong \partial (I \times_{D^2} \mathcal{M}_{k+1,l}(\beta)) = \partial I \times_{D^2} \mathcal{M}_{k+1,l}(\beta) - I \times_{D^2} \partial \mathcal{M}_{k+1,l}(\beta).
\end{equation}
We now further decompose each of the terms in the boundary. For each, we identity them with a fibre product of other moduli spaces (both with and without horocyclic constraints). 
\subsection{Signs of boundary components}
In this section we identify boundary components of the moduli spaces with fibre products of different moduli spaces. We compute the difference in orientation between the induced orientation on the boundary components, and the fibre product orientation.
First we consider the boundary components coming from $I \times_{D^2} \partial \mathcal{M}_{k+1,l}(\beta)$. Let $k = k_1 + k_2 + k_3$ and write $\mathcal{M}_1 := \mathcal{M}_{k_1 + k_3+2,l}(\beta_1)$, $\mathcal{M}_2 := \mathcal{M}_{k_2+1,l}(\beta_2)$. Finally write $\mathcal{M}_{j,\perp} = I \times_{D^2} \mathcal{M}_{j}$. Let $B_\perp$ be a boundary component where a disk bubbles off at the $(k_1 + 1)$-th boundary point, with $k_2$ of the boundary marked points and the interior marked points labelled by $J$. The boundary $I \times_{D^2} \partial \mathcal{M}_{k+1,l}(\beta)$ can be decomposed into two components, $B_{\perp,1}$, where the bubbling is not at the zeroth marked point and $B_{\perp,2}$, where the bubbling is at the zeroth marked point. 
\begin{nlemma}
	\label{orientation on B1} There exists diffeomorphisms: \begin{align}
		&\phi_1: \mathcal{M}_{1,\perp} \prescript{}{evb_{k_1 + 1}^{\beta_1}}{\times}_{evb_0^{\beta_2}} \mathcal{M}_2 \xrightarrow{\sim} B_{\perp,1}, \\
		&\phi_2: \mathcal{M}_{1} \prescript{}{evb_{k_1 + 1}^{\beta_1}}{\times}_{evb_0^{\beta_2}} \mathcal{M}_{2,\perp} \xrightarrow{\sim} B_{\perp,2}.
	\end{align}
	The maps $\phi_j$ change the orientation by $sign(\phi_j)$, where: \begin{align}
		sign(\phi_1) &= (-1)^{1+k_2k_3 +k_1 + n},\\
		sign(\phi_2) &= (-1)^{k_3(k_2 + 1) + n + 1}.
	\end{align}
\end{nlemma}
The proof of this lemma uses:
\begin{nprop}[{\cite[Proposition~2.8]{ST3}}]
	\label{boundary of general moduli}
	Let $k,l \in \mathbb{Z}_{\geq 0}$. Let $P \in S_3[k]$ and $\beta_1 + \beta_2 = \beta \in H_2(X,L)$. Let $I \sqcup J = [l]$ be a partition. Let $B \subset \partial \mathcal{M}_{k+1,l}(\beta)$ be the boundary component where a disk bubbles off at the $k_1 + 1$-th boundary point, with $k_2$ of the boundary marked points and the interior marked points labelled by $J$ lying on the bubble disk. Then the canonical diffeomorphism\begin{equation}
		\theta: \mathcal{M}_{k_1 + k_3+1,l}(\beta_1) \prescript{}{evb_{k_1 + 1}^{\beta_1}}{\times}_{evb_0^{\beta_2}} \mathcal{M}_{k_2+1,l}(\beta_2) \rightarrow B
	\end{equation} 
	changes orientation by the sign $(-1)^{\delta_1}$, with \begin{equation}
		\delta_1 = k_2k_3 + k_1 + n.
	\end{equation}
\end{nprop}
\begin{proof}[Proof of Lemma \ref{orientation on B1}]
	We can decompose $\phi_1$ as \begin{equation}
		\mathcal{M}_{1,\perp} \times_L \mathcal{M}_2 \rightarrow (I \times_{D^2} \mathcal{M}_1) \times_L \mathcal{M}_2 \xrightarrow{m_1} I \times_{D^2} (\mathcal{M}_1 \times_L \mathcal{M}_2) \xrightarrow{\hat{\theta}} I \times_{D^2} B \xrightarrow{t} B_{\perp,1}.
	\end{equation}
	Here $\hat{\theta}$ is the map induced by $\theta$ from the Proposition \ref{boundary of general moduli}. 
	By \cite[Lemma~8.2.3(4)]{FOOO}, we have $sign(\hat{\theta}) = sign(\theta)$. From Equation \eqref{decomposing boundary of horocyclic} it is clear that $sign(t) = -1$. Finally, from the associativity of the fibre product \cite[Lemma~8.2.3(2)]{FOOO}, $sign(m_1) = 1$. Thus \begin{equation}
		sign(\phi_1) = (-1)^{1+k_2k_3 +k_1 + n}.
	\end{equation}
Similarly, we decompose $\phi_2$ as:\begin{equation}
	\mathcal{M}_{1} \times_L \mathcal{M}_{2,\perp} \rightarrow \mathcal{M}_1 \times_L (I \times_{D^2}\mathcal{M}_2) \xrightarrow{m_2} I \times_{D^2} (\mathcal{M}_1 \times_L \mathcal{M}_2) \xrightarrow{\hat{\theta}} I \times_{D^2} B \xrightarrow{t} B_{\perp,2}.
\end{equation}
We can compute $sign(m_2)$ as follows:\begin{align}
	\mathcal{M}_1 \times_L (I \times_{D^2}\mathcal{M}_2) & = (-1)^{(k_2 + 1)(k_1+k_3)} (I \times_{D^2}\mathcal{M}_2) \times_L \mathcal{M}_1\\
	&= (-1)^{(k_2 + 1)(k_1+k_3+1)} I \times_{D^2} (\mathcal{M}_2 \times_L \mathcal{M}_1)\\
	&= (-1)^{(k_2 + 1)(k_1+k_3+1) + k_2(k_1+k_3))} I \times_{D^2} (\mathcal{M}_1 \times_L \mathcal{M}_2).
\end{align}
So that $sign(m_2) = (-1)^{k_1+k_3+1}$, and thus: \begin{equation}
	sign(\phi_2) = (-1)^{k_3(k_2 + 1) + n + 1}.
\end{equation}
\end{proof}

Now we consider the boundary components coming from $\partial I \times_{D^2} \mathcal{M}_{k+1,l}(\beta)$. First we set up some notation regarding the moduli spaces of holomorphic disks. We have \begin{equation}
\widetilde{\mathcal{M}}^{main}(\beta) = \{ u:(D^2,S^1) \rightarrow (X,L)| u \text{ is holomorphic and } u_*([D^2,S^1]) = \beta \in H_2(X,L) \}.
\end{equation}
The superscript $main$ here denotes that this is the top-dimensional stratum of the moduli space. The moduli spaces $\widetilde{\mathcal{M}}^{main}(\beta)$ are oriented using the relative spin structure as in \cite[Chapter~8]{FOOO}. Adding marked points and quotienting by $Aut(D^2)$ gives the moduli spaces $\mathcal{M}_{k+1,l}^{main}(\beta)$ as open subsets \begin{equation}
	\label{eq: orientation of moduli space of curves}
\mathcal{M}_{k+1,l}^{main}(\beta) \subset \left(\widetilde{\mathcal{M}}^{main}(\beta) \times (S^1)^{ k+1} \times (D^2)^{ l}\right)/Aut(D^2).
\end{equation}
Here we need to be careful about the ordering in $(S^1)^{k+1}$. We stick to the convention in \cite{FOOO}, so that $(S^{1})^{k+1} = S^1_0 \times S^1_1 \times \dots \times S^1_{k}$. Here $S^1_i$ is the circle corresponding to the i'th boundary marked point. The orientation of a quotient by a Lie group is defined as in \cite{FOOO}. This means that the orientation on $\mathcal{M}_{k+1,l}(\beta)$ is such that there exists an orientation preserving local diffeomorphism: \begin{equation}
	\label{orientation of moduli space}
\mathcal{M}^{main}_{k+1,l}(\beta) \times Aut(D^2) \cong \widetilde{\mathcal{M}}^{main}(\beta) \times (S^1)^{k+1} \times (D^2)^{ l}.
\end{equation}
\begin{nlemma}
	\label{orientation fixed interior and boundary marked points}
	Let $k\geq 1$ and $l \geq 1$. Fixing the 0th boundary marked point at $-i$, and the first interior marked point at $0$ defines a local diffeomorphism: \begin{equation}
	\mathcal{M}^{main}_{k+1,l}(\beta) \rightarrow \widetilde{\mathcal{M}}^{main}(\beta) \times (S^1)^{ k} \times (D^2)^{ l-1},
	\end{equation}
	which changes orientation by $(-1)^k$.
\end{nlemma}
\begin{proof}
	Recall that in \cite{FOOO} the orientation on $Aut(D^2)$ is given by considering the local diffeomorphism: 	\begin{align}
		Aut(D^2)  &\rightarrow (S^1)^{3}\nonumber\\
		g &\mapsto (gz_0, g z_1, gz_2),
	\end{align}
	for three points $z_0, z_1, z_2 \subset S^1$ in counter-clockwise ordering. By definition, this map is orientation preserving. One can check that the map
	\begin{align}
		Aut(D^2)  &\rightarrow S^1 \times D^2\nonumber\\
		g &\mapsto (g\cdot -i, g \cdot 0),
	\end{align}
	is also orientation preserving. Now multiply both sides by $\widetilde{\mathcal{M}}^{main}(\beta) \times (S^1)^k \times (D^2)^{l-1}$ on the left, and commute the various terms through to obtain a local diffeomorphism:
	 \begin{equation}
		\widetilde{\mathcal{M}}^{main}(\beta) \times (S^1)^{ k} \times (D^2)^{ l-1} \times Aut(D^2) \cong (-1)^k \widetilde{\mathcal{M}}^{main}(\beta) \times (S^1)^{ k+1} \times (D^2)^{ l}.
	\end{equation}
	The sign $(-1)^k$ here comes from the change in ordering from $(S^1)^k \times S^1_0 \cong (-1)^k (S^1)^{k+1}$. Finally, apply Equation \eqref{orientation of moduli space} and cancel the factor $Aut(D^2)$ to obtain the result.
\end{proof} 
Similar considerations show:
\begin{nlemma}
	\label{orientation fixed boundary marked points}
	Let $k = k_1 + k_2 + k_3$. Fixing the three boundary marked points with indices $0, k_1+1$ and $k_1 + k_2 +2$, we obtain a local diffeomorphism: \begin{equation}
		\mathcal{M}_{k+3,l}^{main}(\beta) \rightarrow \widetilde{\mathcal{M}}^{main}(\beta) \times (S^1)^{k} \times (D^2)^{l-1},
	\end{equation}
	which changes orientation by $(-1)^{k + k_2} = (-1)^{k_1 + k_3}$.
\end{nlemma} 
\begin{nlemma}
	\label{orientation fixed interior marked points}
	Fixing the first three marked points at $0,1,\infty$ gives an orientation preserving local diffeomorphism:  \begin{equation}
		\mathcal{M}_{\emptyset,l_1+1}^{main}(\beta) \rightarrow \widetilde{\mathcal{M}}^{main}_{\emptyset}(\beta) \times (S^2)^{l_1 - 2}.
	\end{equation}
\end{nlemma}

We now want to study the boundary components $\partial I \times_{D^2} \mathcal{M}_{k+1,l}(\beta)$. Observe that:
 \begin{equation}
	\label{two interior boundary components}
	\partial I \times_{D^2} \mathcal{M}_{k+1,l}(\beta) = \{1\} \times_{D^2} \mathcal{M}_{k+1,l}(\beta) - \{ 0\}\times_{D^2} \mathcal{M}_{k+1,l}(\beta).
\end{equation}
First we look at the case where the two interior marked points collide. This corresponds to $\{ 0 \} \times_{D^2} \mathcal{M}_{k+1,l}(\beta)$. Let $B_{\perp,3}$ be a boundary component where the interior marked points labelled by $I$ bubble off on a sphere. Note that $1,2 \in I$. Together with the output marked point on the sphere, this gives at least 3 marked points. 

For gluing the moduli spaces of holomorphic maps $\widetilde{\mathcal{M}}^{main}_{\emptyset}(\beta_1)$ and $\widetilde{\mathcal{M}}^{main}(\beta_2)$, we use the following: 
\begin{nprop}[Lemma \ref{gluing at interior points}]
	\label{gluing disk and sphere tangent}
	The gluing map \begin{equation}
		\widetilde{\mathcal{M}}^{main}_{\emptyset}(\beta_1) \times_{X} \widetilde{\mathcal{M}}^{main}(\beta_2) \rightarrow \widetilde{\mathcal{M}}^{main}(\beta_1 + \beta_2)
	\end{equation}
	is a local diffeomorphism which changes orientation by $(-1)^{w_{\mathfrak{s}}(\beta_1)}$.
\end{nprop}
\begin{remark}
	This proposition is implicit in \cite{ST2} and the statement was communicated to the author by Sara Tukachinsky. See also \cite[Remark~2.7]{GeoZ}. As far as the author is aware, the proof of this statement has not appeared in any literature before. We thus prove it in Appendix \ref{orientation properties}.
\end{remark}
We then prove:
\begin{nprop}
	The canonical local diffeomorphism
	\begin{equation}
	\phi_3: \mathcal{M}_{\emptyset,l_1+1}(\beta_1) \times_X \mathcal{M}_{k+1,l_2+1}(\beta_2) \xrightarrow{\sim} B_{\perp,3} \subset \partial \mathcal{M}_{k+1,l; \perp_0}(\beta)
	\end{equation}
	changes orientation by $sign(\phi_3) = (-1)^{1+w_{\mathfrak{s}}(\beta_1)}$. Here $\beta = \beta_1 + \beta_2$.
\end{nprop}
\begin{proof}
	Let $(v,u)$ be a stable map, where $v: S^2 \rightarrow X$, $u: (D^2,S^1) \rightarrow (X,L)$ and $evi_0(v) = evi_0(u) = x \in X$. We will compute the change in orientation locally at $(v,u)$. By definition of the fibre product orientation, we have:
	\begin{align}
	&T_{(0,(v,u))} \{ 0 \} \times_{D^2} \mathcal{M}_{k+1,l}(\beta) \oplus T_0D^2 \cong T_u \MM_{k+1,l}(\beta).\\
	\intertext{by Lemma \ref{orientation fixed interior and boundary marked points} we have:}
	&T_u \MM_{k+1,l}(\beta) \oplus T_xX \cong (-1)^kT_{(v,u)}\widetilde{\mathcal{M}}(\beta) \oplus \mathbb{R}^{k} \oplus \mathbb{C}^{l - 1} \oplus T_xX.\\
	\intertext{Now use Proposition \ref{gluing disk and sphere tangent} to rewrite this as:}
	&\cong (-1)^{w_{\mathfrak{s}}(\beta_1)+k} T_{(v,u)}\widetilde{\mathcal{M}}_{\emptyset}(\beta_1) \times_X \widetilde{\mathcal{M}}(\beta_2)\oplus T_xX \oplus \mathbb{R}^{k} \oplus \mathbb{C}^{l - 1}.\\
	\intertext{By definition of the fibre product orientation, this is isomorphic to}
	&\cong (-1)^{w_{\mathfrak{s}}(\beta_1)+k} T_{v}\widetilde{\mathcal{M}}_{\emptyset}(\beta_1) \oplus T_u \widetilde{\mathcal{M}}(\beta_2) \oplus \mathbb{R}^{k} \oplus \mathbb{C}^{l - 1}.\\
	\intertext{Next, we rearrange the terms, to obtain:}
	&\cong (-1)^{w_{\mathfrak{s}}(\beta_1)+k} (T_{v}\widetilde{\mathcal{M}}_{\emptyset}(\beta_1) \oplus \CC^{l_1 -2}) \oplus (T_u \widetilde{\mathcal{M}}(\beta_2) \oplus \mathbb{R}^{k} \oplus \mathbb{C}^{l_2}) \oplus \CC.\\
	\intertext{By Lemmas \ref{orientation fixed interior and boundary marked points} and \ref{orientation fixed interior marked points}, this is isomorphic to:}
	&\cong (-1)^{w_{\mathfrak{s}}(\beta_1)} T_{v}\mathcal{M}_{\emptyset,l_1+1}(\beta_1) \oplus T_u \mathcal{M}_{k,l_2}(\beta_2) \oplus \CC.\\
	\intertext{Again, by definition of the fibre product orientation, this is isomorphic to:}
	&\cong (-1)^{w_{\mathfrak{s}}(\beta_1)} T_{(v,u)}\mathcal{M}_{\emptyset,l_1+1}(\beta_1) \otimes_X T_u \mathcal{M}_{k,l_2}(\beta_2) \oplus T_xX \oplus \CC.
\end{align}
	Thus, as $T_0\mathbb{D} \cong \CC$, and cancelling the terms $T_xX$, we obtain: \begin{equation}
	T_{(0,(v,u))} \{ 0 \} \times_{D^2} \mathcal{M}_{k+1,l}(\beta) \cong (-1)^{w_{\mathfrak{s}}(\beta_1)} T_{(v,u)}\mathcal{M}_{\emptyset,l_1+1}(\beta_1) \times_X T_u \mathcal{M}_{k,l_2}(\beta_2)
	\end{equation}
	The extra change in sign then comes from Equation \eqref{two interior boundary components}.
\end{proof}
Next we consider the case when one of the horocyclicly constrained points collides with the boundary marked point. This corresponds to $\{1 \} \times_{D^2} \mathcal{M}_{k+1,l}(\beta)$. Here two disks bubble off on either side of the disk. Let $B_{\perp,4}$ be this boundary component. We show:

\begin{nprop}
	\label{sign of phi 4}
	The map \begin{equation}
	\phi_4: \mathcal{M}_{k_4+1, l_3}(\beta_3) \times_L \mathcal{M}_{k_1+k_3+k_5+3,l_1}(\beta_1) \times_L \mathcal{M}_{k_2+1,l_2}(\beta_2) \rightarrow B_{\perp,4} \subset \partial \mathcal{M}_{k+1,l,\perp_0}(\beta)
	\end{equation}
	changes orientation with $sign(\phi_4) = k_4(k_1 + k_2 + k_3) + k_2(k_3+k_5) +k_3$.
\end{nprop}

Applying \cite[Lemma~8.3.5]{FOOO} twice shows:
 \begin{nlemma}
	\label{gluing three disks tangent}
	The gluing map: \begin{equation}
	\theta: \widetilde{\mathcal{M}}(\beta_3) \times_L \widetilde{\mathcal{M}}(\beta_1) \times_L \widetilde{\mathcal{M}}(\beta_2) \rightarrow \widetilde{\mathcal{M}}(\beta),
	\end{equation}
	is an orientation preserving local diffeomorphism. Here $\beta = \beta_1 + \beta_2 + \beta_3$.
\end{nlemma}

\begin{proof}[Proof of \ref{sign of phi 4}]
	Let $u = (u_1,u_2,u_3) \in B_{\perp,4}$ be a stable map. For simplicity, write $\MM_1 = \mathcal{M}_{k_1+k_3+k_5+3,l_1}$, $\MM_2 = \mathcal{M}_{k_2+1, l_2}(\beta_2)$ and $\mathcal{M}_{k_4+1, l_3}(\beta_3)$.
	We first note that by definition of the fibre product orientation:
	\begin{equation}
		T_{(1,u)} \{1\} \times_{D^2} \mathcal{M}_{k+1,l}(\beta) \oplus T_1D^2 \cong T_u \mathcal{M}_{k+1,l}(\beta).
	\end{equation}
	We then use Lemma \ref{orientation fixed interior and boundary marked points} to write:
	 \begin{align}
		&T_u\mathcal{M}_{k+1,l}(\beta) \oplus TL \oplus TL \cong (-1)^k \widetilde{\mathcal{M}}(\beta) \oplus \mathbb{R}^{k} \oplus \mathbb{C}^{l-1} \oplus TL \oplus TL,\\
		\intertext{which, by Proposition \ref{gluing three disks tangent} is isomorphic to:}
		&\cong (-1)^k T_u (\widetilde{\mathcal{M}}(\beta_3) \times_L \widetilde{\mathcal{M}}(\beta_1) \times_L \widetilde{\mathcal{M}}(\beta_2)) \oplus \mathbb{R}^{k} \oplus \mathbb{C}^{l-1} \oplus TL \oplus TL.\\
		\intertext{By applying the definition of the fibre product orientation twice, this is isomorphic to:}
		&\cong (-1)^{k+n} T_{u_3} \widetilde{\mathcal{M}}(\beta_3) \oplus T_{u_1}\widetilde{\mathcal{M}}(\beta_1) \oplus T_{u_2} \widetilde{\mathcal{M}}(\beta_2) \oplus \mathbb{R}^{k} \oplus \mathbb{C}^{l-1}.\\
		\intertext{Commuting the various terms through, noting that $\mathbb{C}$ is even dimensional, and $\mathbb{R}$ is odd dimensional, this gives:}
		&\begin{multlined}[t]
			\cong (-1)^{k+n + A} (T_{u_3} \widetilde{\mathcal{M}}(\beta_3) \oplus \RR^{k_4} \oplus \CC^{l_3-1}) \oplus (T_{u_1}\widetilde{\mathcal{M}}(\beta_1) \oplus \RR^{k_1+k_3 + k_5} \oplus \CC^{l_1}) \\ \oplus (T_{u_2} \widetilde{\mathcal{M}}(\beta_2) \oplus \mathbb{R}^{k_2} \oplus \mathbb{C}^{l_2-1}) \oplus \CC,
		\end{multlined}\\
		\intertext{where $A = k_4(k_1+k_2+k_3) + n(k_1+k_3+k_5) + k_2(k_3+k_5)$. Then apply Lemmas \ref{orientation fixed interior and boundary marked points} and \ref{orientation fixed boundary marked points} to find:}
		&\cong (-1)^{k+n + A + B} T_{u_3} T_{u_3}\MM_3 \oplus T_{u_1} \MM_1 \oplus T_{u_2} \MM_2 \oplus \CC,\\
		\intertext{where $B = k_1 + k_2 + k_4 + k_5$. Finally, apply the definition of the fibre product orientation twice to obtain:}
		& \cong (-1)^{k+n + A+ B + n(k_1+k_3+k_5 + n)} T_u (\MM_3 \times_L \MM_1 \times_L \MM_2) \oplus TL \oplus TL \oplus \CC.
	\end{align}
	The result then follows by cancelling the factors $TL$ and noting that $T_1D^2 \cong \CC$.
\end{proof}

\subsection{Stokes' theorem and push-forward}
The next step is to apply Stokes' theorem for the push-forward of differential forms. 
\begin{nthm}[{\cite[Theorem~1]{ST4}}]
	\label{Stokes theorem}
	Let $M$ be a smooth orbifold with boundary. For a smooth submersion $f: M\rightarrow N$ of relative dimension $s$, and $\xi \in A^t(M)$. We have:
	\begin{equation}
	0 = f_*(d\xi) - d(f_*\xi) + (-1)^{s+t}(f|_{\partial M})_*\xi.
	\end{equation}
\end{nthm}
Recall the following facts about the push-forward of differential forms, see \cite{ST4}:
\begin{nlemma}
	\begin{enumerate}\text{ }
		\item Let $f:M\rightarrow N$ be a proper submersion, $\alpha \in A^*(N)$, $\beta \in A^*(M)$. Then:\begin{equation}
			\label{push-pull relation}
			f_*(f^*\alpha \wedge \beta) = \alpha \wedge f_*\beta.
		\end{equation}
		\item Let \begin{equation}\begin{tikzcd}
			M \times_N P \arrow[d,"q"] \arrow[r,"p"] &P \arrow[d,"g"]\\
			M \arrow[r,"f"] & N
		\end{tikzcd}
	\end{equation}
		be a pull-back diagram of smooth maps, where $g$ and $f$ are proper submersions. Let $\alpha \in A^*(P)$. Then: \begin{equation}
		\label{commute push-pull}
		q_*p^*\alpha = f^*g_*\alpha.
		\end{equation}
		Similarly, if $\beta \in A^*(M)$, then: \begin{equation}
		p_*q^*\beta = (-1)^{(dim(M)-dim(N))(dim(P)-dim(N))}g^*f_*\beta.
		\end{equation}
	\end{enumerate}
\end{nlemma}

To obtain the structure equations for the $\mathfrak{q}^{ST}_{\perp}$ operations, we will apply Stokes' theorem with $M = \mathcal{M}_{k+1,l, \perp_0}(\beta)$, $N = L$, $f = evb_0$ and $\xi = \bigwedge_{i=1} evi^*\gamma_i \wedge \bigwedge_{j=1} evb_j^*\alpha_j$. When it is clear which evaluation maps are used, we will simply write $evi^*\gamma$ for $\bigwedge_{i=1} evi^*\gamma_i$ and similarly for the boundary evaluations.

The first term in Stokes' theorem is:
\begin{align}
	(evb_0)_*(d\xi) &= (evb_0)_*\bigg(\sum_{\substack{S_3[l]\\ (2:3) = \{ i \}}} (-1)^{|\gamma^{(1:3)}|} evi^*(\gamma^{(1:3)} \wedge d\gamma_i \wedge \gamma^{(3:3)})\wedge evb^*\alpha \\
	&\qquad\quad+\sum_{\substack{S_3[k]\\ k_2 = 1}} (-1)^{|\gamma| + \epsilon_1 + k_1}evi^*\gamma \wedge evb^* (\alpha^{(1:3)} \wedge d\alpha_{k_1+1} \wedge \alpha^{(3:3)})\bigg),\\
	\intertext{by definition of the $\mathfrak{q}$ operations, this equals:}
	&= \sum_{\substack{S_3[l]\\ (2:3) = \{ i \}}} (-1)^{|\gamma^{(1:3)}|+\zeta(\alpha) + \zeta_{\perp}(\alpha,d\gamma)} \mathfrak{q}^{ST,\beta}_{k,l;\perp_0}(\otimes_{j=1}^{k} \alpha_j; \gamma^{(1:3)} \otimes d\gamma_i \otimes \gamma^{(3:3)})\\
	&\qquad \quad + \sum_{\substack{S_3[k]\\ k_2 = 1}} (-1)^{|\gamma|+\zeta(\alpha)+1+\epsilon_1 + \zeta_{\perp}(d\alpha,\gamma)} \mathfrak{q}^{ST,\beta}_{k,l;\perp_0}(\alpha^{(1:3)} \otimes \mathfrak{q}^{ST, \beta_0}_{1,0}(\alpha_{k_1+1}) \otimes \alpha^{(3:3)}; \gamma),\\
	\intertext{by expanding the signs $\zeta_{\perp}$, we find:}
	&= \sum_{\substack{S_3[l]\\ (2:3) = \{ i \}}} (-1)^{n+ |\gamma| + \zeta(\alpha) + \epsilon(\alpha) + |\gamma^{(1:3)}| + 1} \mathfrak{q}^{ST, \beta}_{k,l;\perp_0}(\otimes_{j=1}^{k} \alpha_j; \gamma^{(1:3)} \otimes d\gamma_i \otimes \gamma^{(3:3)})\\
	&\qquad \quad + \sum_{\substack{S_3[k]\\ k_2 = 1}} (-1)^{n+|\gamma|+\zeta(\alpha) + \epsilon(\alpha) +\epsilon_1 + |\gamma|} \mathfrak{q}^{ST,\beta}_{k,l;\perp_0}(\alpha^{(1:3)} \otimes \mathfrak{q}^{ST,\beta_0}_{1,0}(\alpha_{k_1+1}) \otimes \alpha^{(3:3)}; \gamma).
	\label{combine with C1}
\end{align}
The second term in Stokes' theorem reads: \begin{align}
	-d((evb_0)_*\xi) = -\mathfrak{q}^{ST,\beta_0}_{1,0}((evb_0)_*\xi) &= (-1)^{\zeta(\alpha) + 1 + \zeta_{\perp}(\alpha,\gamma)} \mathfrak{q}^{ST,\beta_0}_{1,0}(\mathfrak{q}^{ST,\beta}_{k,l,\perp_0}(\alpha,\gamma))\\
	&=(-1)^{n + |\gamma| + \zeta(\alpha) + \epsilon(\alpha) + 1} \mathfrak{q}_{1,0}^{ST,\beta_0}(\mathfrak{q}^{ST,\beta}_{k,l,\perp_0}(\alpha,\gamma)).
	\label{combine with C2}
\end{align}
The final term in Stokes' theorem is given by restricting to the various boundary components $B_{\perp,i}$ for $i = 1,2,3,4$. We first compute the overall sign $(-1)^{s+t}$. Note that $|\xi| = |\gamma| + |\alpha|$, $dim(M) \equiv k+1 \;(mod \; 2)$ and $dim(L) = n$, so $s \equiv k+1+n$. The overall sign is thus \begin{equation}
	s+t \equiv |\gamma| + \epsilon(\alpha) + n + 1\; (mod \; 2).
\end{equation}
\subsubsection{$B_{\perp,1}$}
Here we will show: \begin{nlemma}
	The terms in Stokes' theorem coming from $(evb_0|_{B_{1,\perp}})_*\xi$ are \begin{equation}
		\sum_{\substack{ P \in S_3[k]\\ \beta_1 + \beta_2 = \beta\\ \beta_2 \neq 0\\ J_1 \cup J_2 = [l]\\\gamma_1,\gamma_2 \in J_1}} (-1)^{C_1} \mathfrak{q}_{k_1+1+k_3,l_1,\perp}^{ST,\beta_1}(\alpha^{(1:3)} \otimes \mathfrak{q}^{ST,\beta_2}_{k_2,l_2}(\alpha^{(2:3)};\gamma^{J_2}) \otimes \alpha^{(3:3)}; \gamma^{J_1}),
	\end{equation}
	where $C_1 = n+ |\gamma| + \zeta(\alpha) + \epsilon(\alpha) + sign^\gamma(J_1,J_2) + (|\gamma^{J_2}|+1)\epsilon_1 + |\gamma^{J_1}|$.
\end{nlemma}
Recall that \begin{equation}
\phi_1: \mathcal{M}_{k_1 + k_3+2,l_1,\perp_0}(\beta_1) \prescript{}{evb_{k_1 + 1}^{\beta_1}}{\times}_{evb_0^{\beta_2}} \mathcal{M}_{k_2+1,l_2}(\beta_2) \xrightarrow{\sim} B_{\perp,1}
\end{equation}
changes orientation by $\delta_1 := sign(\phi_1) = n + k_1 + k_2k_3$.

Denote everything associated with $\mathcal{M}_{k_1 + k_3+2,l_1,\perp_0}(\beta_1)$ with a subscript $1$ and everything associated with $\mathcal{M}_{k_2+1,l_2}(\beta_2)$ with a subscript $2$. Consider the commutative diagram: \begin{equation}
	\begin{tikzcd}
		\mathcal{M}_{1,\perp} \times_L \mathcal{M}_2 \arrow[d,"p_1"] \arrow[r,"p_2"] &\mathcal{M}_2 \arrow[d,"evb^2_0"]\\
		\mathcal{M}_{1,\perp}\arrow[r,"evb^1_{k_1+1}"] & L
	\end{tikzcd}
\end{equation}
Let \begin{align}
	\overline{\xi} &= \phi_1^*\xi,\\
	\xi_1 &= (evi^1)^*\gamma^{J_1} \wedge (evb^1)^*(\alpha^{(1:3)}\wedge\alpha^{(3:3)}),\\
	\xi_2 &= (evi^2)^*\gamma^{J_2} \wedge (evb^2)^*\alpha^{(2:3)}.
\end{align}
Define $\delta_2$ by \begin{equation}
	p_1^*\xi_1 \wedge p_2^*\xi = (-1)^{\delta_2}\overline{\xi}.
\end{equation}
We thus find: \begin{equation}
	\delta_2 = sign^\gamma(J_1,J_2) + |\gamma^{J_2}|(\epsilon_1 + k_1 + \epsilon_3 + k_3) + (\epsilon_2+k_2)(\epsilon_3 + k_3).
\end{equation}
Now compute \begin{align}
	(evb_0|_{B_{\perp,1}})_*\xi &= (-1)^{\delta_1+\delta_2} (evb^1_0)_*(p_1)_*(p_1^*\xi_1 \wedge p_2^*\xi_2),\\
	\intertext{by using the relation \ref{push-pull relation}, this equals:}
	& = (-1)^{\delta_1+\delta_2} (evb^1_0)_*(\xi_1 \wedge (p_1)_*(p_2)^*\xi_2),\\
	\intertext{using Equation \eqref{commute push-pull} we obtain:}
	& = (-1)^{\delta_1+\delta_2} (evb^1_0)_*(\xi_1 \wedge (evb^1_{k_1+1})^*(evb_0^2)_* \xi_2),\\
	\intertext{which, by definition of $\mathfrak{q}^{ST}$ equals:}
	& = (-1)^{\delta_1+\delta_2 + \delta_3} (evb^1_0)_*(\xi_1 \wedge \mathfrak{q}^{ST,\beta_2}_{k_2,l_2}(\alpha^{(2:3)};\gamma^{J_2})),\\
	\intertext{where $\delta_3 = \zeta(\alpha^{(2:3)})$. Expanding $\xi_1$ and rearranging gives:}
	& = (-1)^{\delta_1+\delta_2 +\delta_3 + \delta_4} (evb^1_0)_*((evi^1)^*\gamma^{J_1} \wedge (evb^1)^*(\alpha^{(1:3)} \wedge \mathfrak{q}^{ST,\beta_2}_{k_2,l_2}(\alpha^{(2:3)};\gamma^{J_2}) \wedge \alpha^{(3:3)})),\\
	\intertext{where $\delta_4 = (|\gamma^{J_2}| + \epsilon_2)(\epsilon_3 + k_3)$. Finally by definition this equals:}
	& = (-1)^{\delta_1+\delta_2 +\delta_3 + \delta_4 + \delta_5} \mathfrak{q}_{k_1+1+k_3,l_1,\perp}^{ST,\beta_1}(\alpha^{(1:3)} \otimes \mathfrak{q}^{ST,\beta_2}_{k_2,l_2}(\alpha^{(2:3)};\gamma^{J_2}) \otimes \alpha^{(3:3)}; \gamma^{J_1}),
\end{align}
where $\delta_5 = \zeta(\alpha^{(1:3)},\mathfrak{q}^{ST}(\alpha^{(2:3)};\gamma^{J_2}),\alpha^{(3:3)}) + \zeta_{\perp}(\alpha^{(1:3)},\mathfrak{q}^{ST}(\alpha^{(2:3)};\gamma^{J_2}),\alpha^{(3:3)})$. Adding all signs together with the sign in Stokes' theorem, we get an overall sign: \begin{equation}
	C_1 = n+ |\gamma| + \zeta(\alpha) + \epsilon(\alpha) + sign^\gamma(J_1,J_2) + (|\gamma^{J_2}|+1)\epsilon_1 + |\gamma^{J_1}|.
\end{equation}
Here we have used \cite[Lemma~2.9]{ST3} to compute \begin{equation}\zeta(\alpha^{(2:3)}) + \zeta(\alpha^{(1:3)},\mathfrak{q}^{ST}_{k_2,l_2}(\alpha^{(2:3)};\gamma^{J_2}),\alpha^{(3:3)}) = \zeta(\alpha) + \epsilon(\alpha) + \epsilon_1 + (k_1+1)|\gamma^{J_2}| + k_3k_2.
\end{equation}

\subsubsection{$B_{\perp,2}$}
We repeat the above argument for the boundary component $B_{\perp,2}$. We show: \begin{nlemma}
	The terms in Stokes' theorem coming from $(evb_0|_{B_{2,\perp}})_*\xi$ are \begin{equation}
		\sum_{\substack{ P \in S_3[k]\\ \beta_1 + \beta_2 = \beta\\ \beta_1 \neq 0\\ J_1 \cup J_2 = [l]\\\gamma_1,\gamma_2 \in J_2}} (-1)^{C_2} \mathfrak{q}_{k_1+1+k_3,l_1}^{ST,\beta_1}(\alpha^{(1:3)} \otimes \mathfrak{q}^{ST,\beta_2}_{k_2,l_2,\perp}(\alpha^{(2:3)};\gamma^{J_2}) \otimes \alpha^{(3:3)}; \gamma^{J_1}),
	\end{equation}
	where $	C_2 = n+ |\gamma| + \zeta(\alpha) + \epsilon(\alpha) + sign^\gamma(J_1,J_2) + |\gamma^{J_2}|\epsilon_1 + 1$.
\end{nlemma}
Recall that \begin{equation}
\phi_2: \mathcal{M}_{k_1 + k_3+2,l_1}(\beta_1) \prescript{}{evb_{k_1 + 1}^{\beta_1}}{\times}_{evb_0^{\beta_2}} \mathcal{M}_{k_2+1,l_2,\perp_0}(\beta_2) \xrightarrow{\sim} B_{\perp,2}
\end{equation}
changes orientation by $\delta_1 := sign(\phi_1) = n + k_3 + k_2k_3 + 1$.

Denote everything associated with $\mathcal{M}_{k_1 + k_3+2,l_1}(\beta_1)$ with a subscript $1$, everything associated with $\mathcal{M}_{k_2+1,l_2,\perp_0}(\beta_2)$ with a subscript $2$. Consider the commutative diagram: \begin{equation}
	\begin{tikzcd}
		\mathcal{M}_{1} \times_L \mathcal{M}_{2,\perp} \arrow[d,"p_1"] \arrow[r,"p_2"] &\mathcal{M}_{2,\perp} \arrow[d,"evb^2_0"]\\
		\mathcal{M}_{1}\arrow[r,"evb^1_{k_1+1}"] & L
	\end{tikzcd}
\end{equation}
Let \begin{align}
	 \overline{\xi} &= \phi_2^*\xi,\\
	\xi_1 &= (evi^1)^*\gamma^{J_1} \wedge (evb^1)^*(\alpha^{(1:3)}\wedge\alpha^{(3:3)}),\\
	\xi_2 &= (evi^2)^*\gamma^{J_2} \wedge (evb^2)^*\alpha^{(2:3)}.
\end{align}
$\delta_2$ is as before: \begin{equation}
	p_1^*\xi_1 \wedge p_2^*\xi = (-1)^{\delta_2}\overline{\xi}.
\end{equation}

Now compute \begin{align}
	(evb_0|_{B_{\perp,2}})_*\xi &= (-1)^{\delta_1+\delta_2} (evb^1_0)_*(p_1)_*(p_1^*\xi_1 \wedge p_2^*\xi_2)\\
	& = (-1)^{\delta_1+\delta_2} (evb^1_0)_*(\xi_1 \wedge (p_1)_*(p_2)^*\xi_2)\\
	& = (-1)^{\delta_1+\delta_2} (evb^1_0)_*(\xi_1 \wedge (evb^1_{k_1+1})^*(evb_0^2)_* \xi_2)\\
	& = (-1)^{\delta_1+\delta_2 + \delta_3} (evb^1_0)_*(\xi_1 \wedge \mathfrak{q}^{ST,\beta_2}_{k_2,l_2,\perp}(\alpha^{(2:3)};\gamma^{J_2})),\\
	\intertext{where $\delta_3 = \zeta(\alpha^{(2:3)}) + \zeta_{\perp}(\alpha^{(2:3)}; \gamma^{J_2})$. Expanding $\xi_1$ and rearranging gives:}
	& = (-1)^{\delta_1+\delta_2 +\delta_3 + \delta_4} (evb^1_0)_*((evi^1)^*\gamma^{J_1} \wedge (evb^1)^*(\alpha^{(1:3)} \wedge \mathfrak{q}^{ST,\beta_2}_{k_2,l_2,\perp}(\alpha^{(2:3)};\gamma^{J_2}) \wedge \alpha^{(3:3)})),\\
	\intertext{where $\delta_4 = (|\gamma^{J_2}| + \epsilon_2 + 1)(\epsilon_3 + k_3)$. By definition of the $\mathfrak{q}^{ST}$ operations, this equals:}
	& = (-1)^{\delta_1+\delta_2 +\delta_3 + \delta_4 + \delta_5} \mathfrak{q}_{k_1+1+k_3,l_1,\perp}^{ST,\beta_1}(\alpha^{(1:3)} \otimes \mathfrak{q}^{ST,\beta_2}_{k_2,l_2,\perp}(\alpha^{(2:3)};\gamma^{J_2}) \otimes \alpha^{(3:3)}; \gamma^{J_1}),
\end{align}
where $\delta_5 = \zeta(\alpha^{(1:3)},\mathfrak{q}^{ST}_{k_2,l_2,\perp_0}(\alpha^{(2:3)};\gamma^{J_2}),\alpha^{(3:3)})$. Adding all signs together with the sign in Stokes' theorem \ref{Stokes theorem}, we get an overall sign: \begin{equation}
	C_2 = n+ |\gamma| + \zeta(\alpha) + \epsilon(\alpha) + sign^\gamma(J_1,J_2) + |\gamma^{J_2}|\epsilon_1 + 1.
\end{equation}

\subsubsection{$B_{\perp,3}$}
We show: \begin{nlemma}
	The terms in Stokes' theorem coming from $(evb_0|_{B_{3,\perp}})_*\xi$ are \begin{equation}
		\sum_{\substack{\beta_1 + \beta_2 = \beta\\  J_1 \cup J_2 = [l]\\\gamma_1,\gamma_2 \in J_1}} (-1)^{C_3} \mathfrak{q}^{ST}_{k,l_2}(\alpha;\mathfrak{q}^{ST}_{\emptyset,l_1}(\gamma^{J_1}) \otimes \gamma^{J_2}),
	\end{equation}
	where $C_3 =  n + |\gamma| +\zeta(\alpha)  + \epsilon(\alpha) + sign^\gamma(J_1,J_2)$.
\end{nlemma}
Recall that \begin{equation}
\overline{\phi}_3: \mathcal{M}_{\emptyset,l_1+1}(\beta_1) \times_X \mathcal{M}_{k+1,l_2+1}(\beta_2) \xrightarrow{\sim} B_{\perp,3} \subset \partial \mathcal{M}_{k+1,l; \perp_0}(\beta)
\end{equation}
changes orientation by $\delta_1 := sign(\phi_3) = w_{\mathfrak{s}}(\beta_1)+1$.

Denote everything associated with $\mathcal{M}_{\emptyset,l_1+1}(\beta_1)$ with a subscript $1$, and those for $\mathcal{M}_{k+1,l_2+1}(\beta_2)$ with a subscript $2$. Consider the commutative diagram: \begin{equation}
	\begin{tikzcd}
		\mathcal{M}_{1} \times_X \mathcal{M}_{2} \arrow[d,"p_1"] \arrow[r,"p_2"] &\mathcal{M}_{2} \arrow[d,"evi^2_1"]\\
		\mathcal{M}_{1}\arrow[r,"evi^1_{0}"] & X
	\end{tikzcd}
\end{equation}
Let \begin{align}
	\overline{\xi} &= \phi_3^*\xi,\\
	\xi_1 &= (evi^1)^*\gamma^{J_1},\\
	\xi_2 &= (evi^2)^*\gamma^{J_2} \wedge (evb^2)^*\alpha.
\end{align}
$\delta_2$ is defined by: \begin{equation}
	p_1^*\xi_1 \wedge p_2^*\xi = (-1)^{\delta_2}\overline{\xi},
\end{equation}
so that: \begin{equation}
	\delta_2 = sign^\gamma(J_1,J_2).
\end{equation}

Now compute \begin{align}
	(evb_0|_{B_{\perp,3}})_*\xi &= (-1)^{\delta_1+\delta_2} (evb^2_0)_*(p_2)_*(p_1^*\xi_1 \wedge p_2^*\xi_2)\\
	& = (-1)^{\delta_1+\delta_2+\delta_3} (evb^2_0)_*((p_2)_*(p^1)^*\xi_1 \wedge \xi_2),\\
	\intertext{where $\delta_3 = reldim(p_2)|\xi_1| \equiv 0\; (mod \; 2)$. Using the commutative diagram, we find:}
	& = (-1)^{\delta_1+\delta_2} (evb^2_0)_*((evi^2_{1})^*(evi_0^1)_*\xi_1 \wedge \xi_2)\\
	& = (-1)^{\delta_1+\delta_2+\delta_4} (evb^2_0)_*((evi^2_{1})^*(\mathfrak{q}^{ST,\beta_1}_{\emptyset,l_1}(\gamma^{J_1}) \wedge \xi_2)),\\
	\intertext{where $\delta_4 = w_{\mathfrak{s}}(\beta_1)$. By definition this equals:}
	& = (-1)^{\delta_1+\delta_2 + \delta_4+\delta_5} \mathfrak{q}^{ST}_{k,l_2}(\alpha;\mathfrak{q}^{ST}_{\emptyset,l_1}(\gamma^{J_1}) \otimes \gamma^{J_2}),
\end{align}
where $\delta_5 = \zeta(\alpha)$. Adding all signs together with the sign in Stokes' theorem \ref{Stokes theorem}, we get an overall sign: \begin{equation}
	C_3 =  n + |\gamma| +\zeta(\alpha)  + \epsilon(\alpha) + sign^\gamma(J_1,J_2).
\end{equation}
\subsubsection{$B_{\perp,4}$}
\begin{nlemma}
	The terms in Stokes' theorem coming from $(evb_0|_{B_{4,\perp}})_*\xi$ are \begin{equation}
		\sum_{\substack{ P \in S_5[k]\\ \beta_1 + \beta_2+\beta_3 = \beta\\ J_1 \cup J_2 \cup J_3 = [l]\\1 \in J_3, \;2 \in J_2}} \mathfrak{q}_{k_1+1+k_3+k_5,l_1}^{ST,\beta_1}(\alpha^{(1:5)} \otimes \mathfrak{q}_{k_2,l_2}^{ST,\beta_2}(\alpha^{(2:5)};\gamma^{J_2}) \otimes \alpha^{(3:5)} \otimes \mathfrak{q}_{k_4,l_4}^{ST,\beta_4}(\alpha^{(4:5)};\gamma^{J_3}) \otimes \alpha^{(5:5)}; \gamma^{J_1}),
	\end{equation}
	where $C_4 = n+ |\gamma| + \zeta(\alpha) +  \epsilon(\alpha) + sign^\gamma(J_1,J_2,J_3) + (|\gamma^{J_2}|+1)\epsilon_1 + (|\gamma^{J_3}|+1)(\epsilon_1 + \epsilon_2 + \epsilon_3) + |\gamma^{J_2}|$.
\end{nlemma}
Recall that \begin{equation}
\phi_4: \mathcal{M}_{k_4+1, l_3}(\beta_3) \times_L \mathcal{M}_{k_1+k_3+k_5+3,l_1}(\beta_1) \times_L \mathcal{M}_{k_2+1,l_2}(\beta_2) \rightarrow B_{\perp,4} \subset \partial \mathcal{M}_{k+1,l,\perp_0}(\beta)
\end{equation}
changes orientation by $\delta_1 := sign(\phi_4) = k_4(k_1 + k_2 + k_3) + k_2(k_3+k_5) + k_3$.

Denote everything associated with $\mathcal{M}_{k_1+k_3+k_5+3,l_1}(\beta_1)$ with a subscript $1$, everything associated with $\mathcal{M}_{k_2+1,l_2}(\beta_2)$ with a subscript $2$ and subscript 3 for $\mathcal{M}_{k_4+1, l_3}(\beta_3)$. Consider the commutative diagram: \begin{equation}
	\begin{tikzcd}
		\mathcal{M}_3 \times_L \mathcal{M}_{1} \times_L \mathcal{M}_{2} \arrow[r,"p_{13}"] \arrow[d,"p_2"] &\mathcal{M}_3 \times_L \mathcal{M}_{1}\arrow[d,"evb^1_{k_1+1}\circ p^{13}_1",swap] \arrow[r,"p^{13}_3"] \arrow[rd,"p^{13}_1"] & \mathcal{M}_3 \ar[r,"evb^3_0"]&L \\
		\mathcal{M}_{2}\arrow[r,"evb^2_0",swap]  & L & \arrow[l,"evb^1_{k_1+1}"] \mathcal{M}_1 \ar[ru,"evb^1_{k_1+k_3+k_5+2}",swap]
	\end{tikzcd}
\end{equation}
Let \begin{align}
	\overline{\xi} &= \phi_4^*\xi,\\
	\xi_1 &= (evi^1)^*\gamma^{J_1} \wedge (evb^1)^*(\alpha^{(1:5)}\wedge\alpha^{(3:5)}\wedge\alpha^{(5:5)}),\\
	\xi_2 &= (evi^2)^*\gamma^{J_2} \wedge (evb^2)^*\alpha^{(2:5)},\\
	\xi_3 &= (evi^3)^*\gamma^{J_3} \wedge (evb^3)^*\alpha^{(4:5)}.
\end{align}
$\delta_2$ is defined by: \begin{equation}
	p_3^*\xi_3 \wedge p_1^*\xi_1 \wedge p_2^*\xi = (-1)^{\delta_2}\overline{\xi},
\end{equation}
so that: \begin{multline}
	\delta_2 = sign^\gamma(J_1,J_2) + sign^\gamma(J_3,J_1\cup J_2) + |\gamma^{J_1}|(k_4+\epsilon_4) + |\gamma^{J_2}|(k_5 +\epsilon_5 + k_4+\epsilon_4+k_3+\epsilon_3 + k_1 + \epsilon_1)\\
	+ (\epsilon_4 + k_4)(\epsilon_1+k_1+\epsilon_3+k_3) + (\epsilon_2+k_2)(k_5 +\epsilon_5 + k_4+\epsilon_4+k_3+\epsilon_3).
\end{multline}

Now compute \begin{align}
	&(evb_0|_{B_{\perp,4}})_*\xi = (-1)^{\delta_1+\delta_2} (evb^1_0)_*(p_1)_*(p_3^*\xi_3 \wedge p_1^*\xi_1 \wedge p_2^*\xi_2)\\
	& = (-1)^{\delta_1+\delta_2} (evb^1_0)_*(p^{13}_1)_*(p_{13})_*((p_{13})^*((p^{13}_3)^*\xi_3 \wedge (p^{13}_{1})^*\xi_1) \wedge p_2^*\xi_2)\\
	& = (-1)^{\delta_1+\delta_2} (evb^1_0)_*(p^{13}_1)_*((p^{13}_3)^*\xi_3 \wedge (p^{13}_{1})^*\xi_1 \wedge (p_{13})_*p_2^*\xi_2)\\
	& = (-1)^{\delta_1+\delta_2} (evb^1_0)_*(p^{13}_1)_*((p^{13}_3)^*\xi_3 \wedge (p^{13}_{1})^*\xi_1 \wedge (p^{13}_{1})^*(evb^1_{k_1+1})^*(evb_0^2)_*\xi_2)\\
	& = (-1)^{\delta_1+\delta_2 + \delta_3} (evb^1_0)_*((p^{13}_1)_*(p^{13}_3)^*\xi_3 \wedge \xi_1 \wedge (evb^1_{k_1+1})^*(evb_0^2)_*\xi_2),\\
	\intertext{where $\delta_3 = k_4(|\gamma^{J_1}| + |\gamma^{J_2}|+k_5 +\epsilon_5 +k_3+\epsilon_3 + k_1 + \epsilon_1+ \epsilon_2)$.}
	& = (-1)^{\delta_1+\delta_2 +\delta_3 + \delta_4} (evb^1_0)_*((evb^1_{k_1+k_3+k_5+2})^*(evb_0^3)_*\xi_3 \wedge \xi_1 \wedge (evb^1_{k_1+1})^*(evb_0^2)_*\xi_2)\\
	\intertext{where $\delta_4 = k_4(k_1+k_3+k_5)$}
	& = (-1)^{\sum_{i=1}^5\delta_i} (evb^1_0)_*((evb^1_{k_1+k_3+k_5+2})^*(evb_0^3)_*\xi_3 \wedge \xi_1 \wedge (evb^1_{k_1+1})^*\mathfrak{q}^{ST}_{k_2,l_2}(\alpha^{(2:5)};\gamma^{J_2}))\\
	\intertext{where $\delta_5 = \zeta(\alpha^{(2:5)})$}
	& = (-1)^{\sum_{i=1}^6\delta_i} (evb^1_0)_*((evb^1_{k_1+k_3+k_5+2})^*\mathfrak{q}^{ST}_{k_4,l_4}(\alpha^{(4:5)};\gamma^{J_3}) \wedge \xi_1 \wedge (evb^1_{k_1+1})^*\mathfrak{q}^{ST}_{k_2,l_2}(\alpha^{(2:5)};\gamma^{J_2}))\\
	\intertext{where $\delta_6 = \zeta(\alpha^{(4:5)})$}
	& = (-1)^{\sum_{i=1}^7\delta_i} (evb^1_0)_*((evi^1)^*\gamma^{J_1} \wedge (evb^1)^*( \alpha^{(1:5)} \wedge \mathfrak{q}^{ST}_{k_2,l_2}(\alpha^{(2:5)};\gamma^{J_2}) \wedge \alpha^{(3:5)} \wedge \mathfrak{q}^{ST}_{k_4,l_4}(\alpha^{(4:5)};\gamma^{J_3}) \wedge \alpha^{(5:5)}))\\
	\intertext{where $\delta_7 = (|\gamma^{J_1}|+\epsilon_1+k_1+\epsilon_3+k_3)(\epsilon_4 + |\gamma^{J_3}|) + (|\gamma^{J_2}|+\epsilon_2)(k_5 +\epsilon_5 +k_3+\epsilon_3+\epsilon_4 +|\gamma^{J_3}|)$}
	&=(-1)^{\sum_{i=1}^8\delta_i}\mathfrak{q}_{k_1+1+k_3+k_5,l_1}^{ST,\beta_1}(\alpha^{(1:5)} \otimes \mathfrak{q}_{k_2,l_2}^{ST,\beta_2}(\alpha^{(2:5)};\gamma^{J_2}) \otimes \alpha^{(3:5)} \otimes \mathfrak{q}_{k_4,l_4}^{ST,\beta_4}(\alpha^{(4:5)};\gamma^{J_3}) \otimes \alpha^{(5:5)}; \gamma^{J_1})
\end{align}
where $\delta_8 = \zeta(\alpha^{(1:5)}, \mathfrak{q}^{ST}_{k_2,l_2}(\alpha^{(2:5)};\gamma^{J_2}) , \alpha^{(3:5)} , \mathfrak{q}^{ST}_{k_4,l_4}(\alpha^{(4:5)};\gamma^{J_3}) , \alpha^{(5:5)})$. Adding all signs together with the sign in Stokes' theorem \ref{Stokes theorem}, we get an overall sign: \begin{equation}
	C_4 = n+ |\gamma| + \zeta(\alpha) +  \epsilon(\alpha) + sign^\gamma(J_1,J_2,J_3) + (|\gamma^{J_2}|+1)\epsilon_1 + (|\gamma^{J_3}|+1)(\epsilon_1 + \epsilon_2 + \epsilon_3) + |\gamma^{J_2}|.
\end{equation}

The above computation relies on the following lemma, which is similar to \cite[Lemma~2.9]{ST3}:

\begin{nlemma}
	Let $P \in S_5[k]$, then \begin{multline}
		\zeta(\alpha^{(2:5)}) + \zeta(\alpha^{(4:5)}) + \zeta(\alpha^{(1:5)}, \mathfrak{q}^{ST}_{k_2,l_2}(\alpha^{(2:5)};\gamma^{J_2}) , \alpha^{(3:5)} , \mathfrak{q}^{ST}_{k_4,l_4}(\alpha^{(4:5)};\gamma^{J_3}) , \alpha^{(5:5)})\\ 
		\equiv \zeta(\alpha) + \epsilon_2 + \epsilon_3 + k_3 +|\gamma^{J_2}|(k_1 + 1) + |\gamma^{J_3}|(k_1+k_3) + 1 + k_2(\epsilon_3 + \epsilon_4 + \epsilon_5) + k_4\epsilon_5 \;(mod \;2).
	\end{multline}
\end{nlemma}
\begin{proof}
	All our computations here will be done modulo two. From the definition of $\zeta$ we have:
	\begin{equation}
		\zeta(\alpha^{(2:5)}) = 1 + \sum_{i = k_1 + 1}^{k_2}(i-k_1)|\alpha_i|' = 1 + k_1\epsilon_2 + \sum_{i = k_1 + 1}^{k_2} i|\alpha_i|'.
	\end{equation}
	Similarly we have: \begin{equation}
		\zeta(\alpha^{(4:5)}) = 1 + \sum_{i = k_1+k_2+k_3 + 1}^{k_1+k_2+k_3+k_4}(i-(k_1+k_2+k_3))|\alpha_i|' = 1 +(k_1+k_2+k_3)\epsilon_4 + \sum_{i = k_1+k_2+k_3 + 1}^{k_1+k_2+k_3+k_4} i|\alpha_i|'.
	\end{equation}
	Finally:\begin{multline}
		\zeta(\alpha^{(1:5)}, \mathfrak{q}^{ST}_{k_2,l_2}(\alpha^{(2:5)};\gamma^{J_2}) , \alpha^{(3:5)} , \mathfrak{q}^{ST}_{k_4,l_4}(\alpha^{(4:5)};\gamma^{J_3}) , \alpha^{(5:5)})\\
		=1+ \sum_{i = 1}^{k_1}i|\alpha_i|' + (k_1 + 1)(\epsilon_2 + |\gamma^{J_2}| +  1) + \sum_{i = k_1+k_2+1}^{k_1+k_2+k_3}(i+1 -k_2)|\alpha_i|' + (k_1+k_3+2)(\epsilon_4 + |\gamma^{J_3}| + 1)\\
		+ \sum_{i = k_1+k_2+k_3+k_4+1}^{k}(i+2 - k_2 - k_4)|\alpha_i|' \\
		= 	1+ \sum_{i = 1}^{k_1}i|\alpha_i|' + \sum_{i = k_1+k_2+1}^{k_1+k_2+k_3}i|\alpha_i|' + \sum_{i = k_1+k_2+k_3+k_4+1}^{k}i|\alpha_i|' + (k_1 + 1)(\epsilon_2 + |\gamma^{J_2}| +  1)\\
		 + (k_2+1) \epsilon_3 + (k_1+k_3)(\epsilon_4 + |\gamma^{J_3}| + 1) + (k_2+k_4)\epsilon_5.
	\end{multline}
	Adding up these three terms gives the result.
\end{proof}

\subsection{Concluding the proof of the structure equation}
The last step of the proof of Proposition \ref{boundary of ST horocycle moduli} is to combine the terms coming from $B_{\perp,1}$ with term \ref{combine with C1} coming from $(evb_0)_*(d\xi)$. Because the disk bubbles in $B_{\perp,1}$ must be stable, there are no disks contributing with $\beta_2 =\beta_0$. These contributions are exactly provided by the term $(evb_0)_*(d\xi)$ in Stokes' theorem. The same holds for $B_{\perp,2}$ and the terms coming from $d(evb_0)_*\xi$. We then sum all the terms together with those from $B_{\perp,3}$ and  $B_{\perp,4}$. Finally multiply by $(-1)^{n+|\gamma| + \zeta(\alpha) + \epsilon(\alpha)}$ to get Proposition \ref{boundary of ST horocycle moduli}.

To prove Proposition \ref{boundary of horocycle moduli}, one easily adapts the signs. The main computation one needs is the following, which can be directly verified:

\begin{nlemma}We have:\begin{multline}
		\dagger\left(\alpha^{(1:5)},\mathfrak{q}^{ST}_{k_2,j_2}(\alpha^{(2:5)};\gamma^{J_2}), \alpha^{(3:5)} , \mathfrak{q}^{ST}_{k_4,j_3}(\alpha^{(4:5)};\gamma^{J_3}) , \alpha^{(5:5)};\gamma^{J_1}\right) + \dagger(\alpha^{(2:5)}) + \dagger(\alpha^{(4:5)})= \\ \dagger(\al) + 1 + \epsilon_2 + \epsilon_4 + |\gamma^{J_2}|(\epsilon_1 + \epsilon_3 + \epsilon_4 + \epsilon_5 +1) + |\gamma^{J_2}|(\epsilon_1 + \epsilon_2 + \epsilon_3 + \epsilon_5 +1 + |\gamma^{J_2}|).
	\end{multline}
\end{nlemma}

\subsection{Unit on the horocycle}
\label{sec: Proof of unit on the horocycle}
This proof follows that of \cite[Lemma~3.10]{ST2} very closely. Let $p_i: \mathcal{M}_{k+2,l; \perp_i}(\beta) \rightarrow \mathcal{M}_{k+1,l}(\beta)$ be the map given by forgetting the boundary marked point $z_i$ and the horocyclic constraint. This map is injective, as there is a unique horocycle through any two interior points $w_1$ and $w_2$. Fixing the interior marked point $w_1$ and the boundary marked points $z_j$ for $j \neq i$, we see that as $z_i$ moves between the two adjacent boundary marked points $z_j$ and $z_{j+1}$, the point $w_2$ sweeps out a lunar arc between the two horocycles through $w_1$ and $z_j$, and through $w_1$ and $z_{j+1}$ (see Figure \ref{fig:horocycles2}). As these lunar arcs cover the entire unit disk we see that the image of \begin{equation}
	p := \sqcup_{i=1}^{k+1} p_i: \sqcup_{i=1}^{k+1} \mathcal{M}_{k+2,l; \perp_i}(\beta) \rightarrow \mathcal{M}_{k+1,l}(\beta)
\end{equation} is an open dense subset.
\begin{figure}[h]
	\centering
	\includegraphics[scale=0.1]{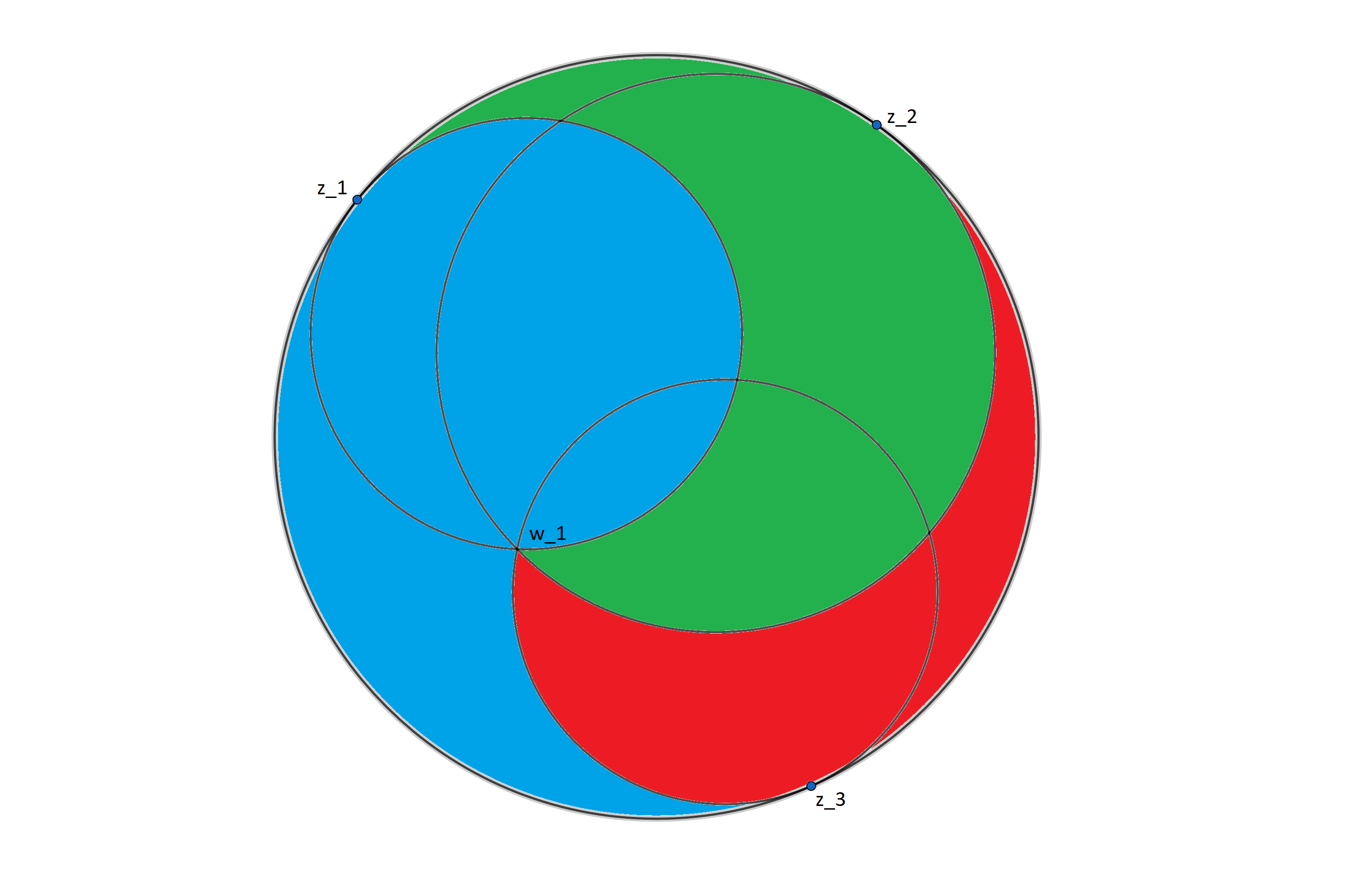}
	\caption{the coloured regions show the areas swept out by $w_2$.}
	\label{fig:horocycles2}
\end{figure}
We next compute the sign of $p_i$: \begin{nlemma}
	The map $p_i: \mathcal{M}_{k+2,l; \perp_i}(\beta) \rightarrow \mathcal{M}_{k+1,l}(\beta)$ is a local diffeomorphism with $sign(p_i) = n + i$.
\end{nlemma}
\begin{proof}
	The orientation on $\MM_{k+2,l,\perp_i}(\beta) = I \times_{D^2} \MM_{k+2,l}(\beta)$ is defined by the fibre-product orientation (see \cite[Section 2.2]{ST4}). At the level of tangent spaces, this means that we have a natural identification \begin{equation}
		T_t I \oplus T_u \MM_{k+2,l}(\beta) \cong T_t D^2 \oplus T_{(t,u)}(I \times_{D^2} \MM_{k+2,l}(\beta)).
	\end{equation}
	To calculate the sign of $p_i$, we multiply both sides by a factor of $D^2$ and $Aut(D^2)$. This does not change the sign. Let $\RR_i$ denote the tangent space to the factor $S^1_i$ in the definition of the orientation on $\MM_{k+2,l}(\beta)$ (see equation \eqref{eq: orientation of moduli space of curves}). Thus, $sign(p_i)$ is equal to the sign of the map: \begin{multline}
		T_t I \oplus T_u \widetilde{\MM}^{main}(\beta) \oplus \RR_0 \oplus \dots \oplus \RR_{k+2} \oplus \CC^l \oplus T_z Aut(D^2) \\\rightarrow  T_t D^2 \oplus T_u \widetilde{\MM}^{main}(\beta) \oplus \RR_0 \oplus \dots \oplus \RR_{i-1} \oplus \RR_{i+1} \oplus \dots \oplus \RR_{k+2} \oplus \CC^l \oplus T_z Aut(D^2).
	\end{multline}
	Bringing the summand $\RR_i$ to the second position, which gives a sign of $(-1)^{i+n}$, and cancelling the obvious summands, we find that $sign(p_i)$ is given by the sign of:
	\begin{equation}
		T_t I \oplus \RR_i \rightarrow  (-1)^{n + i} T_t D^2.
	\end{equation}
	By drawing a local picture, it is clear that the natural map $T_t I \oplus \RR_i \rightarrow T_t D^2$ is orientation preserving. We have thus shown that $sign(p_i) = i + n$.
\end{proof}
\begin{proof}[Proof of Proposition \ref{unit on the horocycle}]
	Let $evi_j$ and $evb_j$ denote the evaluation maps for the space $\MM_{k+1,l}(\beta)$. Use $evi^i_j$ and $evb^i_j$ for those of $\MM_{k+2,l,\perp_i}$. The map $p_i$ then satisfies $evi_j \circ p_i = evi^i_i$ and: \begin{equation}
		evb_j \circ p_i = \begin{cases}
			evb^i_j & \text{if $j<i$}\\
			evb^i_{j+1} & \text{if $j > i$} 
		\end{cases}
	\end{equation} Let $\xi = \bigwedge_{j=1}^l (evi_j)^*\gamma_j \wedge \bigwedge_{j=0}^{k} (evb_j)^* \al_j$ and $\xi_i = p^*_i \xi$. Then as the images of $p_i$ for $1\leq i \leq k+1$ form an open dense set in $\MM_{k+1,l}(\beta)$, we have: \begin{equation}
	pt_*\xi = \sum_{i=1}^{k+1} (-1)^{sign(p_i)} pt_* \xi_i.
	\end{equation} 
	A short sign computation shows that  \begin{equation}
	pt_* \left( (evb_0)_*\xi \wedge \al_0 \right) = (-1)^{\zeta(\widetilde{\al}) + |\al_0|(\epsilon(\widetilde{\al}) + 1)}\langle \q^{ST,\beta}_{k,l}(\al_1, \dots, \al_k; \gamma), \al_0 \rangle_L.
	\end{equation}
	
	Take the partition with $(1:2) = (1, \dots, i-1)$. We then compute: \begin{align}
		pt_* \xi_i &= (-1)^{\zeta(\widetilde{\al}) + |\al_0|(\epsilon(\widetilde{\al}) + 1) + 1 + |\ga| + n + \epsilon_2 + i}\langle \q^{ST,\beta}_{k+1,l,\perp_i}(\al^{(1:2)},1, \al^{(2:2)}); \gamma), \al_0 \rangle_L.
	\end{align}
	Combining this, we have shown: \begin{equation}
		\langle \mathfrak{q}^{ST}_{k,l}(\widetilde{\alpha};\gamma),\alpha_0 \rangle_L = \sum_{P \in S_2[k]} (-1)^{\epsilon_1+|\gamma|+1} \langle \mathfrak{q}^{ST}_{k+1,l,\perp_i}(\alpha^{(1:2)},1,\alpha^{(2:2)};\gamma), \alpha_0 \rangle_L.
	\end{equation}
	Taking into account the sign changes between the operations $\q^ST$ and $\q$, the first result of Proposition \ref{unit on the horocycle} follows. The second equality follows from a cyclic symmetry property of the $\mathfrak{q}_\perp$ operations, which is the direct analogue of Proposition \ref{cyclic symmetry}.
\end{proof}
\section{Applications of Conjecture \ref{VSHS conjecture} in the monotone setting}
\label{Quantum cohomology}
Let $(X,\omega)$ be a $2n$ dimensional monotone symplectic manifold, i.e $c_1 = c_1(TX) = \tau [\omega] \in H^2(X;\mathbb{R})$ for some $\tau \in \mathbb{R}_{>0}$. We can then consider quantum cohomology $QH^*(X) = H^*(X;\mathbb{C})$ with coefficients in $\mathbb{C}$, as for example in \cite{She16}.
\begin{defi}[Quantum E-structure]
	The \emph{quantum E-structure} is the $\mathbb{C}[[u]]$-module $\DX$ with connection:
	\begin{equation}
		\nabla_{\frac{d}{du}} = \frac{d}{d u} + \frac{\mu}{u} + \frac{c_1 \star}{u^2}.
	\end{equation}
	Here $\star$ is the quantum cup product, and $\mu: QH^*(X) \rightarrow QH^*(X)$ is the grading operator with $\mu(\alpha) = \frac{p-n}{2}\alpha$ for $\alpha \in QH^p(X)$. The residue of the connection is given by $c_1 \star$. The canonical map $\pi: \DX \rightarrow QH^*(X)$ is given by evaluation at $u=0$. We will often write $\nabla$ for $\nabla_{\frac{d}{du}}$.
\end{defi}
\begin{defi}
	Quantum cohomology admits a canonically defined splitting, given by:
	\begin{align}
		s^{GW}: QH^*(X) &\rightarrow \DX\nonumber\\
		\alpha &\mapsto \alpha.
	\end{align}
	As this is the splitting relevant for Gromov-Witten theory (see e.g. \cite{Gr}), we call this the \emph{Gromov-Witten splitting}.
\end{defi}

Decompose quantum cohomology as a direct sum of generalised eigenspaces of $c_1 \star$ \begin{equation}
	QH^*(X) = \bigoplus_w QH^*(X)_w.
\end{equation}

Proposition \ref{decomposition} shows that we can extend this decomposition by eigenvalues of $c_1 \star$ to the E-structure $\DX$. 
\begin{nlemma}
	\label{decomposition of quantum D module}
	There exists a unique decomposition of $\DX$:
	\begin{equation}
		\DX = \bigoplus_w \DX_w,
	\end{equation}
	which is compatible with the connection: \begin{equation}
		u^2\nabla: \DX_w \rightarrow \DX_w,
	\end{equation} and respects the eigenvalue decomposition of $QH^*(X)$: \begin{equation}
		\pi(\DX_w) = QH^*(X)_w.
	\end{equation}
\end{nlemma}
\begin{remark}
	Note that in general $\DX_w \neq QH^*(X)_w[[u]]$. This is because the map $u^2\nabla$ does not map $QH^*(X)_w[[u]]$ to itself, as in general $\mu$ will not, see Example \ref{example S2}. Similarly, the Gromov-Witten splitting does not respect the eigenvalue decomposition. So in general $s^{GW}(QH^*(X)_w) \not\subset \DX_w$. See Example \ref{example S2}.
\end{remark}
Lemma \ref{uniqueness of decomposition} together with Corollary \ref{single summand decomposition} then shows:
\begin{ncor}
	\label{cyclic OC respects decompositions}
	Let $X$ be a monotone symplectic manifold. Assume Conjecture \ref{VSHS conjecture} holds. Then the cyclic open-closed map respects the decompositions, that is: \begin{equation}
		\OC^{-}(HC^-_*(Fuk(X)_w)) \subset \DX_w.
	\end{equation}
\end{ncor}

\subsection{Semi-simple quantum cohomology}
\label{semi-simple QH}
Now assume additionally that the quantum cohomology ring $QH^*(X)$ is semi-simple. This is the case for example for $X = \mathbb{CP}^n$. The main result is:
\begin{nthm}
	\label{non-singular quantum D module}
	Let $X$ be a monotone symplectic manifold with semi-simple quantum cohomology. Then the EP-structure $\DX$ is semi-simple in the sense of Definition \ref{Semi-simple EP-structure}.
\end{nthm}
\begin{remark}
	This theorem is a rephrasing of results in \cite{Dub} and \cite{GGI}.
\end{remark}

In order to prove this theorem, we need the following lemma. This was first observed by Dubrovin \cite{Dub}, but see also \cite[Remark~8.2~(iii)]{Te} or \cite[Lemma~2.4.4]{GGI}.

\begin{nlemma}[{\cite[Lemma~3.2]{Dub}}]
	\label{mu property}
	When $QH^*(X)$ is semi-simple, $0 = \pi_w \circ \mu|_{QH_w^*(X)}: QH_w^*(X) \rightarrow QH_w^*(X)$. I.e. the diagonal blocks for a matrix representation of $\mu$ vanish, when we choose a basis of eigenvectors for $c_1 \star$.
\end{nlemma}
\begin{proof}[Proof of theorem \ref{non-singular quantum D module}]
	By Proposition \ref{decomposition}, after changing to a basis for the $\CC[[u]]$-modules $\DX_w$ the connection is given by 
	\begin{equation}\nabla = \frac{d}{du} + u^{-2}c_1 \star + u^{-1}B_{-1} + \sum_{i\geq0} B_i u^i.\end{equation}
	We also have that $B_{-1}|_{QH^*(X)_w} = \pi_w \circ \mu|_{QH^*(X)_w} = 0$ by Lemma \ref{mu property}. As $B_{-1}$ also respects the decomposition of $QH^*(X)$ into eigenspaces for $c_1 \star$, we find that $B_{-1} = 0$. The connection on $\DX_w$ is thus given by: \begin{equation}
		\nabla = \frac{d}{du} + \frac{w}{u^2} + \sum_{i\geq 0} u^i B_i.
	\end{equation}
	The connection $\widetilde{\nabla}^w$ on $\DX_w \otimes \mathcal{E}^{\frac{w}{u}}$ is then given by:\begin{equation}
		\widetilde{\nabla}^w = \frac{d}{du} + \sum_{i\geq 0} u^i B_i.
	\end{equation}
	As in \cite[Chapter~2,~case~1]{Tur}, there now exists a further change of basis so that the connection matrix on $\DX_w \otimes \mathcal{E}^{\frac{w}{u}}$ vanishes. Equivalently, there exists an isomorphism $\DX_w \cong \mathcal{E}^{-\frac{w}{u}}$.
\end{proof}
We thus obtain the following result, which is a part of \cite[Proposition~2.5.1]{GGI}.

\begin{cor}
	\label{w flat basis qcoh}
	When $QH^*(X)$ is semi-simple, there exists a basis $v_i \in \DX$ such that $u^2\nabla_{\frac{d}{du}}v_i = w_iv_i$, where the $w_i$ are the eigenvalues of $c_1 \star$.
\end{cor}
We call the basis $\{v_{i}\}$ a $w$-\emph{flat} basis. As in Example \ref{splitting on SS E-structure}, we obtain a splitting:  \begin{align}
	s^{ss}: QH^*(X)& \rightarrow \DX\nonumber\\
	\pi(v_{i_j})& \rightarrow v_{i_j}.
\end{align}
There then exists a unique transformation $R = \sum_{i \geq 0} u^iR_i$, where $R_i: QH^*(X) \rightarrow QH^*(X)$ such that \begin{equation}
	s^{ss}(\alpha) = \sum_{i \geq 0} u^i s^{GW}(R_i(\alpha)).
\end{equation}
We will now show $R$ agrees with the R-matrix as defined by Teleman \cite{Te}. Start with a basis for $QH^*(X)$ consisting of eigenvectors $\{e_{i}\}$ for $c_1 \star$. Thus in this basis for $\DX$ the connection reads: \begin{equation}
	\nabla = \frac{d}{du} + u^{-2} c_1 \star + u^{-1}\mu.
\end{equation}
The R-matrix then changes bases to $v_{i} = R(e_{i})$, in which the connection can be expanded as \begin{equation}
	R^*\nabla = \frac{d}{du} + u^{-2} c_1 \star.
\end{equation}
We thus obtain the recursive relation for $R$: \begin{equation}
	[c_1 \star, R_{i+1}] + (\mu + i) R_i = 0,
\end{equation}
which agrees with the relation in \cite[Proposition~8.5]{Te}.

\begin{eg}
\label{example S2}
Consider a basis $1,H$ for $S^2$, where $H$ is the point class. Then $c_1 = 2 H$. The quantum multiplication table reads \begin{equation}
	1 \star 1 = 1, \; 1 \star H = H, \; H\star H = 1.
\end{equation}
Thus $QH^*(S^2) \cong \frac{\mathbb{C}[H]}{\langle H^2 - 1 \rangle}$, which is semi-simple. The decomposition into eigenspaces for $c_1$ is: \begin{equation}QH^*(S^2) = QH^*(S^2)_{-2} \oplus QH^*(S^2)_2 = \langle H-1 \rangle \oplus \langle 1+H \rangle := \langle v \rangle \oplus \langle w \rangle.
\end{equation}

Now $\mu(v) = \frac{w}{2}$ and $\mu(w) = \frac{v}{2}$. Thus the $\mathbb{C}[[u]]$-modules $\langle 1+H \rangle_{\mathbb{C}[[u]]}$ and $\langle 1-H \rangle_{\mathbb{C}[[u]]}$ are \emph{not} invariant under $u^2\nabla$. In fact, we can solve the differential equation $u^2 \nabla_{\frac{d}{du}} \widetilde{w} = 2 \widetilde{w}$ directly, and find \begin{equation}
	\widetilde{w} = \sum_{n\geq 0} u^n(\alpha_n w + \beta_n v),
\end{equation}
where \begin{equation}
	\alpha_{n+1} = \frac{4^{-2(n+1)}}{(n+1)!}\prod_{j=0}^{n} (4j^2 - 1) \; \text{and} \; \beta_n = -2n\alpha_n.
\end{equation}
Similarly we obtain a solution to $u^2 \nabla_{\frac{d}{du}} \widetilde{v} = 2 \widetilde{v}$:
\begin{equation}
	\widetilde{v} = \sum_{n\geq 0} u^n(\ga_n w + \delta_n v),
\end{equation}
where \begin{equation}
	\ga_{n} = -2n\delta_{n} \; \text{and} \; \delta_{n+1} = \frac{(-1)^{n+1}4^{-2(n+1)}}{(n+1)!}\prod_{j=0}^{n} (4j^2 - 1). 
\end{equation}
The R-matrix in the basis $(v,w)$ is thus:
\begin{equation}
	R_{n+1} = \frac{4^{-2(n+1)}}{(n+1)!}\prod_{j=0}^{n} (4j^2 - 1)
	\begin{pmatrix}
		(-1)^{n+1}&(-1)^{n}2(n+1)  \\
		-2(n+1)&1 \\
	\end{pmatrix},
\end{equation}
which indeed agrees with the R-matrix computed in \cite[Example~5.4]{AT} for the cyclic homology of $Fuk(S^2)$.
\end{eg}

We now rephrase {\cite[Theorem~5.9]{AT} and provide an alternative proof. The proof in \cite{AT} uses the closed-open map and the Dubrovin-Teleman reconstruction theorem \cite{Te}. Our proof instead uses the cyclic open-closed map and assumes Conjecture \ref{VSHS conjecture}. In particular, it does not rely on the Dubrovin-Teleman reconstruction theorem.
	\begin{nthm}
		\label{obtaining GW from Fuk}
		Let $X$ be a symplectic manifold such that $\OC: HH_*(Fuk(X)) \rightarrow QH^*(X)$ is an isomorphism and $HH^*(Fuk(X))$ is semi-simple. Let $\mu^{\OC} = \OC^{-1} \circ \mu \circ \OC: HH_*(Fuk(X)) \rightarrow HH_*(Fuk(X))$ be the pull-back of the grading operator $\mu$ on $QH^*(X)$. Then the Frobenius manifold $\MM_{\mu^{\OC}}$ associated to $\mu^{\OC}$ (see Section \ref{semi-simple TEP structures}) is isomorphic to the big quantum cohomology of $X$. Here $Fuk(X)$ denotes a non-bulk deformed Fukaya category defined over $\Lambda$ (or $\CC$ in the monotone case).
	\end{nthm}
	\begin{proof}
		As the open-closed map is an isomorphism, so is the closed-open map $\CO: QH^*(X) \rightarrow HH^*(Fuk(X))$. Thus, $QH^*(X)$ is also semi-simple. Lemma \ref{non-singular quantum D module} then shows $\DX$ is a semi-simple EP-structure. Let $R = \CC[[H^*(X)]]$ parametrise bulk-deformations. Let $QH^*(X;R)[[u]]$ denote the quantum TEP-structure over $R$. 
		
		Let $Fuk^t(X)$ denote the bulk-deformed Fukaya category. As $\CO$ is an isomorphism, this is a versal deformation of $Fuk(X)$, and can thus be extracted from the categorical data of $Fuk(X)$.
		
		Now apply the bijection between grading operators and primitive forms  (Corollary \ref{Primitive forms and grading operators}) to the TEP-structure $QH^*(X;R)[[u]]$. It follows directly from the definition of the product and metric in \cite{SaiTak}, that the Frobenius manifold associated to the grading operator $\mu$ and the primitive $\omega = 1 \in QH^*(X)$ is indeed the big quantum cohomology ring $QH^*(X;R)$.
		
		Amorim and Tu \cite[Corollary~3.8]{AT} show that as $HH^*(Fuk(X))$ is semi-simple, $HC^-(Fuk^t(X))$ is a semi-simple TEP-structure. The grading operator $\mu^{\OC}$ on $HH_*(Fuk(X))$ is pulled back from the grading operator on $QH^*(X)$. The primitive element $\omega \in HH_*(Fuk(X))$ is defined as $\OC^{-1}(1)$.
		
		Now consider the bulk deformed cyclic open-closed map \begin{equation}
			\OC^-: HC^-_*(Fuk^t(X)) \rightarrow QH^*(X;R)[[u]].
		\end{equation}
	By conjecture \ref{VSHS conjecture}, this is an isomorphism of TEP-structures. Furthermore, the cyclic open-closed map: \begin{equation}
		\OC^-: HC^-(Fuk(X)) \rightarrow \DX
		\end{equation} 
	respects the grading operator and the primitive element $\omega$. Thus, Corollary \ref{Primitive forms and grading operators} equips the TEP-structures $HC^-(Fuk^t(X))$ and $QH^*(X;R)[[u]]$ with the same primitive form under the cyclic open-closed map. Hence, the associated Frobenius manifolds $\MM_{\mu^{\OC}}$ and $QH^*(X;R)$ are isomorphic.
	\end{proof}

\subsection{Example: intersection of quadrics}
\label{example intersection of quadrics}
We will now give an example where, even though the quantum cohomology is not semi-simple, it is still possible to construct an R-matrix. Let $X$ be a complete intersection of two quadric hypersurfaces in $\mathbb{CP}^5$, which is a monotone symplectic manifold. The eigenvalue decomposition of the Fukuya category is as follows: \begin{equation}
	Fuk(X) = Fuk(X)_{-8} \oplus Fuk(X)_0 \oplus Fuk(X)_8.
\end{equation}
Smith proves an equivalence:
\begin{nthm}[{\cite[Theorem~1.1]{Smi}}]
	$D^\pi Fuk(X)_0 \cong D^\pi Fuk(\Sigma_2)$, for $\Sigma_2$ a genus 2 surface.
\end{nthm}
Assume that $Fuk(X)_{\pm 8} \cong Fuk(pt)$, which \cite[Section~1.6]{Smi} expects. And note that the $Fuk(pt)$ are considered here with curvature $\pm 8$, so that $HC^-_*(Fuk(X)_{\pm 4}) \cong \mathcal{E}^{\mp 8/u}$. Also note that \cite[Chapter~4]{She} proves a natural isomorphism $HC_*^{-}(D^\pi \mathcal{C}) \cong HC_*^{-}(\mathcal{C})$. We thus have an isomorphism of TE-structures: \begin{equation}
	HC_*^{-}(Fuk(X)) \cong \mathcal{E}^{8/u} \oplus HC_*^{-}(Fuk(\Sigma_2)) \oplus \mathcal{E}^{-8/u}.
\end{equation}
As the TEP-structure associated to an $A_\infty$-category is canonical, Conjecture \ref{VSHS conjecture} implies:
\begin{nlemma}
	\label{intersection of quadrics TEP isom}
	There exists an isomorphism of TEP-structures: \begin{equation}
		\Phi: \DX \cong \mathcal{E}^{\frac{-8}{u}} \oplus \DS \oplus \mathcal{E}^{\frac{8}{u}}.
	\end{equation}
\end{nlemma}
We prove Lemma \ref{intersection of quadrics TEP isom} in Appendix \ref{appendix example intersection of quadrics} by a direct computation, providing evidence for the conjecture. Moreover, we show that this isomorphism is unique.

\appendix

\section{Euler-grading on Fukaya category}
\label{Euler-grading on Fukaya category}
In this appendix we will explain how a `standard' definition of the Fukaya category (see for example \cite{Sei08}), can be adapted to define an Euler-graded $A_\infty$-category. We will also show that bulk-deformations by $c_1$ are unobstructed. We will thus construct a Fukaya category satisfying the properties required in section \ref{outline of general proof}. We will show why Assumption \ref{OC assumptions}(\ref{OC assumptions 2}) holds in this setup. We already verified these assumptions for the case of a Fukaya category with a single Lagrangian, but here we work with a Fukaya category with multiple objects.

\subsection{Euler-grading}
Let $X$ be a symplectic manifold. Let $\mathcal{L}X \rightarrow X$ denote the Grassmanian bundle of oriented Lagrangian subspaces of $TX$. Note that an oriented Lagrangian $L$ comes with a canonical section $s_L: L \rightarrow \mathcal{L}X$. 

Let $e$ have degree $2$. We then define a Fukaya category over $\Lambda[e,e^{-1}]$, denoted by $Fuk(X)[e]$, as follows. Objects in $Fuk(X)[e]$ are given by oriented, relatively spin Lagrangian submanifolds $L \subset X$. For any two (transverse) oriented Lagrangians $L_1, L_2$ and $p \in L_1 \cap L_2$, denote by $\mathcal{L}(L_1,L_2,p)$ the set of homotopy classes of paths $\widetilde{p}: [0,1]  \rightarrow \mathcal{L}_pX$ with $\widetilde{p}(0) = T_pL_1$ and $\widetilde{p}(1) = T_p L_2$. Then define:\begin{equation}
	CF^*(L_1,L_2) = \bigoplus_p \langle \mathcal{L}(L_1,L_2,p) \rangle_{\Lambda}.
\end{equation}
As in \cite[Section~11g]{Sei08}, there is a grading on $CF^*(L_1,L_2)$ given by the Maslov-index for the pair of paths $(\widetilde{p},\widetilde{p}(1))$. The $\Lambda[e,e^{-1}]$-module structure is given by defining $e \cdot \pp$ to be the homotopy class of paths which have Maslov-index $Ind(e \cdot \pp) = Ind(\pp) + 2$.

In the `standard' Fukaya category, the product operations are defined by counting holomorphic disks as follows. Let $L_0, \dots, L_k$ be pairwise transversely intersecting Lagrangians. Choose points $p_0, \dots, p_k$, with $p_i \in L_{i-1} \cap L_i$, where $L_{-1}$ denotes $L_k$. Let $z_0, \dots, z_k \in \partial \mathbb{D}$ be marked points. Then the coefficient of $p_0$ in $m_k(p_1, \dots, p_k)$ is given by counting the number of holomorphic maps $u: \mathbb{D}\setminus \{z_0, \dots z_k\} \rightarrow X$, which extend to continuous maps $u: \mathbb{D} \rightarrow X$ with marked points $u(z_i) = p_i$. Furthermore the boundary in between $z_i$ and $z_{i+1}$, $\partial_i \mathbb{D}$ is required to satisfy $u(\partial_i \mathbb{D}) \subset L_i$. Let $\MM^0(p_0, \dots, p_k; [u])$ denote the zero-dimensional component of the moduli space of such disks in the homotopy class $[u]$. Then, \begin{equation}
	\m_k(p_1, \dots, p_k) = \sum_{\substack{p_0 \in L_k \cup L_0 \\ [u]| Ind([u]) = 2-k}} |\MM^0(p_0, \dots, p_k;[u])|Q^{\omega([u])}p_0.
\end{equation}

Given lifts $\pp_1, \dots, \pp_k$, where $\pp_i \in \mathcal{L}(L_{i-1}, L_i, p_i)$, and a holomorphic disk $u$ as above. We explain how to determine a lift $\pp_0 \in \mathcal{L}(L_0, L_k, p_0)$. To this end, concatenate the paths $\pp_i$ together with the paths $s_{L_i}(u|_{\partial_i \mathbb{D}})$, to obtain a path $\gamma \in \mathcal{L}X$. This is a path starting at $T_{p_0}L_0$ and finishing at $T_{p_0}L_k$. As the bundle $u^*TX \rightarrow \mathbb{D}$ is trivial, we obtain a projection $\mathcal{L}|_{im(u)}X \rightarrow \mathcal{L}_{p_0}X$. The projection of $\gamma$ defines the lift $\pp_0 \in \mathcal{L}(L_0,L_k,p_0)$. In other words, we have $[\gamma \circ (\pp_0)^{-1}] = 0 \in \pi_1(\mathcal{L}X, T_{p_0}L_0)$. So the Maslov index of the loop $\gamma \circ (\pp_0)^{-1}$ vanishes. Alternatively, there exists a lift $\widetilde{u}: \mathbb{D}\setminus\{ z_0, \dots, z_k\} \rightarrow \mathcal{L}X$ covering $u$, which by including small extra chords $\gamma_i$ at the punctures $z_i$ extends to a continuous map $\widetilde{u}: \mathbb{D} \rightarrow \mathcal{L}$ with $\widetilde{u} \circ \gamma_i = \pp_i$ for $i = 0 ,\dots, k$.

The above procedure determines the lift $\pp_0$ as a function of $\pp_1, \dots, \pp_k$ and the class of the disk $[u]$. We then define: \begin{equation}
	\m^e_k(\pp_1, \dots, \pp_k) := \sum_{\substack{p_0 \in L_k \cup L_0 \\ [u]| Ind([u]) = 2-k}} |\MM^0(p_0, \dots, p_k;[u])|Q^{\omega([u])}\pp_0(\pp_1, \dots, \pp_k; [u]).
\end{equation}
\begin{nlemma}
	The above definition endows $Fuk(X)[e]$ with the structure of a $\mathbb{Z}$-graded $A_\infty$-category. The grading operator is given by $Gr(\pp) := Ind(\pp)\pp$.
\end{nlemma}
\begin{proof}
	The fact that the product operations define an $A_\infty$-structure follows directly from the verification for $Fuk(X)$. By (one of the many) definition(s) of the Maslov index, we have $Ind(u) = Ind(\pp_0) - \sum_{i = 1} Ind(\pp_i)$ so that the product satisfies $[Gr, \m^e_k] = (2-k) \m^e_k$ as required.
\end{proof}

\subsection{Bulk-deforming by $c_1$}
Let $U$ be a 1-dimensional vector space, with grading $2$. We thus identify $Q_U := \Lambda[[t]]$ for a formal parameter $t$ of degree $0$. We will now construct a $Q_U^e$-linear category $Fuk^t[e]$. One should think of this as bulk-deforming $Fuk(X)[e]$ by $\ga = tc_1$.

Let $\mathcal{S}X := S^1(\Lambda^n_\CC T)X \xrightarrow[]{\pi} X$ be the circle bundle associated to the top (complex) exterior power of the tangent bundle of $X$. By \cite[Chapter~11]{BT} there exist a global angular form $\theta \in \Omega^1(\mathcal{S}X)$ such that $d\theta = \pi^*c_1$, where we have picked a representative $c_1 \in \Omega^2(X)$ for the first Chern class of $TX$. 

Consider the map $\Phi: \LX \rightarrow \SX$ defined by sending an oriented Lagrangian subspace to its orientation class. For an oriented Lagrangian $l:L \hookrightarrow X$ we thus obtain a map $\widetilde{l} = \Phi \circ s_L: L \rightarrow \SX$. Define the element $\alpha_L = \widetilde{l}^*\theta \in \Omega^1(L)$. This satisfies $d \al_L = l^*c_1$. Thus, the element $(c_1,\al_L) \in H^2(X,L)$ represents the Maslov class of $L$. Note that here we are using a de Rham model for relative cohomology as in \cite[Chapter~6]{BT}.

For $\al = (\al_{L_0}, \dots, \al_{L_k})$ with $\al_{L_i} \in \Omega^*(L_i)$ as defined above, set:
 \begin{equation}
	(c_1,\al)(u) = \int_\mathbb{D} u^*c_1 + \sum_i \int_{\partial_i \mathbb{D}} u^*\al_{L_i}.
\end{equation}
Note that this only depends on the homotopy class of $u$. We then define: \begin{equation}
	\m^{t,e}_k(\pp_1, \dots, \pp_k) := \sum_{\substack{p_0 \in L_k \cup L_0 \\ [u]| Ind([u]) = 2-k}} |\MM^0(p_0, \dots, p_k;[u])|Q^{\omega([u])}\exp\left(t(c_1,\al)([u])\right)\pp_0(\pp_1, \dots, \pp_k; [u]).
\end{equation}
\begin{nlemma}
	The product $\m^{t,e}$ makes $Fuk^t(X)[e]$ into an Euler-graded $A_\infty$-category with $E = e\partial_e$.
\end{nlemma}
We will now show that Assumption \ref{OC assumptions}(\ref{OC assumptions 2}) holds for $Fuk^t(X)[e]$. To this end, we need to pick a basis for each morphism space $CF^*(L, L')$. Here, a basis is determined by a choice of designated paths $\pp \in \mathcal{L}(L, L', p)$ for each $p \in L \cup L'$. Given such a choice of paths, define the length-1 Hochschild cochain $\phi \in CC^*\left(Fuk^t(X)[e]\right)$ by: \begin{equation}
	\phi(\pp) = -\left(\int_{\pp} \Phi^*\theta\right)\pp,
\end{equation} 
for the designated path $\pp$. Then extend $\phi$ linearly in $t$ and $e$. Note that this integral is well-defined. For if $[\pp] = [\pp'] \in \mathcal{L}(L,L',p)$, then there exists a homotopy $h: [0,1]^2 \rightarrow \mathcal{L}_pX$. We then have: \begin{equation}
	\int_{\pp} \Phi^*\theta - \int_{\pp'} \Phi^*\theta = \int_{[0,1]^2} h^*\Phi^* d\theta = \int_{[0,1]^2} h^*\Phi^* \pi^* c_1.
\end{equation}
But $\pi \circ \Phi \circ h$ is the constant map $[0,1]^2 \rightarrow \{p\}$. Thus the latter integral is zero.
\begin{nlemma}
	We have: $\partial_{t}(\m^{t,e}) = e \partial_e(\mte) + [\m^{t,e},\phi]$.
\end{nlemma}
\begin{proof}
	Fix Lagrangians $L_0, \dots, L_k$ and intersection points $p_i \in L_{i-1} \cap L_i$. Let $\pp_i \in \mathcal{L}(L_{i-1},L_i,p_i)$ be the designated paths defining the bases. Let $u \in \MM^0(p_0, \dots, p_k;[u])$ be a holomorphic disk contributing to the product $\mte(\pp_1, \dots, \pp_k)$. Let $\pp_0' \in \mathcal{L}(L_0, L_k, p_0)$ be the output path defined by $u$. We can then write $\pp_0' = e^{s}\pp_0$ for some $s \in \mathbb{Z}$. 
	
	The contribution to $e\partial_e(\mte)(\pp_1, \dots, \pp_k)$ from this disk is thus given by: \begin{equation}
		s\exp(t(c_1,\al)([u]))\pp_0'.
	\end{equation}
	The contribution to $\partial_t(\mte)(\pp_1, \dots, \pp_k)$ from $u$ is given by: \begin{equation}
		(c_1,\alpha)(u) \exp(t(c_1,\al)([u]))\pp_0'.
	\end{equation}
	Finally, the contribution of this disk to $[\mte,\phi](\pp_1, \dots, \pp_k)$ is given by: \begin{equation} 
		\left( \int_{\pp_0} \Phi^*\theta - \sum_{i=0}^k \int_{\pp_i} \Phi^*\theta \right) \exp(t(c_1,\al)([u]))\pp'_0.
	\end{equation}
	It suffices to prove the statement in the lemma for the contribution of each disk $u$ separately. It thus suffices to show: \begin{equation}
		\label{WTS derivative of m}
		(c_1,\alpha)(u) = s + \int_{\pp_0} \Phi^*\theta - \sum_{i=0}^k \int_{\pp_i} \Phi^*\theta,
	\end{equation}
	for every disk $u$ as above. To this end, recall the path $\gamma \subset \LX$ used to define the lift $\pp_0'$, and consider the loop $\gamma \circ (\pp'_0)^{-1} \subset \LX$. By construction $[\gamma \circ (\pp'_0)^{-1}] = 0 \in \pi_1(\mathcal{L}X, T_{p_0}L_0)$. The loop $\Phi \circ (\gamma \circ (\pp'_0)^{-1}): S^1 \rightarrow \SX$ is then also null-homologous. As before, this means that we can construct a lift $\widetilde{u}: \mathbb{D}\setminus\{ z_0, \dots, z_k\} \rightarrow \SX$ covering $u$. By including small extra chords $\gamma_i$ at the punctures $z_i$, this extends to a continuous map $\widetilde{u}: \mathbb{D} \rightarrow \SX$ with $\widetilde{u}(\gamma_i) = \pp_i$ for $i = 1 ,\dots, k$ and $\widetilde{u}(\gamma_0) = \pp_0'$. We then use Stokes' theorem to obtain: \begin{equation}
		\int_{\mathbb{D}} u^*c_1 = \int_{\mathbb{D}} \widetilde{u}^*\pi^*c_1 = \int_{\partial\mathbb{D}} \widetilde{u}^*\theta.
	\end{equation}
	There are two different kinds of boundary to $\widetilde{u}(\partial \mathbb{D})$. The chords $\gamma_i$ get mapped to the paths $\phi(\pp_i)$ and the segments $\partial_i \mathbb{D}$ to the image of the Lagrangians $L_{i}$ under $\Phi \circ s_{L_i}$. We thus have: \begin{align}
		\int_{\partial\mathbb{D}} \widetilde{u}^*\theta &= -\sum_{i=0}^{k} \int_{\partial_i\mathbb{D}} \widetilde{u}^*\theta - \sum_{i=1}^{k} \int_{\pp_i} \Phi^*\theta + \int_{\pp'_0} \Phi^*\theta,\\
		&= -\sum_{i=0}^{k} \int_{\partial_i\mathbb{D}} u^*\al_{L_i} - \sum_{i=1}^{k} \int_{\pp_i} \Phi^*\theta + \int_{\pp_0} \Phi^*\theta + \int_{\pp'_0 - \pp_0} \Phi^*\theta,\\
		&= -\sum_{i=0}^{k} \int_{\partial_i\mathbb{D}} u^*\al_{L_i} - \sum_{i=1}^{k} \int_{\pp_i} \Phi^*\theta + \int_{\pp_0} \Phi^*\theta + s.
	\end{align}
	For the second equality, we have used that $\widetilde{u}|_{\partial_i \mathbb{D}} = \Phi \circ s_{L_i} \circ u|_{\partial_i \mathbb{D}}$. The last equality here follows as $\theta$ is the global angular form on $\SX$ and the fact that $Ind(\pp_0) + 2s = Ind(\pp_0')$. The path $\Phi \circ (\pp_0' \circ (\pp_0)^{-1}): S^1 \rightarrow SX$ then has winding number $s$, as $\LX$ consists of \emph{oriented} Lagrangian subspaces. We have thus proved Equation \eqref{WTS derivative of m}.
\end{proof}
For an open-closed map defined using holomorphic disks, exactly the same reasoning would show Assumption \ref{OC assumptions}(\ref{OC assumptions 3}) holds. 

\section{Example: intersection of quadrics}
\label{appendix example intersection of quadrics}
Let $X$ be a complete intersection of two quadric hypersurfaces in $\mathbb{CP}^5$. In this section we will show: \begin{nlemma}
	There exists an isomorphism of TE-structures over $\mathbb{C} \supset \mathbb{C}$:
	\begin{equation}
		\Phi: \DX \cong \mathcal{E}^{\frac{8}{u}} \oplus \DS \oplus \mathcal{E}^{\frac{-8}{u}}.
	\end{equation}
\end{nlemma} 
First, observe that for both $\DX$ and $\DS$, the TEP-structure associated to the odd-degree cohomology is trivial, as here both $c_1 \star$ and $\mu$ vanish on the odd-degree cohomology.

Take the basis $\langle -2H, 1 \rangle$ for $\DS$, and the standard basis $\langle 1 \rangle$ for both factors $\mathcal{E}^{\pm \frac{8}{u}}$. The connection $\nabla_{\frac{d}{du}}$ for $\mathcal{E}^{\frac{-8}{u}} \oplus \DS \oplus \mathcal{E}^{\frac{8}{u}}$ is then given by: \begin{equation}
	\label{connection on DS}
	\nabla_{\frac{d}{du}} = \frac{d}{du} + u^{-1}	\left(\begin{array}{rrrr}
		0 & 0 & 0 & 0 \\
		0 & \frac{1}{2} & 0 & 0 \\
		0 & 0 & -\frac{1}{2} & 0 \\
		0 & 0 & 0 & 0
	\end{array}\right)
	 +u^{-2} 
	\left(\begin{array}{rrrr}
		-8 & 0 & 0 & 0 \\
		0 & 0 & 1 & 0 \\
		0 & 0 & 0 & 0 \\
		0 & 0 & 0 & 8
	\end{array}\right) =: \frac{d}{du} + u^{-1}N + u^{-2}J.
\end{equation}
Start out with the basis $\langle 1,H, H^{\cup 2},H^{\cup 3}\rangle$ for $\DX$. The connection is then given by:
\begin{equation}
	\nabla_{\frac{d}{du}} = \frac{d}{du} + u^{-1}
	\left(\begin{array}{rrrr}
		-\frac{3}{2} & 0 & 0 & 0 \\
		0 & -\frac{1}{2} & 0 & 0 \\
		0 & 0 & \frac{1}{2} & 0 \\
		0 & 0 & 0 & \frac{3}{2}
	\end{array}\right)
	+u^{-2} 
	\left(\begin{array}{rrrr}
		0 & 8 & 0 & 32 \\
		2 & 0 & 16 & 0 \\
		0 & 2 & 0 & 8 \\
		0 & 0 & 2 & 0
	\end{array}\right).
\end{equation}
We will show there exists a basis $\langle v_0, v_1, v_2, v_3\rangle$ for $\DX$ such that the connection for $\DX$ agrees with the connection in Equation \eqref{connection on DS}. To this end, on $\DX$, first change to a basis $\langle \wv_0, \wv_1, \wv_2, \wv_3\rangle$ consisting of generalised eigenvectors for $c_1 \star$. The change of basis matrix is given by \begin{equation}
	P = \left(\begin{array}{rrrr}
		1 & 0 & 1 & 1 \\
		-\frac{3}{4} & \frac{2}{3} & 0 & \frac{3}{4} \\
		\frac{1}{4} & 0 & -\frac{1}{12} & \frac{1}{4} \\
		-\frac{1}{16} & -\frac{1}{6} & 0 & \frac{1}{16}
	\end{array}\right).
\end{equation}
Thus, for example, $\wv_1 = \frac{2}{3}H - \frac{1}{6}H^{\cup 3}$. In this new basis, the connection is given by: 
\begin{equation}
	\nabla_{\frac{d}{du}} = \frac{d}{du} + u^{-1}
	\left(\begin{array}{rrrr}
	0 & \frac{2}{3} & -\frac{1}{4} & 0 \\
	\frac{9}{16} & 1 & 0 & -\frac{9}{16} \\
	-\frac{3}{2} & 0 & -1 & -\frac{3}{2} \\
	0 & -\frac{2}{3} & -\frac{1}{4} & 0
	\end{array}\right)
	+u^{-2} 
	\left(\begin{array}{rrrr}
	-8 & 0 & 0 & 0 \\
	0 & 0 & 1 & 0 \\
	0 & 0 & 0 & 0 \\
	0 & 0 & 0 & 8
	\end{array}\right) =: \frac{d}{du} + u^{-1}M + u^{-2}J.
\end{equation}
Thus, in the basis $\langle \wv_0, \wv_1, \wv_2, \wv_3\rangle$ for $\DX$, the $u^{-2}$ term of the connection agrees with the one in Equation \eqref{connection on DS}. In Section \ref{construction of the R-matrix}, we show there exists a unique R-matrix $R = \sum_{i=0}^{\infty} u^iR_i$ such that in the basis $\langle v_0, v_1, v_2, v_3\rangle$, where $v_i = R(\wv_i)$, the connection on $\DX$ is also given by Equation \eqref{connection on DS}. We thus define the isomorphism $\Phi: \DX \rightarrow \mathcal{E}^{\frac{8}{u}} \oplus \DS \oplus \mathcal{E}^{\frac{-8}{u}}$ by \begin{align}
	v_0 &\mapsto (1,0,0),\\
	v_1 &\mapsto (0,-2H,0),\\
	v_2 &\mapsto (0,1,0),\\
	v_3 &\mapsto (0,0,1).
\end{align}

\subsection{Construction of the R-matrix}
\label{construction of the R-matrix}
We need $R$ to satisfy: \begin{equation}
	(\frac{d}{du} + u^{-1}N + u^{-2}J)R = R(\frac{d}{du} + u^{-1}M + u^{-2}J).
\end{equation}
Equating powers of $u$ yields the relation: \begin{equation}
	\label{recursion relation R}
	[J,R_{i+1}] = R_i(M - i Id) - NR_i. 
\end{equation}
We will show this equation has a unique solution $R$ with $R_0 = Id$. Suppose we have solved this equation up to $R_i$. Let \begin{equation}
	\label{variables in R matrix}
	R_{i} = \left(\begin{array}{rrrr}
		x^{0}_{00} & x^{0}_{01} & x^{0}_{02} & x^{0}_{03} \\
		x^{0}_{10} & x^{0}_{11} & x^{0}_{12} & x^{0}_{13} \\
		x^{0}_{20} & x^{0}_{21} & x^{0}_{22} & x^{0}_{23} \\
		x^{0}_{30} & x^{0}_{31} & x^{0}_{32} & x^{0}_{33}
	\end{array}\right), \; R_{i+1} = \left(\begin{array}{rrrr}
		x^{1}_{00} & x^{1}_{01} & x^{1}_{02} & x^{1}_{03} \\
		x^{1}_{10} & x^{1}_{11} & x^{1}_{12} & x^{1}_{13} \\
		x^{1}_{20} & x^{1}_{21} & x^{1}_{22} & x^{1}_{23} \\
		x^{1}_{30} & x^{1}_{31} & x^{1}_{32} & x^{1}_{33}
	\end{array}\right).
\end{equation}
We then find: \begin{equation}
	[J,R_{i+1}] = \left(\begin{array}{rrrr}
		0 & -8 \, x^{1}_{01} & -x^{1}_{01} - 8 \, x^{1}_{02} & -16 \, x^{1}_{03} \\
		8 \, x^{1}_{10} + x^{1}_{20} & x^{1}_{21} & -x^{1}_{11} + x^{1}_{22} & -8 \, x^{1}_{13} + x^{1}_{23} \\
		8 \, x^{1}_{20} & 0 & -x^{1}_{21} & -8 \, x^{1}_{23} \\
		16 \, x^{1}_{30} & 8 \, x^{1}_{31} & -x^{1}_{31} + 8 \, x^{1}_{32} & 0
	\end{array}\right).
\end{equation}
Thus, Equation \eqref{recursion relation R} uniquely determines the entries $x^{1}_{01}$, $x^{1}_{02}$, $x^{1}_{03}$, $x^{1}_{13}$, $x^{1}_{23}$,$x^{1}_{10}$, $x^{1}_{20}$, $x^{1}_{30}$, $x^{1}_{31}$ and $x^{1}_{32}$. These are the entries of $R_{i+1}$ which are in the off-diagonal blocks with respect to the Jordan-decomposition of $J$. Furthermore, the entry $x^{1}_{21}$ is also determined, but a solution only exists provided that \begin{equation}
	 (R_i(M - i Id) - NR_i)_{11} = - (R_i(M - i Id) - NR_i)_{22}.
\end{equation}
Similarly we need to have \begin{equation}
	(R_i(M - i Id) - NR_i)_{00} = (R_i(M - i Id) - NR_i)_{33} =(R_i(M - i Id) - NR_i)_{21} = 0.
\end{equation}
These conditions hold for $i=0$, so assume they hold for $i$. We will now write out all the entries of $R_{i+1}$ we determined, and the conditions on $R_i$.
Conditions: \begin{align}
	\label{equation for x00}
	-i x^{0}_{00} + \frac{9}{16} \, x^{0}_{01} - \frac{3}{2} \, x^{0}_{02} &= 0\\
	\label{equation for x33}
	-i x^{0}_{33} - \frac{9}{16} \, x^{0}_{31} - \frac{3}{2} \, x^{0}_{32} &= 0\\
	\label{remaining condition}
	{\left(\frac{3}{2} - i\right)} x^{0}_{21} + \frac{2}{3} \, x^{0}_{20} - \frac{2}{3} \, x^{0}_{23} &= 0\\
	\label{equation for x11}
	-{\left(i + \frac{1}{2}\right)} x^{0}_{22} - \frac{1}{4} \, x^{0}_{20} - \frac{1}{4} \, x^{0}_{23} &= {\left(i - \frac{1}{2}\right)} x^{0}_{11} - \frac{2}{3} \, x^{0}_{10} + \frac{2}{3} \, x^{0}_{13}
\end{align}
We also found the solutions: \begin{align}
x^{1}_{01} &= \frac{1}{8} \, {\left(i - 1\right)} x^{0}_{01} - \frac{1}{12} \, x^{0}_{00} + \frac{1}{12} \, x^{0}_{03}\\
x^{1}_{02} &= -\frac{1}{64} \, {\left(i - 1\right)} x^{0}_{01} + \frac{1}{8} \, i x^{0}_{02} + \frac{1}{24} \, x^{0}_{00} + \frac{1}{8} \, x^{0}_{02} + \frac{1}{48} \, x^{0}_{03}\\
x^{1}_{03} &= \frac{1}{16} \, i x^{0}_{03} + \frac{9}{256} \, x^{0}_{01} + \frac{3}{32} \, x^{0}_{02}\\
x^{1}_{10} &= -\frac{1}{64} \, i {\left(8 \, x^{0}_{10} - x^{0}_{20}\right)} - \frac{1}{16} \, x^{0}_{10} + \frac{9}{128} \, x^{0}_{11} - \frac{3}{16} \, x^{0}_{12} - \frac{1}{128} \, x^{0}_{20} - \frac{9}{1024} \, x^{0}_{21} + \frac{3}{128} \, x^{0}_{22}\\
x^{1}_{13} &= \frac{1}{64} \, i {\left(8 \, x^{0}_{13} + x^{0}_{23}\right)} + \frac{9}{128} \, x^{0}_{11} + \frac{3}{16} \, x^{0}_{12} + \frac{1}{16} \, x^{0}_{13} + \frac{9}{1024} \, x^{0}_{21} + \frac{3}{128} \, x^{0}_{22} - \frac{1}{128} \, x^{0}_{23}\\
x^{1}_{20} &= -\frac{1}{8} \, i x^{0}_{20} + \frac{1}{16} \, x^{0}_{20} + \frac{9}{128} \, x^{0}_{21} - \frac{3}{16} \, x^{0}_{22}\\
x^{1}_{23} &= \frac{1}{8} \, i x^{0}_{23} + \frac{9}{128} \, x^{0}_{21} + \frac{3}{16} \, x^{0}_{22} - \frac{1}{16} \, x^{0}_{23}\\
x^{1}_{30} &= -\frac{1}{16} \, i x^{0}_{30} + \frac{9}{256} \, x^{0}_{31} - \frac{3}{32} \, x^{0}_{32}\\
x^{1}_{31} &= -\frac{1}{8} \, i x^{0}_{31} + \frac{1}{12} \, x^{0}_{30} + \frac{1}{8} \, x^{0}_{31} - \frac{1}{12} \, x^{0}_{33}\\
x^{1}_{32} &= -\frac{1}{64} \, i {\left(x^{0}_{31} + 8 \, x^{0}_{32}\right)} - \frac{1}{48} \, x^{0}_{30} + \frac{1}{64} \, x^{0}_{31} - \frac{1}{8} \, x^{0}_{32} - \frac{1}{24} \, x^{0}_{33}\\
x^{1}_{21} &= -i x^{0}_{11} + \frac{2}{3} \, x^{0}_{10} + \frac{1}{2} \, x^{0}_{11} - \frac{2}{3} \, x^{0}_{13}.
\end{align}
Finally, we also find a relation between the entries $x^{1}_{11}$ and $x^{1}_{22}$: \begin{equation}
	\label{other equation for x11}
	x^{1}_{22}- x^{1}_{11} = -i x^{0}_{12} - \frac{1}{4} \, x^{0}_{10} - \frac{3}{2} \, x^{0}_{12} - \frac{1}{4} \, x^{0}_{13}.
\end{equation}
The entries $x^1_{00}$, $x^1_{12}$ and $x^1_{33}$ are as of yet undetermined. Now observe that if we are to solve the next recurrence for $R_{i+2}$ the conditions obtained above should hold with $i$ replaced by $i+1$. Equation \eqref{equation for x00} thus determines $x^1_{00}$ in terms of the entries of $R_{i+1}$ we already found. We obtain: \begin{equation}
	x^{1}_{00} = \frac{6 \, {\left(i - 1\right)} x^0_{01} - 12 \, {\left(i + 1\right)} x^0_{02} - 7 \, x^0_{00} + x^0_{03}}{64 \, {\left(i + 1\right)}}.
\end{equation} 
Similarly, Equation \eqref{equation for x33} determines $x^1_{33}$: \begin{equation}
	x^1_{33} = \frac{6 \, i {\left(x^{0}_{31} + 2 \, x^{0}_{32}\right)} - x^{0}_{30} - 6 \, x^{0}_{31} + 12 \, x^{0}_{32} + 7 \, x^{0}_{33}}{64 \, {\left(i + 1\right)}}.
\end{equation}
Next, observe that the coefficients in front of $x^1_{11}$ and $x^1_{22}$ in Equations \ref{equation for x11} and \ref{other equation for x11} are linearly independent. Thus these two equations can be solved to determine $x^1_{11}$ and $x^1_{22}$. This yields: \begin{equation}
		x^{1}_{11} = \frac{192 \, i^{2} x^{0}_{12} + 8 \, i {\left(4 \, x^{0}_{10} + 72 \, x^{0}_{12} + 4 \, x^{0}_{13} + x^{0}_{20} - x^{0}_{23}\right)} + 64 \, x^{0}_{10} + 384 \, x^{0}_{12} + 64 \, x^{0}_{13} - 4 \, x^{0}_{20} - 9 \, x^{0}_{21} + 4 \, x^{0}_{23}}{384 \, {\left(i + 1\right)}}
		\end{equation}
	\begin{multline}
		x^1_{22} = -i x^{0}_{12} - \frac{1}{4} \, x^{0}_{10} - \frac{3}{2} \, x^{0}_{12} - \frac{1}{4} \, x^{0}_{13} \\+ \frac{192 \, i^{2} x^{0}_{12} + 8 \, i {\left(4 \, x^{0}_{10} + 72 \, x^{0}_{12} + 4 \, x^{0}_{13} + x^{0}_{20} - x^{0}_{23}\right)} + 64 \, x^{0}_{10} + 384 \, x^{0}_{12} + 64 \, x^{0}_{13} - 4 \, x^{0}_{20} - 9 \, x^{0}_{21} + 4 \, x^{0}_{23}}{384 \, {\left(i + 1\right)}}.
	\end{multline}

We are thus left with a single entry, $x^1_{21}$, to be determined, and a single condition, Equation \eqref{remaining condition} to be satisfied. We will use this condition to determine $x^1_{21}$. For $R_{i+1}$ this condition reads: \begin{equation}
	{\left(\frac{1}{2} - i\right)} x^{1}_{21} + \frac{2}{3} \, x^{1}_{20} - \frac{2}{3} \, x^{1}_{23} = 0.
\end{equation}
We now substitute in the solutions we found. This gives a new relation for the entries $x^{0}$ of $R_i$: \begin{equation}
	-\frac{1}{4}x^{0}_{22} + \frac{1}{24}(2-i)(x^0_{20} + x^0_{23}) + (\frac{1}{2} - i)( (\frac{1}{2} - i)x^0_{11} + \frac{2}{3}x^0_{10} - \frac{2}{3}x^0_{13}) = 0.
\end{equation}
Again, assume that this condition is satisfies. We want to construct $R_{i+1}$ so that this condition holds with $i+1$ instead of $i$. The relation is: \begin{equation}
	-\frac{1}{4}x^{1}_{22} + \frac{1}{24}(1-i)(x^1_{20} + x^1_{23}) + (-\frac{1}{2} - i)( (-\frac{1}{2} - i)x^1_{11} + \frac{2}{3}x^1_{10} - \frac{2}{3}x^1_{13}) = 0.
	\end{equation}
 Again, all of these entries have already been determined in terms of the entries $x^0$ of $R_i$. So substitute these in. This gives a new relation for the variables $x^0$ which includes a term $x^0_{12}$ with a coefficient which is non-zero for all $i > 0$. We have not written out the equation, as it is rather large. Again, we want this relation to hold for the variables $x^1$ with $i+1$ instead of $i$. This equation then determines $x^{1}_{21}$ in terms of the other entries $x^{1}$ which have already been determined. Substituting in the solutions for these entries gives: \begin{multline}
 	x^1_{12} = 	({768 \, {\left(i^{3} + 6 \, i^{2} + 12 \, i + 8\right)}})^{-1} \bigg( 8 \, {\left(4 \, i^{3} + 16 \, i^{2} + 22 \, i + 9\right)} x^{0}_{10} - 18 \, {\left(4 \, i^{2} + 8 \, i + 5\right)} x^{0}_{11}\\ - 8 \, {\left(4 \, i^{3} + 16 \, i^{2} + 22 \, i + 9\right)} x^{0}_{13} - {\left(8 \, i^{2} + 6 \, i - 5\right)} x^{0}_{20} - 6 \, {\left(4 \, i + 5\right)} x^{0}_{22} - {\left(8 \, i^{2} + 6 \, i - 5\right)} x^{0}_{23}\bigg).
 \end{multline}
By construction, the matrix $R_{i+1}$ satisfies the conditions we assumed held for $R_i$. Thus, this recursion uniquely determines $R$.

\section{Orientation properties of gluing at interior points}
\label{orientation properties}
In this section we will prove: \begin{nlemma}
	\label{gluing at interior points}
	The gluing map \begin{equation}
	\widetilde{\mathcal{M}}^{main}_{\emptyset}(\beta_1) \times_{X} \widetilde{\mathcal{M}}^{main}(\beta_2) \rightarrow \widetilde{\mathcal{M}}^{main}(\beta_1 + \beta_2)
\end{equation}
is a local diffeomorphism which changes orientation by $(-1)^{w_{\mathfrak{s}}(\beta_1)}$.
\end{nlemma}
Here we will use notation and definitions from \cite{So}. Instead of using Pin$^\pm$ structures, we will use spin structures. 
\begin{proof}
It suffices to prove this locally. To this end, let $v: S^2 \rightarrow X$ and $u: (\mathbb{D},S^1) \rightarrow (X,L)$ be holomorphic maps of degree $\beta_1$ and $\beta_2$ respectively. Assume $u(0) = v(0)$. For a fixed gluing parameter, consider the glued map $u \# v: (\mathbb{D},S^1) \rightarrow (X,L)$. The gluing map gives an isomorphism of determinant lines \begin{equation}
	\label{WTS determinant lines}
	Det(D_u) \otimes Det(D_v) \rightarrow Det(D_{u\#v}).
\end{equation} 
We need to compute the change in orientation of this map. Recall first that the relative spin structure on $i: L \rightarrow X$ consists of a triangulation of $X$ and $L$ and a vector bundle $V \rightarrow X^3$ such that $i^*w_2(V) = w_2(TL)$, and a spin structure $\mathfrak{p}$ on $TL \oplus V|_L$. Here $w_2$ denotes the second Stiefel-Whitney class.

We will now explain briefly, following \cite{So}, how to orient $Det(D_u)$. First, up to homotopy, we may assume that $u: (\mathbb{D},S^1) \rightarrow (X^3,L^2)$. We then consider the Fredholm problem $D_u \oplus D^0_u$ on $u^*(TX \oplus V)$ for some choice of Fredholm operator $D^0_u$ on the bundle $u^*(V\otimes \CC) \rightarrow \mathbb{D}$. The spin structure $\mathfrak{p}$ on $TL \oplus V|_L$ defines a canonical orientation of $Det(D_u \oplus D^0_u)$ (see \cite[Proposition~2.8]{So}). Furthermore, as $u^*(V\otimes \CC) \rightarrow \mathbb{D}$, and $\mathbb{D}$ is contractible, there exists a spin structure $\widetilde{\mathfrak{p}}_0^u$ on $u^*(V\otimes \CC)$. By restricting to $(u|_L)^*V \rightarrow S^1$, this defines a spin structure $\mathfrak{p}_0^u$, thus equipping $Det(D_0^u)$ with an orientation. Solomon shows (\cite[Lemma~2.11]{So}) this orientation does not depend on the choice of spin structure $\widetilde{\mathfrak{p}}_0^u$. As \begin{equation}
	Det(D_u) \cong Det(D_u \oplus D_0^u) \otimes Det(D_0^u)^*,
\end{equation}
we obtain a canonical orientation on $Det(D_u)$. 

The procedure for $D_{u \# v}$ is similar. We pick a Fredholm operator $D_0^{u \# v}$ and a spin structure $\widetilde{\mathfrak{p}}_0^{u \# v}$ and obtain an isomorphism: \begin{equation}
	Det(D_{u \# v}) \cong Det(D_{u \# v} \oplus D_0^{u \# v}) \otimes Det(D_0^{u \# v})^*.
\end{equation}

To orient $Det(D_v)$, we note that $D_v$ is a (complex) Cauchy Riemann problem, and thus obtains a canonical orientation from the complex structure on $ker(D_v)$ and $coker(D_v)$.

We can thus rewrite the isomorphism \ref{WTS determinant lines} as \begin{equation}
	\label{WTS determinant lines 2}
	Det(D_u \oplus D_0^u) \otimes Det(D_0^u)^* \otimes Det(D_v) \cong Det(D_{u \# v} \oplus D_0^{u \# v}) \otimes Det(D_0^{u \# v})^*.
\end{equation}
Now let $D_0^v$ be an arbitrary (complex) Cauchy-Riemann operator on $v^*(V \otimes \mathbb{C}) \rightarrow S^2$. We can then glue (see \cite[Section~2.4]{WW}) the bundle $v^*(TX \oplus (V \otimes \CC)) \rightarrow S^2$ with $u^*(TX \oplus (V \otimes \CC)) \rightarrow \mathbb{D}$ and obtain a Cauchy-Riemann operator \begin{equation}(D_u \oplus D_0^u) \# (D_v \oplus D_0^v) \cong (D_u \# D_v) \oplus (D_0^u \oplus D_0^v) \cong (D_{u \# v} \oplus D_0^{u \# v}).\end{equation}
Here we have to be careful about spin structures. The gluing map equips the right-hand side with the same spin structure on the boundary as on the left-hand side. Thus, on both sides the spin structure is given by $u|_{\partial \mathbb{D}}^*\mathfrak{p}$. This spin structure was also used for the canonical orientation of $Det(D_{u \# v} \oplus D_0^{u \# v})$. The result is that the induced map on determinant line bundles: \begin{equation}
	Det(D_u \oplus D_0^u) \otimes Det(D_v \oplus D_0^v) \cong Det(D_{u \# v} \oplus D_0^{u \# v}),
\end{equation}
respects orientations (see \cite[Theorem~4.3.3]{WW}). Substituting this into Equation \eqref{WTS determinant lines 2}, and noting that $Det(D_v \oplus D_0^v) \cong Det(D_v) \otimes Det(D_0^v)$, we can rewrite the isomorphism as: \begin{equation}
	Det(D_0^u; \mathfrak{p}_0^u) \otimes Det(D_0^v) \cong Det(D_0^{u \# v}; \mathfrak{p}_0^{u \# v}).
\end{equation}
Here we have included in the notation the spin structures that we consider on the boundary. Finally, gluing of the operators $D_0^u$ and $D_0^v$ allows us to compare them with $D_0^{u \# v}$. Again, the gluing map equips the bundles both before and after with the same spin structure. The result is an orientation preserving isomorphism: \begin{equation}
	Det(D_0^u) \otimes Det(D_0^v) \cong Det(D_0^{u \# v}; \mathfrak{p}_0^u).
\end{equation}
We have thus reduced our main problem to figuring out the change in orientation between: \begin{equation}
	\label{WTS determinant lines 3}
	Det(D_0^{u \# v}; \mathfrak{p}_0^u) \; \text{and} \; Det(D_0^{u \# v}; \mathfrak{p}_0^{u \# v}).
\end{equation}
The spin structure $ \mathfrak{p}_0^{u \# v}$ extends over all of $\mathbb{D}$, but $\mathfrak{p}_0^u$ might not. As $\mathfrak{p}_0^u$ extends over $u^*V \rightarrow \mathbb{D}$, general obstruction theory tells us that the obstruction to extending $\mathfrak{p}_0^u$ over $(u \# v)^*V \rightarrow \mathbb{D}$ is given by $w_2(V)(v_*([S^2])$.

Now, for any bundle $F \rightarrow \mathbb{D}$, the spin structures on $F|_{S^1} \rightarrow S^1$ are classified by whether or not they extend over the entirety of $\mathbb{D}$. Combining this with \cite[Lemma~2.10]{So}, we find that the orientations in Equation \eqref{WTS determinant lines 3} agree if $w_2(V)(v_*([S^2]) = 0$, and are opposed otherwise.
\end{proof}

\printbibliography
\end{document}